\documentclass{article}
\usepackage{multicol,graphicx,color}
\usepackage{pslatex}
\usepackage{authblk}
\usepackage{amsthm}
\usepackage{amsmath}
\usepackage{amssymb}
\usepackage{appendix}
\usepackage{indentfirst}
\usepackage{latexsym}
\usepackage{lscape}
\usepackage{epsfig}
\usepackage{pstricks}
\usepackage{amsfonts}
\usepackage{mathrsfs}
\usepackage{mathrsfs}
\usepackage{enumitem}
\usepackage[numbers,sort&compress]{natbib}
\usepackage[
  hmarginratio={1:1},     
  vmarginratio={1:1},     
  textwidth=15cm,        
  textheight=21cm,
  heightrounded,          
]{geometry}

\usepackage{graphicx,color}
\usepackage[colorlinks]{hyperref}
\hypersetup{linkcolor=blue,citecolor=blue,filecolor=black,urlcolor=blue}

\theoremstyle{plain}
\newtheorem{theorem}{Theorem}
\newtheorem{corollary}[theorem]{Corollary}
\newtheorem{lemma}[theorem]{Lemma}
\newtheorem{proposition}[theorem]{Proposition}
\theoremstyle{definition}

\newtheorem{remark}[theorem]{Remark}

\numberwithin{equation}{section}
\numberwithin{theorem}{section}

\newcommand{\R}{\mathbb{R}}
\newcommand{\eps}{\varepsilon}
\newcommand{\dist}{{\rm dist}}
\newcommand{\supp}{{\rm supp}}

\author{Yuxin Li\footnote{liyx097@nenu.edu.cn},~~ Meijie Yang\footnote{yangmj@nenu.edu.cn} ~and~Xiaojun Chang \footnote{changxj100@nenu.edu.cn}} \affil{School of Mathematics and Statistics \& Center for Mathematics and Interdisciplinary Sciences,
 Northeast Normal University, Changchun 130024, Jilin,
PR China}

\title{Normalized solutions for a Sobolev critical quasilinear Schr\"odinger equation}
\date{}

\begin{document}

\maketitle

\begin{abstract}
	In this paper, we study the existence of normalized solutions for the following quasilinear Schr\"odinger equation with Sobolev critical exponent:
	\begin{eqnarray*}
			-\Delta u-u\Delta (u^2)+\lambda u=\tau|u|^{q-2}u+|u|^{2\cdot2^*-2}u,~~~~x\in\mathbb{R}^N,
	\end{eqnarray*}
	under the mass constraint $\int_{\R^N}|u|^2dx=c$ for some prescribed $c>0$. Here $\tau\in \mathbb{R}$ is a parameter, $\lambda\in\R$ appears as a Lagrange multiplier, $N\ge3$, $2^*:=\frac{2N}{N-2}$ and $2<q<2\cdot2^*$. By deriving precise energy level estimates and establishing new convergence theorems, we apply the perturbation method to establish several existence results for $\tau>0$ in the Sobolev critical regime:
	\begin{enumerate}
		\item[(a)] For the case of $2<q<2+\frac{4}{N}$, we obtain the existence of two solutions, one of which is a local minimizer, and the other is a mountain pass type solution, under explicit conditions on $c>0$;
		\item[(b)] For the case of $2+\frac{4}{N}\leq q<4+\frac{4}{N}$, we obtain the existence of normalized solutions of mountain pass type under different conditions on $c>0$;
		\item[(c)] For the case of $4+\frac{4}{N}\leq q<2\cdot2^*$, we obtain the existence of a ground state normalized solution under different conditions on $c>0$. 
\end{enumerate}
Moreover, when $\tau\le 0$, we derive the non-existence result for $2<q<2\cdot2^*$ and all $c>0$.
Our research provides a comprehensive analysis across the entire range $q\in(2, 2 \cdot 2^*)$ and for all $N\ge3$.
The methods we have developed are flexible and can be extended to a broader class of nonlinearities.
\end{abstract}

\medskip


\section{Introduction}\label{intro}\setcounter{equation}{0}
In this paper, we investigate the quasilinear Schr\"odinger equation with Sobolev critical exponent, which takes the following form: 
\begin{equation}\label{quasilinear eq1}
	\begin{cases}
		i\partial_t\psi+\Delta\psi+\psi\Delta(|\psi|^2)+\tau|\psi|^{q-2}\psi+|\psi|^{2\cdot2^*-2}\psi=0,~~&(t,x)\in\R^+\times\R^N,\\
		\psi(0,x)=\psi_0(x),~~&x\in\R^N,
	\end{cases}
\end{equation}
where $N\ge3$, $2<q<2\cdot2^*$, $\tau>0$ is a parameter, $i$ denotes the imaginary unit, and $\psi:\R^+\times\R^N\to\mathbb{C}$ represents the time-dependent wave function with the initial datum $\psi_0$. This class of quasilinear equations has received widespread attention in recent years due to its rich applications in mathematical physics, for example, in dissipative quantum mechanics, plasma physics and fluid mechanics. For more physical motivations and applications, we refer the interested readers to \cite{BL1986,BL19861,kuri1981,laed1983,LPT1999,PSW2002} and the references therein.

We are interested in the standing wave solutions to \eqref{quasilinear eq1}, especially those of the form $\psi(t,x)=e^{-i\lambda t}u(x)$, where $\lambda\in\R$. In this case, the function $u(x)$ satisfies the following stationary equation
\begin{equation}\label{quasilinear eq2}
	-\Delta u-u\Delta (u^2)+\lambda u=\tau|u|^{q-2}u+|u|^{2\cdot2^*-2}u,~~~~x\in\R^N.
\end{equation}
When looking for solutions to \eqref{quasilinear eq2}, a typical strategy is to treat $\lambda\in\R$ as a fixed parameter. The solutions of \eqref{quasilinear eq2} correspond to the critical points of the energy functional given by
\begin{equation*}
	J_{\lambda}(u):=\frac{1}{2}\int_{\R^N}|\nabla u|^2dx+\frac{\lambda}{2}\int_{\R^N}|u|^2dx+\int_{\R^N}|u|^2|\nabla u|^2dx-\frac{\tau}{q}\int_{\R^N}|u|^qdx-\frac{1}{2\cdot2^*}\int_{\R^N}|u|^{2\cdot2^*}dx,
\end{equation*}
which is defined on the natural space
\begin{equation*}
	\widetilde{X}:=\{u\in H^1(\R^N): \int_{\R^N}|u|^2|\nabla u|^2dx<+\infty\}.
\end{equation*}
Owing to the presence of the quasilinear term $\int_{\R^N}|u|^2|\nabla u|^2dx$, the functional $J_{\lambda}$ is non-differentiable in $\widetilde{X}$ for $N\ge2$. To address this difficulty, several variational techniques have been developed, including minimization methods \cite{LW2003,PSW2002,RS2010}, the Nehari manifold approach \cite{AW2003,LLW20041}, the dual approach \cite{cj2004,LLW20031}, and perturbation approaches \cite{LLW2014,LLW2013}.

It is noteworthy that the number $2\cdot2^*$ acts as a critical exponent for \eqref{quasilinear eq2}, as first observed in \cite{LLW20041}. In that work,
Liu, Wang and Wang considered the following quasilinear Schr\"odinger equation
\begin{equation}\label{quasi llw eq}
-\Delta u+V(x)u-u\Delta (u^2)=f(u),~~~~x\in\R^N,
\end{equation}
where $f(u)=|u|^{q-2}u$ and $V(x)$ is a given potential. 
By using a Pohozaev type identity, they proved that \eqref{quasi llw eq} admits no positive solutions in $\widetilde{X}$ if $q\ge2\cdot2^*$ and $x\cdot\nabla V(x)\ge0$ for all $x\in\R^N$.
In \cite{BMS2010}, Bezerra, Miyagaki and Soares demonstrated the existence of positive solution for \eqref{quasi llw eq} with combined nonlinearities $f(u)=|u|^{q-2}u+|u|^{2\cdot2^*-2}u$ for $4<q<2\cdot2^*$.
They utilized a change of variables, similar to those in \cite{cj2004,LLW20031}, to transform \eqref{quasi llw eq} into a semilinear elliptic equation whose associated functionals are well-defined in the Sobolev space $H^1(\R^N)$.
Subsequently, Liu, Liu and Wang \cite{LLW20131} established the existence of both positive and nodal ground states of \eqref{quasi llw eq} by analyzing the behavior of solutions in subcritical cases 
and passing to the limit as the exponent tends to $2\cdot2^*$.
In \cite{LLW201311}, they adopted a perturbation method to investigate a more general quasilinear elliptic equations and achieved new existence results.
For further details on critical quasilinear problems,  the reader is referred to \cite{DGY2019,SV2010,YD2013} and their references.

Alternatively, it is intriguing to investigate solutions to \eqref{quasilinear eq2} with a prescribed $L^2$-norm, motivated by the physical principle of mass conservation. Specifically, we seek pairs $(u,\lambda)\in \widetilde{X}\times\R$ satisfying
\begin{equation}\label{quasi eq}
	\begin{cases}
		-\Delta u-u\Delta (u^2)+\lambda u=\tau|u|^{q-2}u+|u|^{2\cdot2^*-2}u,~~~~x\in\R^N,\\
		\int_{\R^N}|u|^2dx=c,
	\end{cases}
\end{equation}
where $c>0$ is a given constant, and the parameter $\lambda\in\mathbb{R}$ arises as a Lagrange multiplier. Solutions of this form are commonly termed \emph{normalized solutions}, and the present paper is dedicated to exploring  this topic.
Formally, a natural approach to find solutions of \eqref{quasi eq} is to search for critical points of the associated functional
\begin{equation}\label{energy functional}
	I(u):=\frac{1}{2}\int_{\R^N}|\nabla u|^2dx+\int_{\R^N}|u|^2|\nabla u|^2dx-\frac{\tau}{q}\int_{\R^N}|u|^qdx-\frac{1}{2\cdot2^*}\int_{\R^N}|u|^{2\cdot2^*}dx
\end{equation}
on the constraint set
\begin{equation*}
	\widetilde{S}(c):=\{u\in \widetilde{X}: \int_{\R^N}|u|^2dx=c\}.
\end{equation*}
The $L^2$-constraint gives rise to a critical exponent $\bar{q}:=4+\frac{4}{N}$,  which is pivotal for analyzing dynamical properties like global existence, blow-up, and stability of ground states. Furthermore, this threshold significantly influences the geometric structure of the functional $I$. 
Specifically, for the pure nonlinearity $|u|^{q-2}u$, the case $q<\bar{q}$ is known as $L^2$-subcritical. In this scenario, the functional $I$ is bounded from below on $\widetilde{S}(c)$, and global minimizers always exist. Conversely, when $q>\bar{q}$, the case is designated as $L^2$-supercritical. Here, the functional $I$ is unbounded from below on $\widetilde{S}(c)$, necessitating additional technical arguments, as one can not expect to directly locate a minimum of $I$ on $\widetilde{S}(c)$.

The $L^2$-subcritical case, i.e., $2<p<\bar{q}$, has been widely studied. In \cite{CJS2010}, Colin, Jeanjean and Squassina considered the following equation
\begin{equation}\label{quasi eq norm ex}
	\begin{cases}
		-\Delta u-u\Delta(u^2)+\lambda u=|u|^{p-2}u,~~~~x\in\R^N,\\
		\int_{\R^N}|u|^2dx=c.
	\end{cases}
\end{equation}
They employed a minimization approach:
\begin{equation}\label{mini prob cjs}
	\widetilde{m}(c):=\inf_{u\in\widetilde{S}(c)}E(u),
\end{equation}
where
\begin{equation*}
	E(u):=\frac{1}{2}\int_{\R^N}|\nabla u|^2dx+\int_{\R^N}|u|^2|\nabla u|^2dx-\frac{1}{p}\int_{\R^N}|u|^{p}dx.
\end{equation*}
They proved that the minimization problem \eqref{mini prob cjs} has a minimizer for $p\in(2,2+\frac{4}{N})$ across all $c>0$, and for $p\in[2+\frac{4}{N},4+\frac{4}{N})$ when $c>0$ is sufficiently large. Notably, the non-differentiability of $E$ did not impact their proof. Jeanjean and Luo \cite{JLuo2013} later extended these results and provided a sharp conclusion.
Using a perturbation method, Jeanjean, Luo, and Wang \cite{JLW2015} investigated the case $p\in(2+\frac{4}{N},4+\frac{4}{N})$ and established the existence of a mountain pass type normalized solution, as well as another solution that is either a local minimizer or a global minimizer, depending on different conditions on the mass $c>0$.
In \cite{ZLW2023}, Zhang, Li and Wang presented a change of variables technique to construct multiple normalized solutions of \eqref{quasi eq norm ex} with energies tending to $0$ for $N\ge2$ and $p\in(2,4+\frac{4}{N})$. By demonstrating the equivalence between the quasilinear and semilinear equations under the  $L^2$-constraint, they established the existence of infinitely many pairs of normalized solutions in this dual framework.

In the $L^2$-critical/supercritical case, i.e., $p\ge\bar{q}$, Ye and Yu \cite{YY2021} studied the $L^2$-critical problem and identified a threshold value of $c>0$ that distinguishes between the existence and nonexistence of critical points of the functional $E$ restricted on $\widetilde{S}(c)$.  Successively, Li and Zou \cite{LZ2023} employed a perturbation method to establish the existence of ground state normalized solutions and further obtained multiplicity results through index theory.
Zhang, Chen, and Wang \cite{ZCW2023} worked directly with the functional $E$ in spite of its non-smoothness in low-dimensional space. They proved the existence of ground states to \eqref{quasi eq norm ex} for $p\in[4+\frac{4}{N},\frac{2N}{N-2}]$ when $N=3$ and $p\in[4+\frac{4}{N},\infty)$ when $N=1,2$.
More recently, Mao and Lu  \cite{ML2024} studied the existence of normalized solutions for quasilinear Schr\"odinger equations with combined nonlinearities of the form $\tau|u|^{q-2}u+|u|^{p-2}u$, where $2<q<2+\frac{4}{N}<p\leq2^*$. In a more general setting,  Deng, Squassina, Zhang and Zhong \cite{DSZZ2024} utilized a dual approach, fixed point index theory, and global bifurcation techniques to explore the existence of normalized solutions for quasilinear Schr\"odinger equations.

We observe that the aforementioned studies were all conducted within the $H^1$-subcritical and critical framework, i.e., $p\leq2^*$. In a recent result,
Jeanjean, Zhang and Zhong \cite{JZZ2025} established the existence of ground state normalized solutions for (\ref{quasi eq norm ex}) within the range $2^*<p<2\cdot2^*$ and for all $N\ge 1$.  Specifically, they demonstrated that the energy ground state is attained for all mass values $c>0$ when $1\le N\le 4$. For $N\ge5$, they identified an explicit threshold $c_0>0$ such that the energy ground state exists if and only if  $c\in(0, c_0)$.  Additionally, they analyzed the asymptotic behavior of these ground states as the mass $c>0$ varies. 
For normalized solution of quasilinear Schr\"odinger equations involving potentials with power nonlinearity in the range $4+\frac{4}{N}<p<2\cdot2^*$, we refer the reader to \cite{gg1,gg2}.

In this paper, we focus on the study of \eqref{quasi eq}, which involves the critical exponent $2\cdot2^*$. This equation serves as an analogue to the Brezis-Nirenberg problem within the framework of normalized solutions for quasilinear Schr\"odinger equations.
If we ignore the fourth term $-u\Delta (u^2)$, the equation \eqref{quasi eq} reduces to the classical nonlinear Schr\"odinger equation, and the associated
 Sobolev critical problem becomes
\begin{equation}\label{classic eq mix}
	\begin{cases}
		-\Delta u+\lambda u=\tau|u|^{q-2}u+|u|^{2^*-2}u,~~~~x\in\R^N,\\
		\int_{\R^N}|u|^2dx=c,
	\end{cases}
\end{equation}
where $2<q<2^*$.
In \cite{S2},  Soave employed Nehari-Pohozaev manifold decomposition to establish the existence and asymptotic behavior of ground state solutions to \eqref{classic eq mix} in the cases of $L^2$-subcritical, $L^2$-critical, and $L^2$-supercritical perturbations, respectively. 
Subsequently, for the $L^2$-subcritical scenario where $q\in (2, 2+\frac{4}{N})$, Jeanjean and Le \cite{JL2022} demonstrated the existence of a second solution at the mountain pass level for $N\ge4$. Concurrently, Wei and Wu \cite{WW2022} established the existence of mountain pass type solutions for $N\ge3$.
For more results on this topic,  the reader may refer to \cite{CT2024,CT2024-2,JJLV2022,L2021}.

The aim of this paper is to 
conduct a thorough analysis of the existence of normalized solutions for \eqref{quasi eq} across the entire range $q\in(2, 2 \cdot 2^*)$ and for all $N\ge3$.
To achieve this, we will investigate in detail the interaction between the lower-order term $\tau|u|^{q-2}u$ and the Sobolev critical term $|u|^{2\cdot2^*-2}u$, and analyze how this interaction influences the structure of the constrained functional.

Compared to the classical Schr\"odinger equation \eqref{classic eq mix}, the study of \eqref{quasi eq} is significantly more intricate. We must not only address the non-smoothness caused by the quasilinear term  $-u\Delta(u^2)$, but also resolve the loss of compactness resulting from the Sobolev critical term $|u|^{2\cdot2^*-2}u$.  The main challenges are outlined as follows:
\begin{itemize}
	\item In the process of compactness analysis, we encounter two main obstacles arising from the quasilinear term $-u\Delta (u^2)$.
	On the one hand, it is widely recognized that Brezis-Lieb type lemma plays a crucial role in studying $L^2$-constraint problems. However, for our problem, such a splitting lemma is not readily available.
	The primary difficulty here lies in the emergence of undesirable cross-term interactions:
	\begin{align*}
		\int_{\R^N}|v_1+v_2|^2|\nabla (v_1+v_2)|^2dx&=\int_{\R^N}|v_1|^2|\nabla v_1|^2dx+2\int_{\R^N}|v_1|^2\nabla v_1\nabla v_2dx+\int_{\R^N}|v_1|^2|\nabla v_2|^2dx\nonumber\\
		&\quad+2\int_{\R^N}v_1v_2|\nabla v_1|^2dx+4\int_{\R^N}v_1v_2\nabla v_1\nabla v_2dx+2\int_{\R^N}v_1v_2|\nabla v_2|^2dx\nonumber\\
		&\quad+\int_{\R^N}|v_2|^2|\nabla v_1|^2dx+2\int_{\R^N}\nabla v_1\nabla v_2|v_2|^2dx+\int_{\R^N}|v_2|^2|\nabla v_2|^2dx.
	\end{align*}
	On the other hand, for the bounded Palais-Smale sequence $\{u_n\}$, the weak convergence $u_n\rightharpoonup u$ in $H^1(\R^N)$  does not imply the desired convergence
	\begin{equation*}
		\int_{\R^N}\left(u_n\phi|\nabla u_n|^2+|u_n|^2\nabla u_n\nabla \phi\right)dx\to\int_{\R^N}\left(u\phi|\nabla u|^2+|u|^2\nabla u\nabla \phi\right)dx~~~\text{as}~~n\to+\infty,
	\end{equation*}
	where $\phi\in\mathcal{C}_0^{\infty}(\R^N)$. In our arguments, the critical term $|u|^{2\cdot2^*-2}u$ introduces additional complexity, preventing us from bypassing the discussion on such splitting property as was done in \cite{JLW2015,LZ2023} for addressing the subcritcal problems.

	\item 
	When investigating the problem \eqref{classic eq mix}, the parameter range of $q$ in the $L^2$-subcritcal perturbation term $\tau|u|^{q-2}u$ typically yields two distinct types of solutions: local minimizers and mountain pass solutions. Since problem \eqref{quasi eq} exhibits a similar variational structure, it is reasonable to seek these solutions when $q$ lies within the $L^2$-subcritcal range, i.e., $q<4+\frac{4}{N}$.
	However, for quasilinear problems with Sobolev critical growth, we face the following challenges:
	\begin{enumerate}[label=(\roman*)]
		\item How can we establish the strict subadditivity of the local minimization energy, and then derive the compactness of minimizing sequences in $\widetilde X$ from the local minimization geometry?
		\item How can we establish a precise upper bound estimate for the corresponding mountain pass energy levels, a commonly used strategy for recovering compactness in such problems? This step is crucial in the Sobolev critical setting, ensuring that the minimax levels remain below the threshold needed for compactness recovery.
	For our problem \eqref{quasi eq}, by the Sobolev inequality as follows,
	\begin{equation*}
		\mathcal{S}\left(\int_{\R^N}|u|^{2\cdot2^*}dx\right)^{\frac{2}{2^*}}\leq4\int_{\R^N}|u|^{2}|\nabla u|^2dx,
	\end{equation*}
	we observe that the terms $-u\Delta (u^2)$ and $|u|^{2\cdot2^*-2}u$ in \eqref{quasi eq} are comparable. This necessitates careful handling of the additional perturbation $-\Delta u$ when performing subtle estimates and analyses. Therefore, it is essential to separately determine the compactness threshold for our problem when
	$2<q<2+\frac{4}{N}$ and when $q\ge2+\frac{4}{N}$, and then conduct a meticulous analysis to obtain the relevant energy estimates.
	\end{enumerate}
	\item 
 Beyond the aforementioned obstacles, the non-vanishing of the Lagrange multiplier is also instrumental in controlling potential losses of compactness.
	The conventional approach to proving $\lambda\neq0$ involves combining the Nehari-Pohozaev type identity, as seen in \cite{LZ2023}.  Within this framework, the result that $\lambda>0$ could only be established for \eqref{quasi eq norm ex} when $p\leq2^*$, which leads to existence results that are valid only in low-dimensional spaces, specifically for $N\leq3$ (see also \cite{ML2024,ZCW2023}). However, since we are actually dealing with a $H^1$-supercritical problem, the standard method for determining the sign of $\lambda\in\R$ is no longer applicable.
	Therefore, alternative techniques are needed to ascertain the sign of $\lambda\in\R$, which will remove additional conditions on $N$, $p$, and allow for extension to a broader class of problems.
\end{itemize}





\subsection{Main results}

It is well known that for $N\ge2$, the functional $I$ defined by \eqref{energy functional} typically lacks smoothness in $\widetilde{X}$,
 which precludes the direct application of classical critical point theory. To handle the quasilinear term $-u\Delta(u^2)$, it is necessary to work in a smaller space to  achieve smoothness.
To overcome this challenge, we adopt a perturbation method as used in \cite{JLW2015} and \cite{LZ2023}. To be more precise, we consider the following perturbed functional:
\begin{equation}\label{pertu fun}
	I_{\mu}(u):=\frac{\mu}{\theta}\int_{\R^N}|\nabla u|^{\theta}dx+I(u),~~~\mu\in(0,1]
\end{equation}
on the reflexive Banach space $X:=W^{1,\theta}(\R^N)\cap H^1(\R^N)$,
where
\begin{equation*}
	\frac{4N}{N+2}<\theta<\min\left\{\frac{4N+4}{N+2},N\right\}.
\end{equation*}
Then $I_{\mu}$ is a well-defined and of class $\mathcal{C}^1$ on $X$, as shown in \cite[Lemma A.1]{LZ2023}. At this point, we study the perturbed problem by searching for critical points of $I_{\mu}$ restricted on
\begin{equation*}
	S(c):=\{u\in X: \int_{\R^N}|u|^2dx=c\}.
\end{equation*}
We introduce the related Pohozaev type mainfold defined by
\begin{equation}\label{Pohozaev type mainfold}
	\Lambda_{\mu}(c):=\{u\in S(c): Q_{\mu}(u)=0\},
\end{equation}
where
\begin{align}\label{Pohozaev type mainfold id}
	Q_{\mu}(u)&:=\mu(1+\gamma_{\theta})\int_{\R^N}|\nabla u|^{\theta}dx+\int_{\R^N}|\nabla u|^2dx+(N+2)\int_{\R^N}|u|^2|\nabla u|^{2}dx\nonumber\\
	&\quad-\tau\gamma_q\int_{\R^N}|u|^qdx-\gamma_{2\cdot2^*} \int_{\R^N}|u|^{2\cdot2^*}dx,
\end{align}
with $\gamma_q:=\frac{N(q-2)}{2q}$.
Our main strategy is to first identify the critical point $u_{\mu}$ of the perturbed functional $I_\mu$ for any fixed $\mu\in (0,1]$. Subsequently, we take the limit as $\mu\to0^+$, which yields the normalized solutions to the original problem \eqref{quasi eq}.

To state our main results, we introduce the following constants and notations.  Throughout this paper, 
we define
\begin{equation*}
	\mathcal{S}:=\inf_{\substack{u\in\mathcal{D}^{1,2}(\R^N) \\ u\neq0}}\frac{\|\nabla u\|_2^2}{\|u\|^2_{2^*}},
\end{equation*}
where
\begin{equation*}
	\mathcal{D}^{1,2}(\R^N):=\{u\in L^{2^*}(\R^N): |\nabla u|\in L^2(\R^N)\}
\end{equation*}
is the completion of $C_0^{\infty}(\R^N)$ with the norm $\|u\|_{\mathcal{D}^{1,2}(\R^N)}=\|\nabla u\|_2$.

Set
\begin{equation}\label{func pro c value}
	c_0:=\left(\frac{1}{2K}\right)^{\frac{N}{2}}>0,
\end{equation}
where
\begin{align*}\label{func pro K value}
	K:=\frac{\tau \mathcal{C}_2(q,N)}{q}\left[-\frac{\alpha_0}{\alpha_2}\frac{\tau \mathcal{C}_2(q,N)2^*\mathcal{S}^{\frac{2^*}{2}}}{2^{2^*-1}q}\right]^{\frac{\alpha_0}{\alpha_2-\alpha_0}}+\frac{1}{2\cdot2^*}\left(\frac{4}{\mathcal{S}}\right)^{\frac{2^*}{2}}\left[-\frac{\alpha_0}{\alpha_2}\frac{\tau \mathcal{C}_2(q,N)2^*\mathcal{S}^{\frac{2^*}{2}}}{2^{2^*-1}q}\right]^{\frac{\alpha_2}{\alpha_2-\alpha_0}}>0,
\end{align*}
where $\alpha_0$ and $\alpha_2$ are given by \eqref{alpha012}.
\begin{equation}\label{mass normal critical sup}
	\bar{c}_1:=\left(\frac{N+2}{2\tau N\mathcal{C}_1^{2+\frac{4}{N}}\left(2+\frac{4}{N},N\right)}\right)^{\frac{N}{2}}.
\end{equation}
\begin{equation}\label{mass sub restriction1}	\bar{c}_2:=\left(\frac{1}{\tau\mathcal{C}_1^q(q,N)}\frac{(N^2+4)q}{N[4N-q(N-2)]}\right)^{\frac{4N+4-Nq}{N(q-2)}}\left(\frac{D_1(N)}{\tau\mathcal{C}_2(q,N)}\frac{2q(N+2)}{4N-q(N-2)}\right)^{\frac{(N+2)[N(q-2)-4]}{2N(q-2)}},
\end{equation}
where
\begin{equation*}
	D_1(N):=\min\left\{\frac{N^2+4}{2N(N+2)},\frac{2}{N}\right\}.
\end{equation*}
\begin{equation}\label{mass quasi critical sup}
	\bar{c}_3:=\left(\frac{4N+4}{\tau N(N+2)\mathcal{C}_2\left(4+\frac{4}{N},N\right)}\right)^{\frac{N}{2}}.
\end{equation}
Here $\mathcal{C}_1(q,N)$ and $\mathcal{C}_2(q,N)$ denote the best constants for the Gagliardo-Nirenberg type inequality in $\R^N$, as given in \eqref{gn} and \eqref{gn2}, respectively.

We define the set
\begin{equation}\label{local ground state ball}
	\mathcal{V}_0(c):=\{v\in\widetilde{S}(c):\|v\nabla v\|_2^2<\tilde{\rho}_0\},
\end{equation}
where
\begin{equation*}
	\tilde{\rho}_0:=\left(\frac{\tau N\mathcal{C}_2(q,N)}{2}\left(\frac{1}{q}-\frac{(N-2)\gamma_q}{N(N+2)}\right)\right)^{\frac{2(N+2)}{4N+4-Nq}}c_0^{\frac{4N-q(N-2)}{4N+4-Nq}}.
\end{equation*}

Now, we present our main results on the existence and multiplicity of normalized solutions to equation \eqref{quasi eq}. We first consider the case where
$q\in(2,2+\frac{4}{N})$.
\begin{theorem}\label{thm: main result1-1}
	Let $N\ge3$, $q\in(2,2+\frac{4}{N})$, $\tau>0$ and $c\in(0,c_0)$. Then \eqref{quasi eq} admits a solution pair $(\tilde v_0,\tilde\lambda_0)\in(\widetilde{S}(c)\cap L^{\infty}(\R^N))\times \R^+$ satisfying
	\begin{equation*}
		\tilde v_0\in\mathcal{V}_0(c),~~~~\tilde v_0>0,~~~~I(\tilde v_0)=\inf_{v\in\mathcal{V}_0(c)}I(v)<0.
	\end{equation*}
\end{theorem}
\begin{theorem}\label{thm: main result1}
	Let $N\ge3$, $q\in(2,2+\frac{4}{N})$ and $c\in(0,c_0)$. Then there exists a sufficiently large $\tau^*=\tau^*(c)>0$ such that for any $\tau>\tau^*$, \eqref{quasi eq} admits a solution pair $(\bar u,\bar\lambda)\in(\widetilde{S}(c)\cap H^1_{rad}(\R^N)\cap L^{\infty}(\R^N))\times \R^+$ satisfying
	\begin{equation*}
		0<I(\bar u)<I(\tilde v_0)+\frac{\mathcal{S}^{\frac{N}{2}}}{2N}.
	\end{equation*}
\end{theorem}
\begin{remark}\label{thm1rem}
	In Theorem \ref{thm: main result1-1}, we develop new analytical techniques to address the obstacles
caused by the lack of compactness due to the critical exponent and the non-radial space setting. 
	We begin by establishing the strict subadditial property of the minimizing energy associated with the perturbed functional $I_{\mu}$ restricted on $V_\mu(c)$, as defined in \eqref{local mini set}. Subsequently, by employing the concentration-compactness principle, we obtain local minimizers $v_\mu\in V_\mu(c)$ for any given $\mu>0$, satisfying $0>I_{\mu}(v_\mu)=m_{\mu}(c):=\inf_{u\in V_\mu(c)} I_{\mu}(u)$. These minimizers do not necessarily exhibit radial symmetry. 

To prove a convergence result, we need to establish a Brezis-Lieb type lemma. However, the quasilinear term $-u\Delta (u^2)$ poses significant challenges in achieving this.  To overcome this issue, 
	in Lemma \ref{strong convergence}, for the minimizing sequence $\{u_n\}$ of $I_\mu$ on $V_\mu(c)$, we decomposing $u_n$ into two disjoint families of functions, thereby effectively neutralizing the impact of the cross terms. This leads to the following result: weak convergence $u_n\rightharpoonup v_{\mu}$ in $X$ implies that
	\begin{equation*}
		u_n\to v_{\mu}~~~\text{strongly in}~X~~~\text{and}~~~\|u_n\nabla u_n\|^2_2\rightarrow \|v_{\mu}\nabla v_{\mu}\|^2_2.
	\end{equation*}
	Next, we establish a new profile decomposition theorem and apply it to derive the convergence of the sequence $\{v_{\mu_n}\}\subset V_{\mu_n}(c)$ within the non radially symmetry space framework, as stated in Proposition \ref{mu to0 local mini}, to obtain a local minimizer for the original problem \eqref{quasi eq} in $\mathcal{V}_0(c)\subset\widetilde{X}$.

 	While in Theorem \ref{thm: main result1}, we work in the radially symmetric space $X_r$, a setting crucial for establishing the compactness of Palais-Smale sequences at the mountain pass level $M_{\mu}(c)$. However, in the Sobolev critical case, deriving sharp energy estimates is a vital step towards recovering compactness, and this will be achieved by choosing appropriate test functions. To address this, we slightly modify the Aubin-Talenti bubbles associated with the 
Sobolev inequality, centering them at the origin and cutting them off (see \eqref{test function}). This allows us to devise innovative approaches to establish the rigorous inequality concerning the energy level $M_{\mu}(c)$.
	More precisely, for small $\mu>0$, starting from $v_{\mu}\in V_{\mu}(c)$, we construct the following family of functions:
	 \begin{equation*}
		\widetilde{W}_{\eps,t}:=v_{\mu}+tU_\eps,~~~~W_{\eps,t}(x):=\eta^{\frac{N-2}{4}}\widetilde{W}_{\eps,t}(\eta x)\in S_r(c),
	 \end{equation*}
	 where $\eps>0$, $t>0$ and $\eta:=\left(\frac{\|\widetilde{W}_{\eps,t}\|_2^2}{c}\right)^{\frac{2}{N+2}}$. Due to the quasilinear term
	  \begin{equation*}\label{cross term}
		\int_{\R^N}|\nabla \widetilde{W}^2_{\eps,t}|^2dx=\int_{\R^N}|\nabla (v_{\mu}+tU_\eps)^2|^2dx,
	  \end{equation*}
	we must handle several cross terms. Using the regularity results from Remark \ref{regular rmk}, we derive the necessary estimates for these interaction terms, see \eqref{energy estimate eq10}-\eqref{energy estimate eq14}.
	Since $v_{\mu}$ is a solution to the perturbed problem, we develop refined techniques to carefully control the rate at which the parameter $\mu$ tends to $0$. This ultimately establish the following relationship between $m_\mu(c)$ and $M_\mu(c)$:
	  \begin{equation*}
		M_{\mu}(c)\leq m_{\mu}(c)+\frac{1}{2N}\mathcal{S}^{\frac{N}{2}}-\delta,~~~\text{for}~~~2<q<2+\frac{4}{N}~~\text{and}~~\mu>0~~\text{small enough}.
	\end{equation*}
	We then apply this vital inequality to show the strong convergence in $\widetilde{X}$, as detailed in Proposition \ref{converge process for mu to 0}.
\end{remark}	

In comparison with \eqref{classic eq mix}, the presence of both $-\Delta u$ and $-u\Delta(u^2)$ in \eqref{quasi eq} gives rise to two distinct exponents, namely $2+\frac{4}{N}$ and $4+\frac{4}{N}$. These exponents correspond to the $L^2$-critical indices for $-\Delta u$ and $-u\Delta(u^2)$, respectively. This dual nature prevents us from classifying the problem as either $L^2$-subcritical or $L^2$-supercritical in the conventional sense. Instead, an intermediate gap emerges, i.e., $2+\frac{4}{N}<q<4+\frac{4}{N}$, which complicates the structure of the associated functional restricted on $S(c)$. Our result in this case can be stated as follows.

\begin{theorem}\label{thm: main result2}
	Assume that one of the following conditions holds:
	\begin{enumerate}[label=(\roman*)]
		\item $N\ge3, ~q=2+\frac{4}{N},~c\in(0,\bar{c}_1)$;
		\item $N=3, ~q\in(2+\frac{4}{N},4+\frac{4}{N}), ~c\in(0,\bar{c}_2)$;
		\item $N\ge4, ~q\in(2+\frac{4}{N},2^*], ~c\in(0,\bar{c}_2)$;
		\item $N\ge4, ~2^*<q <4+\frac{4}{N}$, $c\in(0,c_1^*)$ for some sufficiently small $c_1^*>0$.
	\end{enumerate}
	Then there exists $\tau^*=\tau^*(c)>0$ sufficiently large such that for any $\tau>\tau^*$, \eqref{quasi eq} admits a solution pair $(\bar u,\bar\lambda)\in(\widetilde{S}(c)\cap H^1_{rad}(\R^N)\cap L^{\infty}(\R^N))\times \R^+$ satisfying
	\begin{equation*}
		0<I(\bar u)<\frac{\mathcal{S}^{\frac{N}{2}}}{2N}.
	\end{equation*}
\end{theorem}
\begin{remark}
In Theorem \ref{thm: main result2}, it is essential to choose $\tau>0$ large enough to control the sign of the Lagrange multiplier. 
However, this leads to an additional difficulty:
	\begin{itemize}
		\item In the process of analyzing the convergence $u_{\mu_n}\to\bar u$ as $\mu_n\to0^+$, it is necessary to prove $I(\bar u)\ge0$ in order to derive a contradiction to
		\begin{equation*}
			\lim_{n\to+\infty}\check{M}_{\mu_n}(c)<\frac{1}{2N}\mathcal{S}^{\frac{N}{2}},
		\end{equation*}
		where $\check{M}_{\mu_n}(c)$ is given in Lemma \ref{mp struc2}.
	\end{itemize}
	To handle this issue, we first select an appropriate Gagliardo-Nirenberg inequality by comparing the values of $2^*$ and $4+\frac{4}{N}$ ( see \eqref{gn} and \eqref{gn2}, respectively).  We then introduce an innovative classification into two cases: 
	\begin{enumerate}[label=(\roman*)]
		\item $\|\nabla \bar u\|_2^2+\|\bar u\nabla \bar u\|_2^2\leq\rho_*$;
		\item $\|\nabla \bar u\|_2^2+\|\bar u\nabla \bar u\|_2^2\geq\rho_*$,
	\end{enumerate}
where $\rho_*$ is defined by \eqref{compact of sol eq17}.  By employing a contradiction argument, we exclude the first case in our convergence result Proposition \ref{converge process for mu to 0}, and combine this with the fact that $\bar\lambda\neq0$ to establish the strong convergence for the approximating solutions $\{u_{\mu_n}\}$.
	Moreover, this analysis clarifies the necessity of the condition on $c\in(0,\bar{c}_2)$ for $q\leq2^*$, see \eqref{compact of sol eq20} and \eqref{compact of sol eq19}.
	 The argument is delicate, primarily due to the interplay between $I(\bar u)=\check{M}_0(c)>0$ and the $L^2$-subcritical perturbation $q<4+\frac{4}{N}$ of the problem.

	However, when $N\ge4, ~2^*<q <4+\frac{4}{N}$, no such explicit expression on the mass $c>0$ is available.
	This is mainly because, in the $H^1$-supercritical framework, additional techniques such as interpolation inequalities and the Sobolev inequality must be employed, making the process more intricate.
\end{remark}	

Next, we turn our attention to the $L^2$-supercritical case $q\in[4+\frac{4}{N},2\cdot2^*)$. Our results are as follows.

\begin{theorem}\label{thm: main result3-1}
	Assume that one of the following conditions holds:
	\begin{enumerate}[label=(\roman*)]
		\item $N=3~\text{or}~4, ~q=4+\frac{4}{N}, ~c\in(0,\bar{c}_3)$;
		\item $N=3~\text{or}~4, ~q\in\left(4+\frac{4}{N},\frac{2(N+2)}{N-2}\right], ~c>0$.
	\end{enumerate}
	Then there exists $\tau^*=\tau^*(c)>0$ sufficiently large such that for any $\tau>\tau^*$, \eqref{quasi eq} admits a solution pair $(\bar u_0,\bar\lambda_0)\in(\widetilde{S}(c)\cap H^1_{rad}(\R^N)\cap L^{\infty}(\R^N))\times \R^+$ satisfying
	\begin{equation*}
		0<I(\bar u_0)<\frac{\mathcal{S}^{\frac{N}{2}}}{2N}.
	\end{equation*}
	Moreover, $\bar u_0$ is a ground state normalized solution in the sense that
	\begin{equation*}
		I(\bar u_0)=\inf\{I(u): dI|_{\widetilde{S}(c)}(u)=0, u\in \widetilde{S}(c)\}
	\end{equation*}
\end{theorem}

\begin{theorem}\label{thm: main result3}
		Assume that one of the following conditions holds:
		\begin{enumerate}[label=(\roman*)]
			\item $N=3~\text{or}~4, ~q\in\left(\frac{2(N+2)}{N-2},2\cdot2^*\right), ~\tau>0$;
			\item $N\ge5, ~q\in\left[4+\frac{4}{N},2\cdot2^*\right), ~\tau>0$.
		\end{enumerate}
		Then there exists $c_2^*>0$ sufficiently small such that for any $c\in(0,c_2^*)$, \eqref{quasi eq} admits  solution pair $(\bar u_0,\bar\lambda_0)\in(\widetilde{S}(c)\cap H^1_{rad}(\R^N)\cap L^{\infty}(\R^N))\times \R^+$ satisfying
		\begin{equation*}
			0<I(\bar u_0)<\frac{\mathcal{S}^{\frac{N}{2}}}{2N}.
		\end{equation*}
	Moreover, $\bar u_0$ is a ground state normalized solution in the sense that
	\begin{equation*}
		I(\bar u_0)=\inf\{I(u): dI|_{\widetilde{S}(c)}(u)=0, u\in \widetilde{S}(c)\}
	\end{equation*}
\end{theorem}
\begin{remark}
In Theorem \ref{thm: main result3-1}, when $q=4+\frac{4}{N}$, the limitation on $c$ arises solely from the construction of the mountain pass geometry. When $q>4+\frac{4}{N}$, although no restriction is imposed on the mass $c$, it is necessary to assume that $\tau>0$ is large enough for the case $N=3,4, ~q\leq q_N^*:=\frac{2(N+2)}{N-2}$. This condition stems from the lack of a precise energy estimate and the difficulty in determining the sign of the Lagrange multiplier in such cases.
For quasilinear problems,  the energy threshold for recovering compactness differs from that of the classical problems. The energy estimate obtained in Lemma \ref{energy estimate2}  is influenced by this special value $q_N^*$. The primary reason lies in the fact that, when selecting test functions for estimation, comparisons are made based on the term
 $-\Delta u$, while the quasilinear term $-u\Delta(u^2)$ plays a less significant role, as shown in \eqref{intro esti eq}.

In Theorem \ref{thm: main result3}, we can overcome these obstacles to obtain existence results for all $\tau>0$.
 Indeed, for both ranges of $q$, we provide the precise energy estimates. Moreover, to prove that $\lambda\neq 0$, we first investigate the perturbed functional $I_{\mu}$, establish some properties of $\Lambda_{\mu}(c)$, and then define $\widehat{m}_\mu(c):=\inf_{u\in\Lambda_\mu(c)}I_\mu(u)$, which offers an alternative minimax characterization of the mountain pass level $\widehat{M}_{\mu}(c)$ defined in Lemma \ref{mp struc3 eq3}.
By constructing a family of functions with disjoint compact supports, we show that $\widehat{m}_\mu(c)$ is a non-increasing function with respect to $c>0$, see Lemma \ref{poho mani pro6}. 
Using the monotonicity of $\widehat{m}_\mu(c)$ and the implicit function theorem,  we conclude that $\bar\lambda_0>0$ in a small neighborhood of the mass $c>0$ for any $\tau>0$, leading to the existence of a mountain pass ground state solution for the original problem.
\end{remark}

Finally, we establish a nonexistence result.

\begin{theorem}\label{thm: main result4}
	Let $2<q<2\cdot2^*, \tau\leq0$ and $c\in(0,+\infty)$. Then \eqref{quasi eq} has no solutions in $\widetilde{S}(c)\times \R^+$.
\end{theorem}

\subsection{Highlights of this paper}

In contrast to the pure nonlinearity case (see \cite{CJS2010,JLW2015,LZ2023}), the functional $I$ consists of four distinct terms that exhibit different scaling behaviors with respect to the dilation $t^{\frac{N}{2}}u(t\cdot)$, which enriches significantly the geometric structures of the constrained functional.
In this subsection, we provide a thorough exposition of the challenges we have overcome and the methods we have developed in addressing our problem. We consider these contributions to be central to our work and among its most innovative aspects.

	\textbf{First of all}, since we are dealing with problems involving the Sobolev critical exponent, establishing a rigorous upper bound for the energy level is a crucial step in proving our main results.
Through precise calculations and innovative analytical techniques, we derive two key energy estimates: the first is presented in Remark \ref{thm1rem}, and the second is as follows:
\begin{equation}\label{classic eq esti}
	\begin{cases}
		\check{M}_{\mu}(c)\leq\frac{1}{2N}\mathcal{S}^{\frac{N}{2}}-\delta,~~~&\text{for}~~~\frac{2(N+2)}{N-2}<q<2\cdot2^*,\\
		\widehat{M}_{\mu}(c)\leq\frac{1}{2N}\mathcal{S}^{\frac{N}{2}}-\delta,~~~&\text{for}~~~q_N<q<\frac{4N}{N-2},
	\end{cases}
\end{equation}
where $\mu>0$ and $\delta>0$ are sufficiently small. Here $\check{M}_{\mu}(c), \widehat{M}_{\mu}(c)$ are the mountain pass levels given in Lemma \ref{mp struc2} and Lemma \ref{mp struc3 eq3}, respectively, and $q_N$ is defined by
\begin{equation}\label{qn definition}
	q_N:=
	\begin{cases}
		4+\frac{4}{N},~~~\text{for}~~~N\ge5,\\
		\frac{2(N+2)}{N-2},~~~\text{for}~~~N=3,4.
	\end{cases}
\end{equation}
To achieve these estimates, a pivotal issue lies in how to choose appropriate test functions, and we are confronted with two main obstacles:
\begin{itemize}
	\item The test function defined by \eqref{test function} is not suitable for use in the calculation of the mountain pass level. The reason is that the truncation function in \eqref{test function} lacks a clear and precise definition in the annular region $B_2\backslash B_1$.  As a result, the function $U_\eps$ does not belong to $S(c)$, as shown in \eqref{test fun esti5}.
	\item The utilization of the perturbation technique prevents us from first projecting $U_\eps$ onto $S(c)$ and then onto $\Lambda_\mu(c)$ to accurately compute the mountain pass level, as was typically done in \cite{S2} for the $L^2$-supercritical case. Moreover, for the $L^2$-subcritical case, such a projection onto the Nehari-Pohozaev manifold is not unique, which introduces additional complications.
\end{itemize}
To address these challenges, we first construct a family of test functions as shown in equation \eqref{test functions2} for the original functional $I_0$. 
Through providing an exact definition of $\widehat{U}_\eps$ on the annulus $B_{\eps^{-\beta}}\backslash B_{\eps^{-\alpha}}$ and carefully selecting appropriate parameters $\alpha,\beta>0$, we are able to show that $\|\widehat{U}_\eps\|_2^2\to c$ as $\eps\to0$, which ensures that $\widehat{U}_\eps$ lies near the constraint $S(c)$, see \eqref{energy2 est2}; Next, we define $\widehat{V}_\eps:=\frac{\sqrt{c}}{\|\widehat{U}_{\eps}\|_2}\widehat{U}_{\eps}$ to project $\widehat{U}_\eps(x)$ onto $S_r(c)$. Then for all $q\in(2,2\cdot2^*)$, the following inequality holds.
\begin{align}\label{intro esti eq}	I_{0}(t^{\frac{N}{2}}\widehat{V}_{\eps}(tx))&\leq\frac{1}{2N}\mathcal{S}^{\frac{N}{2}}+\frac{t^2}{2}\mathcal{O}\left(\eps^{\frac{(4\alpha+1)(N-2)}{2N}}\right)+t^{N+2}\mathcal{O}\left(\eps^{\frac{(4\alpha+1)(N^2-4)}{4N}}\right)\nonumber\\
	&\quad-\frac{t^{\frac{N(N+2)}{N-2}}}{2\cdot2^*}\mathcal{O}\bigg(\eps^{\frac{(4\alpha+1)(N+2)}{4}}\bigg)-\frac{\tau t^{\frac{N(q-2)}{2}}}{q}\mathcal{O}\left(\eps^{\frac{N}{2}-\frac{(N-2)q}{8}}\right),~~~~\forall t>0.
\end{align}
Owing to the lack of smoothness of $I_0$, and by observing that for sufficiently small $\mu\in(0,1]$, the perturbed functional $I_\mu$ and the original functional $I_0$ are equivalent in a certain sense under the scaling transformation $t^{\frac{N}{2}}\widehat{V}_{\eps}(t\cdot)$, we can
demonstrate \eqref{classic eq esti} and further apply the compactness result established for the smooth functionals $I_\mu$ in Proposition \ref{converge process for mu to 0}.
For more details, see Lemma \ref{energy estimate2}.

\textbf{Secondly}, we develop two distinct convergence frameworks: one for the non-radial case through profile decomposition, and another for radial solutions using delicate compactness analysis.
Based on the perturbed method, for any given $\mu\in (0,1]$, we study the existence of solutions to the following perturbed problem:
\begin{equation}\label{pertu prob eq}
	\begin{cases}
		-\mu\Delta_{\theta} u-\Delta u-u\Delta (u^2)+\lambda u=\tau|u|^{q-2}u+|u|^{2\cdot2^*-2}u~~\text{in}~~\R^N,\\
		\int_{\R^N}|u|^2dx=c.
	\end{cases}
\end{equation}
We address the local minimizers $\{v_{\mu_n}\}$ and the mountain pass solutions $\{u_{\mu_n}\}$ separately, providing two different analytical versions of concentration compactness type results in Section \ref{Convergence issues} to establish strong convergence in $\widetilde{X}$.
As previously mentioned, as $\mu\to0^+$, these two convergence issues fundamentally relies on the Brezis-Lieb lemma for the quasilinear term as follows:
\begin{equation}\label{bl pro intro}
		\int_{\R^N}|w_{\mu_n}|^2|\nabla w_{\mu_n}|^2dx=\int_{\R^N}|\tilde w_{\mu_n}|^2|\nabla \tilde w_{\mu_n}|^2dx+\int_{\R^N}|\bar w|^2|\nabla \bar w|^2dx+o_n(1),
	\end{equation}
where $\tilde w_{\mu_n}:=w_{\mu_n}-\bar w$. To demonstrate \eqref{bl pro intro}, we first derive the $L^{\infty}$-estimate for the approximating solutions by using Moser's iteration, a process that simultaneously proves the existence of weak solutions to the original problem, see Lemma \ref{decomposition of quasi} for further details.
We note that this splitting property relies essentially on the relationship between solutions to the perturbed equations and those of the limiting equation. This not only enables us to establish a profile decomposition theorem in $\widetilde{X}$, but also serves as the crucial mechanism for our compactness analysis within the Sobolev critical regime.
Let us now briefly outline the proof strategies for both convergence schemes.

\begin{itemize}
	\item To prove Proposition \ref{mu to0 local mini}, we employ the profile decomposition theorem developed in Theorem \ref{profile decom} for the approximating solutions $\{v_{\mu_n}\}$ to separately decompose each term in the functional $I_{\mu_n}$.
By utilizing a novel scaling transformation as in \eqref{scal pro eq}, we project each decomposed component function onto the set $V_{\mu_n}(c)$. Using the definition of $m_{\mu_n}(c)$, we show that the perturbation tends to zero, i.e., $\mu_n\|\nabla v_{\mu_n}\|_\theta^\theta\to0$, and we establish strong convergence in the non-radial symmetric space $\widetilde{X}$.
	\item In the proof of Proposition \ref{converge process for mu to 0}, we discuss in the radially symmetric space $X_r$.
	For any fixed $\mu\in(0,1]$, we establish the existence of normalized solutions $u_\mu$ 
  to \eqref{pertu prob eq}. This proof is standard since, in this case, $2\cdot2^*$ is no longer a critical exponent for the perturbation problem.
	Nevertheless, we still confront a loss of compactness in the convergence issues as $\mu\to0^+$, due to the Sobolev critical term $|u|^{2\cdot2^*-2}u$.
	By utilizing the splitting property \eqref{bl pro intro}, we perform a detailed compactness analysis on the Nehari-Pohozaev manifold $\Lambda_{\mu}(c)$ to obtain an alternative conclusion.
\end{itemize}
We believe that these new convergence results constitute one of the central contributions of our work, offering fresh insights for further developments in the field, particularly for quasilinear problems with Sobolev critical growth.

\textbf{Finally}, the rigorous treatment of the Lagrange multiplier's non-vanishing property serves as a key element in our analysis, as it permits the relaxation of dimensional and exponent-range restrictions encountered in earlier studies. 
Unlike the classical Schr\"odinger equation, the Nehari-Pohozaev identity we derive in \eqref{pro local min eq4} includes additional gradient terms, which complicates the verification of $\lambda\neq0$. Consequently, our analysis distinguishes between two regimes.
\begin{itemize}
	\item In cases where precise energy estimates are unavailable: We consider sufficiently large $\tau>0$ to effectively control the mountain pass energy. This facilitates the application of our convergence theorem, as established in Proposition \ref{converge process for mu to 0}, and simultaneously ensure $\lambda>0$ through \eqref{pro local min eq4}.
	\item In cases where precise energy estimates are available: We innovatively exploit the monotonicity of the mountain pass level with respect to the mass $c>0$.  This approach enables the establishment of $\lambda>0$ for arbitrary $\tau>0$ and sufficiently small $c>0$.
\end{itemize}
Previous works \cite{LZ2023, ML2024, ZCW2023} focused on results for specific space dimensions. In contrast, this paper provides a comprehensive study of the existence and multiplicity of normalized solutions to equation \eqref{quasi eq}, covering the entire spectrum of subcritical perturbations in the interval $2<q<2\cdot2^*$ and for $N\ge3$. The main innovation of this work is the novel analytical approaches we employ to resolve convergence issues, which empower us to rigorously establish the non-triviality of the corresponding Lagrange multipliers. This breakthrough unifies the analysis across the full range of $q$ and $N$. We believe that this methodological advancement provides a valuable framework for addressing related problems in quasilinear Schr\"odinger equations.

\medskip

\subsection{Paper outline}
This paper is organized as follows.
\begin{itemize}
	\item Section \ref{intro} presents our main results and highlights of this paper.
	\item Section \ref{preliminary} provides some preliminaries that will be frequently used throughout the paper.
	\item Section \ref{Mountain pass solution subcritical} is devoted to the case where $q\in(2,2+\frac{4}{N})$. We obtain a local minimizer and study the mountain pass structure for $I_{\mu}|_{S(c)}$ in  Subsection \ref{Local minimum solution} and Subsection \ref{Mountain pass struc1}, respectively.
	\item Section \ref{Mountain pass solution subcritical1}  explores the case where $q\in[2+\frac{4}{N},4+\frac{4}{N})$, offering a precise energy estimate.
	\item Section \ref{mass super quasi} addresses the case where $q\in[4+\frac{4}{N},2\cdot2^*)$. We explore properties of the associated Nehari-Pohozaev manifold and establish an equivalent minimax characterization of the mountain pass level.
	\item Section \ref{compactness1} demonstrates the compactness of the Palais-Smale sequences obtained in Sections \ref{Mountain pass solution subcritical}-\ref{mass super quasi} and proves the existence of mountain pass type critical points for $I_{\mu}|_{S(c)}$.
	\item Section \ref{Convergence issues} investigates convergence issues as $\mu\to0^+$ and completes the proof of Theorems \ref{thm: main result1-1}-\ref{thm: main result1}, Theorem \ref{thm: main result2}, Theorems \ref{thm: main result3-1}-\ref{thm: main result3} and Theorem \ref{thm: main result4}.
	\item The Appendix provides an $L^{\infty}$ estimate, a splitting lemma and the profile decomposition of approximating solutions, all of which play a fundamental role  in proving our main results.
\end{itemize}

\medskip

{\small  Throughout this paper, we make use of the following notations:
\begin{itemize}
	\item  $L^s(\R^N),1\leq s\leq+\infty$, denotes the usual Lebesgue space endowed with the norm
		\begin{equation*}
			\|u\|_s:=\left(\int_{\R^N}|u|^sdx\right)^{\frac{1}{s}},~~1\leq s<+\infty,~~~~\|u\|_{\infty}:=\inf\{M>0:|u(x)|\leq M~\text{almost everywhere in}~\R^N\};
		\end{equation*}
	\item  $W^{1,s}(\R^N),1\leq s<+\infty$, denotes the usual Sobolev space endowed with the norm $\|u\|_{W^{1,s}(\R^N)}:=\left(\|\nabla u\|^s_s+\|u\|^s_s\right)^{\frac{1}{s}}$. In particular, $W^{1,2}(\R^N)$ is denoted by $H^1(\R^N)$;
	\item  $W^{1,s}_{rad}(\R^N):=\{u\in W^{1,s}(\R^N)~|~u(x)=u(|x|)~~\text{a.e. in}~~\R^N\}$;
	\item $C_0^{\infty}(\R^N)$ denotes the space of infinitely differentiable functions with compact support in $\R^N$;
	\item For any $x\in\R^N$ and $r>0$, $B_r(x):=\{y\in\R^N~:~|y-x|<r\}$ denotes the open ball centered at $x$ with radius $r$, and $B_r:=B_r(0)$ is the open ball centered at the origin with radius $r$;
	\item The symbols $"\rightarrow"$ and $"\rightharpoonup"$ represent strong convergence and weak convergence, respectively;
	\item $\mathcal{O}(h)$ and $o(h)$ mean that $|\mathcal{O}(h)|\leq C|h|$ and $o(h)/|h|\to0$, respectively;
	\item $C, C_1, C_2, \cdots$ denote positive constants whose value may change from line to line.
\end{itemize}
}

\section{Preliminary}\label{preliminary}\setcounter{equation}{0}

In this section, we present some preliminary results that will be utilized later. We begin with several useful inequalities.
The following Gagliardo-Nirenberg inequality can be found in \cite{1983wein}:
\begin{equation}\label{gn}
	\|u\|_p\leq \mathcal{C}_1(p,N)\|\nabla u\|_2^{\gamma_p}\|u\|_2^{1-\gamma_p},~~~\text{where}~~~N\ge3,~ p\in\left[2,\frac{2N}{N-2}\right].
\end{equation}

Next, we introduce the following sharp Gagliardo-Nirenberg type inequality, which is detailed in \cite{agueh2008}:
\begin{equation}\label{gn1}
	\int_{\R^N}|u|^{\frac{p}{2}}dx\leq\frac{C(p,N)}{\|Q_p\|_{1}^{\frac{p-2}{N+2}}}\left(\int_{\R^N}|u|dx\right)^{\frac{4N-p(N-2)}{2(N+2)}}\left(\int_{\R^N}|\nabla u|^2\right)^{\frac{N(p-2)}{2(N+2)}},~~\forall u\in\mathcal{E}^1,
\end{equation}
where $2<p<2\cdot2^*$, the constant $C(p,N)$ is explicitly given by
\begin{equation*}
	C(p,N)=\frac{p(N+2)}{[4N-(N-2)p]^{\frac{4-N(p-2)}{2(N+2)}}[2N(p-2)]^{\frac{N(p-2)}{2(N+2)}}},
\end{equation*}
and the Banach space $\mathcal{E}^q, q\ge1$, is defined as
\begin{equation*}
	\mathcal{E}^q:=\{u\in L^q(\R^N):|\nabla u|\in L^2(\R^N)\},
\end{equation*}
equipped with the norm $\|u\|_{\mathcal{E}^q}:=\|\nabla u\|_2+\|u\|_q$. For more details, see \cite{KP1997}.  Moreover, $Q_p$ is an extremal function for the inequality \eqref{gn1} in $\R^N$ and is the unique radially symmetric positive solution to the following equation (see \cite{ST2000}):
\begin{equation*}
	-\Delta u+1=u^{\frac{p}{2}-1}~~~\text{in}~\R^N.
\end{equation*}
By substituting $u$ with $u^2$ in \eqref{gn1}, we have the following Gagliardo-Nirenberg type inequality:
\begin{equation}\label{gn3}
	\int_{\R^N}|u|^pdx\leq\frac{C(p,N)}{\|Q_p\|_{1}^{\frac{p-2}{N+2}}}\left(\int_{\R^N}|u|^2dx\right)^{\frac{4N-p(N-2)}{2(N+2)}}\left(4\int_{\R^N}|u|^2|\nabla u|^2\right)^{\frac{N(p-2)}{2(N+2)}}.
\end{equation}
For simplicity, we rewrite \eqref{gn3} as:
\begin{equation}\label{gn2}
	\int_{\R^N}|u|^pdx\leq\mathcal{C}_2(p,N)\left(\int_{\R^N}|u|^2dx\right)^{\frac{4N-p(N-2)}{2(N+2)}}\left(\int_{\R^N}|u|^2|\nabla u|^2\right)^{\frac{N(p-2)}{2(N+2)}}.
\end{equation}

We need the Lions' lemma as follows.

\begin{lemma}\cite[Lemma I.1]{lionslemma}\label{lions lemma}
	Let $1<r\leq\infty, 1\leq s<\infty$ with $s\neq\frac{Nr}{N-r}$ if $r<N$. Assume that $\{u_n\}$ is bounded in $L^s(\R^N)$, $\{|\nabla u_n|\}$ is bounded in $L^r(\R^N)$, and
	\begin{equation*}
		\sup_{y\in\R^N}\int_{B_R(y)}|u_n|^sdx\to0,~~~\text{for some}~~R>0.
	\end{equation*}
	Then $u_n\to0$ as $n\to+\infty$ in $L^{\beta}(\R^N)$, for any $s<\beta<\frac{Nr}{N-r}$.
\end{lemma}

Adapting the argument from \cite[Lemma 2.1]{LZ2023}, we obtain the following result.

\begin{lemma}\label{poho id}
	Any critical point $u$ of $I_{\mu}$ restricted on $S(c)$ is contained in $\Lambda_{\mu}(c)$.
\end{lemma}

Define
\begin{equation*}
	S_r(c):=S(c)\cap X_r,~~~~X_r:=W^{1,\theta}_{rad}(\R^N)\cap H^1_{rad}(\R^N).
\end{equation*}
Let $u\in S_r(c)$ be arbitrary but fixed. For any given $\mu\in(0,1]$ and $s\in\R$, we introduce an auxiliary functional, similar to that in \cite{Jeanjean1997}, to construct a Palais-Smale sequence:
\begin{align*}
	\tilde{I}_{\mu}:~ &\R\times S_r(c)\to\R,\\
	&(s,u)\mapsto I_{\mu}(\mathcal{H}(s,u)),
\end{align*}
where
\begin{equation*}\label{scaling for fun}
	\mathcal{H}(s,u(x)):=e^{\frac{Ns}{2}}u(e^sx),~~\forall x\in\R^N.
\end{equation*}
By a direct computation, we obtain
\begin{align}\label{fibermap1}
	\tilde{I}_{\mu}(s,u)=\frac{\mu}{\theta}e^{\theta(1+\gamma_{\theta})s}\|\nabla u\|_\theta^\theta+\frac{e^{2s}}{2}\|\nabla u\|_2^2+e^{(N+2)s}\|u\nabla u\|_2^2-\frac{\tau\, e^{\frac{N(q-2)s}{2}}}{q}\|u\|_q^q-\frac{e^{N(2^*-1)s}}{2\cdot2^*}\|u\|_{2\cdot2^*}^{2\cdot2^*}.
\end{align}
Note that $\tilde{I}_{\mu}$ is of class $\mathcal{C}^1$, and a Palais-Smale sequence for $\tilde{I}_{\mu}|_{\R\times S_r(c)}$ is also a  Palais-Smale sequence for $\tilde{I}_{\mu}|_{\R\times S(c)}$,  as shown in \cite[Theorem 1.28]{Willem1996}.
We also recall that the tangent space at a point $u\in S(c)\subset X$ is given by
\begin{equation*}
	T_{u}S(c)=\{v\in X:\int_{\R^N}uvdx=0\}.
\end{equation*}
Following standard arguments as in \cite{LZ2023}, we have
\begin{lemma}\label{isomorphism thm}
	For $u\in S(c)$ and $s\in\R$, the map $\varphi\mapsto\mathcal{H}(s,\varphi)$ from $T_{u}S(c)$ to $T_{\mathcal{H}(s,\varphi)}S(c)$ is a linear isomorphism with inverse $\psi\mapsto\mathcal{H}(-s,\varphi)$.
\end{lemma}

\section{The case $2<q<2+\frac{4}{N}$}\label{Mountain pass solution subcritical}\setcounter{equation}{0}
\subsection{Local minimizers for the perturbed problem}\label{Local minimum solution}
Observe that for any $u\in S(c)$, using \eqref{gn2} and the Sobolev inequality, we get 
\begin{align}
	I_{\mu}(u)&\ge\frac{\mu}{\theta}\int_{\R^N}|\nabla u|^{\theta}dx+\frac{1}{2}\int_{\R^N}|\nabla u|^{2}dx+\int_{\R^N}|u|^2|\nabla u|^{2}dx\nonumber\\
	&\quad-\frac{\tau \mathcal{C}_2(q,N)c^{\frac{4N-q(N-2)}{2(N+2)}}}{q}\left(\int_{\R^N}|u|^2|\nabla u|^{2}dx\right)^{\frac{N(q-2)}{2(N+2)}}-\frac{1}{2\cdot2^*}\left(\frac{4}{\mathcal{S}}\right)^{\frac{2^*}{2}}\left(\int_{\R^N}|u|^2|\nabla u|^{2}dx\right)^{\frac{N}{N-2}}\nonumber\\
	&\ge\xi_{\mu}(u)\left(\frac{1}{2}-\frac{\tau \mathcal{C}_2(q,N)c^{\frac{4N-q(N-2)}{2(N+2)}}}{q}(\xi_{\mu}(u))^{\frac{N(q-2)}{2(N+2)}-1}-\frac{1}{2\cdot2^*}\left(\frac{4}{\mathcal{S}}\right)^{\frac{2^*}{2}}(\xi_{\mu}(u))^{\frac{2}{N-2}}\right), \label{iupro1}
\end{align}
where
\begin{equation*}
	\xi_{\mu}(u):=\frac{\mu}{\theta}\int_{\R^N}|\nabla u|^{\theta}dx+\int_{\R^N}|\nabla u|^{2}dx+\int_{\R^N}|u|^2|\nabla u|^{2}dx.
\end{equation*}

Set
\begin{equation}\label{alpha012}
	\alpha_0:=\frac{N(q-2)}{2(N+2)}-1,~~~\alpha_1:=\frac{4N-q(N-2)}{2(N+2)},~~~\alpha_2:=\frac{2}{N-2}.
\end{equation}
Since $q<2+\frac{4}{N}$, it is easy to verify that
\begin{equation*}
	\alpha_0\in(-1,0),~~~\alpha_1\in\left(\frac{2}{N},1\right),~~~\alpha_2\in(0,2].
\end{equation*}
We define
\begin{equation*}
	f(c,\rho):=\frac{1}{2}-\frac{\tau \mathcal{C}_2(q,N)}{q}c^{\alpha_1}\rho^{\alpha_0}-\frac{1}{2\cdot2^*}\left(\frac{4}{\mathcal{S}}\right)^{\frac{2^*}{2}}\rho^{\alpha_2}.
\end{equation*}
Set $g_{c}(\rho):=f(c,\rho)$. Similar to the proof in \cite[Lemma 2.1]{JJLV2022}, we obtain the following result.
\begin{lemma}\label{functioncrho pro}
The function $g_{c}(\rho)$ admits a unique global maximum for each $c>0$, and the maximum value satisfies
	\begin{equation*}
		\begin{cases}
			\max\limits_{\rho>0}g_c(\rho)>0~~~\text{if}~~~c<c_0,\\
			\max\limits_{\rho>0}g_c(\rho)=0~~~\text{if}~~~c=c_0,\\
			\max\limits_{\rho>0}g_c(\rho)<0~~~\text{if}~~~c>c_0,
		\end{cases}
	\end{equation*}
	where $c_0$ is defined by \eqref{func pro c value}.
\end{lemma}
\begin{remark}
Taking into account the ranges of $\alpha_0$ and $\alpha_2$, a standard argument establishes that the function $g_{c}(\rho)$ admits a unique maximum point at $\rho_c$ given by
\begin{equation*}\label{func pro rho value}
	\rho_c:=\left[-\frac{\alpha_0}{\alpha_2}\frac{\tau \mathcal{C}_2(q,N)2^*\mathcal{S}^{\frac{2^*}{2}}}{2^{2^*-1}q}\right]^{\frac{1}{\alpha_2-\alpha_0}}c^{\frac{\alpha_1}{\alpha_2-\alpha_0}}.
\end{equation*}
Obviously, $\rho_0:=\rho_{c_0}>0$ depends only on $\tau, q, N$ and is independent of $\mu\in(0,1]$. Moreover, it is straightforward to verify that Lemma \ref{functioncrho pro} holds for both the $\mu=0$ case and the $L^2$-subcritical case, i.e., $2<q<4+\frac{4}{N}$.
\end{remark}

\begin{lemma}\label{functioncrho pro1} Let $c_1>0$ and $\rho_1>0$ satisfy $f(c_1,\rho_1)\ge0$. Then
$f(c_2,\rho_2)\ge0$ for all $c_2\in(0,c_1]$ and $\rho_2\in\left[\frac{c_2}{c_1}\rho_1,\rho_1\right]$.
\end{lemma}
\begin{proof}
The proof can be derived following the arguments in \cite{JJLV2022}, and is therefore omitted here.
\end{proof}
\begin{remark}\label{remark f pro1}
	For further reference, we note that by Lemma \ref{functioncrho pro1}, it is not difficult to obtain that $f(c_0,\rho_0)=0$ and $f(c,\rho_0)>0$ for all $c\in(0,c_0)$. Moreover, for any $c_1\in(0,c]$, we have
	\begin{equation*}
		f(c_1,\rho)\ge0,~~~\text{for any}~~~~\rho\in\left[\frac{c_1}{c}\rho_0,\rho_0\right].
	\end{equation*}
\end{remark}

Define $\mathcal{B}_{\rho_0}:=\{u\in X: \xi_{\mu}(u)<\rho_0\}$. Consider the set $V_\mu(c)$ given by
\begin{equation}\label{local mini set}
	V_\mu(c):=S(c)\cap \mathcal{B}_{\rho_0},
\end{equation}
and the minimization problem:
\begin{equation*}
	m_\mu(c):=\inf_{u\in V_\mu(c)} I_{\mu}(u),~~~~\text{for any}~~c\in(0,c_0).
\end{equation*}
The properties of $m_\mu(c)$ are stated as follows.

\begin{lemma}\label{mcproperty}
	For any $\mu\in(0,1]$ and $c\in(0,c_0)$, the following three assertions hold.
	\begin{enumerate}[label=(\roman*)]
		\item $m_\mu(c)<0<\inf\limits_{u\in\partial V_\mu(c)}I_{\mu}(u)$.
		\item The function $c\mapsto m_\mu(c)$ is a continuous mapping.
		\item For all $\alpha\in(0,c)$, we have $m_\mu(c)\leq m_\mu(\alpha)+m_\mu(c-\alpha)$. Moreover, if $m_\mu(\alpha)$ or $m_\mu(c-\alpha)$ is attained, then the inequality is strict.
	\end{enumerate}
\end{lemma}
\begin{proof}
\begin{enumerate}[label=(\roman*)]
\item 
For any $u\in \partial V_\mu(c)$, using \eqref{iupro1} and Remark \ref{remark f pro1}, we get that
$$I_\mu(u)\geq \xi_{\mu}(u)f(\|u\|_2^2,\xi_{\mu}(u)) =\rho_0 f(c,\rho_0)>0. $$
Now for any $u\in S(c)$, define $u_s(x):=s^{\frac{N}{2}}u(sx), \forall s\in(0, +\infty)$.
Then we have
\begin{align}\label{mcproperty eq1-1}
	\psi_u(s)&:=I_\mu(u_s)=\frac{\mu}{\theta}s^{\theta(1+\gamma_\theta)}\int_{\R^N}|\nabla u|^\theta dx +\frac{s^2}{2}\int_{\R^N}|\nabla u|^2dx+s^{N+2}\int_{\R^N}|u|^2|\nabla u|^2dx\nonumber\\
				&\quad-\frac{\tau}{q}s^{\frac{N(q-2)}{2}}\int_{\R^N}|u|^qdx-\frac{1}{2\cdot2^*}s^{N(2^*-1)}\int_{\R^N}|u|^{2\cdot2^*}dx.
\end{align}
Due to the fact that $\frac{N(q-2)}{2}<2$ and $N(2^*-1)>N+2$, we deduce that there exists sufficiently small $s_0>0$ such that $\xi_\mu(u_{s_0})<\rho_0$ and $I_\mu(u_{s_0})=\psi_u(s_0)<0$. Hence, $m_\mu(c)<0$.
\item 
Fix an arbitrary $c\in(0,c_0)$. Assume that $\{c_n\}\subset (0,c_0)$ is a sequence such that $c_n\rightarrow c$. By (i), for any $\epsilon>0$, there exists a sequence $\{u_n\}\subset V_\mu(c_n)$ satisfying
\begin{align}\label{mcproperty eq3}
I_\mu(u_n)\leq m_\mu(c_n)+\epsilon \quad  \text{and} \quad  I_\mu(u_n)<0.
\end{align}
We set $v_n(x):=t_n^{\frac{1}{N+2}}u_n(t_n^{-\frac{1}{N+2}}x)$ for $t_n=\frac{c}{c_n}$. Consequently, $v_n\in S(c)$. Moreover,
if $c_n\geq c$, which means $c_n\rightarrow c^{+}$, we have
\begin{align*}
	\frac{\mu}{\theta}\|\nabla v_n\|_{\theta}^{\theta}+\|\nabla v_n\|_2^2+\|v_n\nabla v_n\|_2^2&\leq\frac{\mu}{\theta}\|\nabla u_n\|_{\theta}^{\theta}+\|\nabla u_n\|_2^2+\|u_n\nabla u_n\|_2^2<\rho_0,
\end{align*}
which implies that $v_n\in V_\mu(c)$. On the other hand, taking $c_n< c$ such that $c_n\rightarrow c^{-}$, by Lemma \ref{functioncrho pro1}, we can see that $f(c_n,\rho)\ge0$ for any $\rho\in\left[\frac{c_n}{c}\rho_0,\rho_0\right]$. Using Remark \ref{remark f pro1} again, we get that $f(c_n,\xi_\mu(u_n))<0$, which implies that $\xi_\mu(u_n)<\frac{c_n}{c}\rho_0$ and
\begin{align*}
	\frac{\mu}{\theta}\|\nabla v_n\|_{\theta}^{\theta}+\|\nabla v_n\|_2^2+\|v_n\nabla v_n\|_2^2&\leq\frac{\mu}{\theta}t_n\|\nabla u_n\|_{\theta}^{\theta}+t_n\|\nabla u_n\|_2^2+t_n\|u_n\nabla u_n\|_2^2<\rho_0.
\end{align*}
This implies that $v_n\in V_\mu(c)$. Therefore, we deduce that
\begin{equation*}
	m_\mu(c)\leq I_\mu(v_n)=I_\mu(u_n)+(I_\mu(v_n)-I_\mu(u_n))=I_\mu(u_n)+o_n(1),
\end{equation*}
which together with \eqref{mcproperty eq3} implies that
\begin{equation*}
	m_\mu(c)\leq m_\mu(c_n)+\varepsilon+o_n(1).
\end{equation*}
Reversing the argument we can deduce similarly that $m_\mu(c_n)\leq m_\mu(c)+\varepsilon+o_n(1)$.
Therefore, (\romannumeral2) holds.
\item 
It is suffices to prove that, for any fixed $\alpha\in (0,c)$,  the following inequality holds:
\begin{align}\label{mcproperty eq4}
m_\mu(t\alpha)\leq tm_\mu(\alpha),~~~\forall ~t\in \left(1,\frac{c}{\alpha}\right].
\end{align}
To prove \eqref{mcproperty eq4}, we first observe, by virtue of point (\romannumeral1), for any $\varepsilon>0$ sufficiently small, there exists $u\in V_\mu(\alpha)$ such that
\begin{align}\label{mcproperty eq5}
I_\mu(u)\leq m_\mu(\alpha)+\varepsilon \quad \text{and} \quad I_\mu(u)<0.
\end{align}
By Lemma \ref{functioncrho pro1} and \eqref{mcproperty eq5}, we get $\xi_\mu(u)<\frac{\alpha}{c}\rho_0$.
Next, we define a new function $v(x):=t^{\frac{1}{N+2}}u(t^{-\frac{1}{N+2}}x)$. Then 
\begin{equation*}\begin{aligned}
\xi_\mu(v)<t \xi_\mu(u)<\rho_0.
\end{aligned}
\end{equation*}
This implies that $v\in V_\mu(t \alpha)$. Therefore, we obtain
\begin{equation*}
m_\mu(t\alpha)\leq I_\mu(v)<tI_\mu(u)\leq t(m_\mu(\alpha)+\varepsilon).
\end{equation*}
Since $\varepsilon>0$ is arbitrary, we conclude that $m_\mu(t\alpha)\leq tm_\mu(\alpha)$.  If $m_\mu(\alpha)$ is attained, then there exists a function $u_0\in V_\mu(\alpha)$ such that $I_{\mu}(u_0)=m_{\mu}(\alpha)$ and $I_\mu(u_0)<0$. In this case, we can take $\varepsilon=0$ in \eqref{mcproperty eq5}, and the strict inequality follows immediately.
\end{enumerate}
\end{proof}

The following result is a direct consequence of applying Lemma \ref{lions lemma}.
\begin{lemma}\label{nonvanish}
	Assume that $\mu\in(0,1]$ and $c\in (0,c_0)$, and a sequence $\{u_n\}\subset \mathcal{B}_{\rho_0}$ satisfies $\|u_n\|^2_2\rightarrow c$ and $I_\mu(u_n)\rightarrow m_\mu(c)$. Then, there exist $\delta>0$ and a sequence $\{y_n\}\subset\mathbb{R}^N$ such that for some $R>0$,
  \begin{equation*}\label{nonvanish eq1}
  \int_{B_R(y_n)}|u_n|^2dx\geq \delta>0.
  \end{equation*}  
\end{lemma}

\begin{lemma}\label{strong convergence}
	Assume that $\mu\in(0,1]$ and $c\in (0,c_0)$, and a sequence $\{u_n\}\subset \mathcal{B}_{\rho_0}$ satisfies $\| u_n\|^2_2\rightarrow c$ and $I_{\mu}(u_n)\to m_\mu(c)$. Then there exists $v_\mu\in V_{\mu}(c)$ such that, up to translation, $u_n\rightarrow v_\mu$ in $X$.
\end{lemma}
\begin{proof}
	We employ the concentration-compactness principle to analyze the sequences $\{u_n\}$. Since $\{u_n\}\subset \mathcal{B}_{\rho_0}$, it follows that $\{u_n\}$ is bounded in $X$ and $\{u_n\nabla u_n\}$ is bounded in $L^2(\R^N)^N$. Define $\omega_n:=u_n(\cdot+y_n)$. Then, by Lemma \ref{nonvanish} and the Rellich compactness theorem, there exists $0\neq v_\mu\in X$ such that, up to a subsequence,
	\begin{equation}\label{strong eq1}
		\begin{cases}
		\omega_n\rightharpoonup v_\mu ~~~&\text{in}~~~ W^{1,\theta}(\R^N),\\
		\omega_n\rightharpoonup v_\mu ~~~&\text{in}~~~ H^1(\R^N),\\
		\omega_n\nabla \omega_n\rightharpoonup \tilde\omega~~~&\text{in}~~~L^2(\R^N)^N,\\
		\omega_n\to v_\mu~~~&\text{in}~~~L^{\beta}_{loc}(\R^N),~~\beta\in[1,\frac{N\theta}{N-\theta}).
		\end{cases}
	  \end{equation}
	 Next, we claim that $\tilde\omega=v_\mu\nabla v_\mu$. To see this, observe that since $\|\nabla \omega_n^2\|_2^2=4\int_{\R^N}|\omega_n|^2|\nabla \omega_n|^2dx$ is bounded, we have $\omega_n^2\rightharpoonup \bar\omega$ in $H^1(\R^N)$. By the uniqueness of weak limits, it follows that $\bar\omega=v_\mu^2$. Consequently, we conclude that $\tilde\omega=v_\mu\nabla v_\mu$, and the claim holds.

	  Denote $c_1=\|v_\mu\|_2^2>0$. Clearly, $c_1\leq c$. In what follows, we are going to prove that the dichotomy does not occur, i.e., $c_1=c$. Arguing by contradiction, we may assume that $0<c_1<c$. First of all, for any $\eps>0$, there exists $R_1=R_1(\eps)>0$ such that
	  \begin{equation*}
		c_1-\eps<\int_{B_{R_1}}|v_\mu|^2dx\leq c_1.
	  \end{equation*}
	  Thus, for $n$ sufficiently large, we have
	  \begin{equation}\label{strong eq2}
		c_1-\eps<\int_{B_{R_1}}|\omega_n|^2dx\leq c_1~~~~\text{and}~~~~\int_{2B_{R_1}\backslash B_{R_1}}|\omega_n|^2dx\leq\eps.
	  \end{equation}
	  Let $\tilde{R}=\max\{R,R_1\}$, where $R$ is given in Lemma \ref{nonvanish}. Choose a smooth function $\eta_{\tilde{R}}$ defined by
	  \begin{equation*}
		\eta_{\tilde{R}}(t):=
		\begin{cases}
			1~~~~&\text{if}~~~0\leq|t|\leq\tilde{R},\\
			0~~~~&\text{if}~~~|t|\ge2\tilde{R},
		\end{cases}
	  \end{equation*}
	  with $|\nabla\eta_{\tilde{R}}|\leq\frac{2}{\tilde{R}}$. Set $z_n(x):=\eta_{\tilde{R}}(|x-y_n|)u_n(x)$ and $w_n(x):=\left(1-\eta_{\tilde{R}}(|x-y_n|)\right)u_n(x)$. Clearly, $z_n$ and $w_n$ belong to $X$ and $u_n=z_n+w_n$. Moreover, we have
	  \begin{equation*}
		\liminf_{n\to+\infty}\int_{B_{\tilde{R}}(y_n)}|z_n|^2dx=\liminf_{n\to+\infty}\int_{B_{\tilde{R}}(y_n)}|u_n|^2dx\ge\delta>0.
	  \end{equation*}
	  From this, we derive that 
	   \begin{align}\label{strong eq8}
		I_\mu(u_n)&=I_{\mu}(z_n)+I_{\mu}(w_n)+\frac{\mu}{\theta}\int_{B_{2\tilde{R}}(y_n)\backslash B_{\tilde{R}}(y_n)}\left(|\nabla u_n|^{\theta}-|\nabla z_n|^{\theta}-|\nabla w_n|^{\theta}\right)dx\nonumber\\
		&\quad+\frac{1}{2}\int_{B_{2\tilde{R}}(y_n)\backslash B_{\tilde{R}}(y_n)}\left(|\nabla u_n|^{2}-|\nabla z_n|^{2}-|\nabla w_n|^{2}\right)dx\nonumber\\
		&\quad+\int_{B_{2\tilde{R}}(y_n)\backslash B_{\tilde{R}}(y_n)}\left(|u_n|^2|\nabla u_n|^{2}-|z_n|^2|\nabla z_n|^{2}-|w_n|^2|\nabla w_n|^{2}\right)dx\nonumber\\
		&\quad-\frac{\tau}{q}\int_{B_{2\tilde{R}}(y_n)\backslash B_{\tilde{R}}(y_n)}\left(|u_n|^{q}-|z_n|^{q}-|w_n|^{q}\right)dx\nonumber\\
		&\quad-\frac{1}{2\cdot2^*}\int_{B_{2\tilde{R}}(y_n)\backslash B_{\tilde{R}}(y_n)}\left(|u_n|^{2\cdot2^*}-|z_n|^{2\cdot2^*}-|w_n|^{2\cdot2^*}\right)dx.
	  \end{align}
	We can now verify that
	  \begin{equation}\label{strong eq3}
		\int_{B_{2\tilde{R}}(y_n)\backslash B_{\tilde{R}}(y_n)}\left(|\nabla u_n|^{\theta}-|\nabla z_n|^{\theta}-|\nabla w_n|^{\theta}\right)dx\ge\int_{B_{2\tilde{R}}(y_n)\backslash B_{\tilde{R}}(y_n)}\left(1-\eta_{\tilde{R}}^{\theta}-(1-\eta_{\tilde{R}})^{\theta}\right)|\nabla u_n|^{\theta}dx-o_{\tilde{R}}(1),
	  \end{equation}
	  \begin{equation}\label{strong eq4}
		\int_{B_{2\tilde{R}}(y_n)\backslash B_{\tilde{R}}(y_n)}\left(|\nabla u_n|^{2}-|\nabla z_n|^{2}-|\nabla w_n|^{2}\right)dx\ge\int_{B_{2\tilde{R}}(y_n)\backslash B_{\tilde{R}}(y_n)}\left(1-\eta_{\tilde{R}}^{2}-(1-\eta_{\tilde{R}})^{2}\right)|\nabla u_n|^{2}dx-o_{\tilde{R}}(1),
	  \end{equation}
	  \begin{align}\label{strong eq5}
		&\int_{B_{2\tilde{R}}(y_n)\backslash B_{\tilde{R}}(y_n)}\left(|u_n|^2|\nabla u_n|^{2}-|z_n|^2|\nabla z_n|^{2}-|w_n|^2|\nabla w_n|^{2}\right)dx\nonumber\\
		&\ge\int_{B_{2\tilde{R}}(y_n)\backslash B_{\tilde{R}}(y_n)}\left(1-\eta_{\tilde{R}}^{4}-(1-\eta_{\tilde{R}})^{4}\right)|u_n|^2|\nabla u_n|^{2}dx-o_{\tilde{R}}(1).
	  \end{align}
	  Moreover, using \eqref{strong eq2} and the interpolation inequality, we obtain
	  \begin{equation}\label{strong eq6}
		\left|\int_{B_{2\tilde{R}}(y_n)\backslash B_{\tilde{R}}(y_n)}\left(|u_n|^{q}-|z_n|^{q}-|w_n|^{q}\right)dx\right|\leq3\eps
	  \end{equation}
	  and
	  \begin{equation}\label{strong eq7}
		\left|\int_{B_{2\tilde{R}}(y_n)\backslash B_{\tilde{R}}(y_n)}\left(|u_n|^{2\cdot2^*}-|z_n|^{2\cdot2^*}-|w_n|^{2\cdot2^*}\right)dx\right|\leq3\eps,
	  \end{equation}
	  for $n$ large enough.
	  Substituting \eqref{strong eq3}-\eqref{strong eq7} into \eqref{strong eq8}, we deduce that
	  \begin{equation}\label{strong eq9}
		I_\mu(u_n)\ge I_{\mu}(z_n)+I_{\mu}(w_n)-o_n(1)-o_{\tilde{R}}(1).
	  \end{equation}
	  By a direct computation, we have
	  \begin{align*}
		\|w_n\|_2^2&=\left\langle u_n(x)-\eta_{\tilde{R}}(|x-y_n|)u_n(x),u_n(x)-\eta_{\tilde{R}}(|x-y_n|)u_n(x)\right\rangle\\
		&=\left\langle u_n,u_n\right\rangle-2\left\langle u_n,\eta_{\tilde{R}}(|x-y_n|)u_n\right\rangle+\left\langle \eta_{\tilde{R}}(|x-y_n|)u_n,\eta_{\tilde{R}}(|x-y_n|)u_n\right\rangle.
	  \end{align*}
	  It follows from $\omega_n\to v_\mu$ in $L^2_{loc}(\R^N)$ that as $n\to+\infty$,
	  \begin{equation*}
		\|w_n\|_2^2\to c-2\int_{\R^N}\eta_{\tilde{R}}v_\mu^2dx+\int_{\R^N}\eta_{\tilde{R}}^2v_\mu^2dx:=\delta_{\tilde{R}}>0.
	  \end{equation*}
	 Moreover, observe that
	  \begin{align*}
		\xi_{\mu}(w_n)\leq o_{\tilde{R}}(1)+\frac{\mu}{\theta}\int_{\R^N\backslash B_{\tilde{R}}}|\nabla \omega_n|^{\theta}dx+\int_{\R^N\backslash B_{\tilde{R}}}|\nabla \omega_n|^{2}dx+\int_{\R^N\backslash B_{\tilde{R}}}|\omega_n|^2|\nabla \omega_n|^{2}dx,
	  \end{align*}
	  which implies that $\{w_n\}\subset\mathcal{B}_{\rho_0}$ for $\tilde{R}$ sufficiently large. Hence, in view of Lemma \ref{mcproperty} (\romannumeral2) we get
	  \begin{equation*}
		\liminf_{n\to+\infty}~I_{\mu}(w_n)\ge\liminf_{n\to+\infty}~m_\mu(\|w_n\|_2^2)= m_\mu(\delta_{\tilde{R}}).
	  \end{equation*}
	  Taking the limit as $n\to+\infty$ in \eqref{strong eq9} and using \eqref{strong eq1}, we infer that
	  \begin{equation*}
		m_\mu(c)\ge m_\mu(\delta_{\tilde{R}})+I_{\mu}(\eta_{\tilde{R}}v_\mu)-o_{\tilde{R}}(1).
	  \end{equation*}
	 Then, letting $\tilde{R}\to+\infty$, we deduce that
	  \begin{equation}\label{strong eq10}
		m_\mu(c)\ge m_\mu(c-c_1)+I_{\mu}(v_\mu).
	  \end{equation}
	  By the lower semi-continuity we have $v_\mu\in V_\mu(c_1)$, which implies that $I_\mu(v_\mu)\geq m_\mu(c_1)$. Suppose $I_\mu(v_\mu)>m_\mu(c_1)$. Then, combining \eqref{strong eq10} and Lemma \ref{mcproperty} (\romannumeral3), we obtain
 \begin{align*}
m_\mu(c)>m_\mu(c-c_1)+m_\mu(c_1)\geq m_\mu(c-c_1+c_1)=m_\mu(c),
  \end{align*}
which yields a contradiction. Therefore, it must hold that $I_\mu(v_\mu)=m_\mu(c_1)$, implying that $v_\mu$ is a local minimizer of $I_\mu$ in $V_\mu(c_1)$. Using Lemma \ref{mcproperty} (\romannumeral3) and \eqref{strong eq10}, we obtain
 \begin{align*}
m_\mu(c)\geq m_\mu(c-c_1)+I_\mu(v_\mu)=m_\mu(c-c_1)+m_\mu(c_1)> m_\mu(c-c_1+c_1)=m_\mu(c),
  \end{align*}
which is also a contradiction. Hence, it follows that $\|v_\mu\|^2_2=c$. Consequently, $\omega_n\to v_\mu$ in $L^2(\R^N)$. Using the lower semi-continuity and interpolation, we infer that
\begin{align*}
	&m_\mu(c)=\lim_{n\to+\infty}I_\mu(u_n)=\lim_{n\to+\infty}I_\mu(\omega_n)\ge I_\mu(v_\mu)\ge m_\mu(c).
\end{align*}
This implies that $\|\nabla\omega_n\|^2_2\rightarrow \|\nabla v_\mu\|^2_2$, $\| \nabla \omega_n\|^\theta_\theta\rightarrow \| \nabla v_\mu\|^\theta_\theta$ and $\|\omega_n\nabla\omega_n\|^2_2\rightarrow \|v_\mu\nabla v_\mu\|^2_2$. Therefore, $u_n\rightarrow v_\mu$ in $X$.
\end{proof}

Building on the preceding preliminary results, we conclude that
\begin{lemma}\label{local minimum for perturbed}
  For any fixed $\mu\in(0,1]$ and $c\in(0,c_0)$, the functional $I_\mu$ has a critical point $v_\mu\in V_\mu(c)$ satisfying
  \begin{equation*}
	I_\mu(v_\mu)=m_\mu(c),~~~I'_\mu(v_\mu)+\tilde\lambda_\mu v_\mu=0,
  \end{equation*}
  for some $\tilde\lambda_\mu\in\R$.
\end{lemma}
\begin{proof}
For any sequence $\{u_n\}\subset V_\mu(c)$ satisfying $I_\mu(u_n)\rightarrow m_\mu(c)$, we deduce from Lemma \ref{strong convergence} that, up to translation, $u_n\rightarrow v_\mu\in V_{\mu}(c)$ in $X$.
Furthermore, using the fact that $m_\mu(c)<0$ and the argument presented in Lemma \ref{strong convergence}, we conclude that $v_\mu\in V_\mu(c)$ is a local minimizer for $I_\mu$ on $V_\mu(c)$. 
\end{proof}
\begin{remark}\label{symmetric local}
	For further reference, we note that $v_\mu\in V_\mu(c)$ obtained in the above lemma is positive and radially symmetric non-increasing.
	In fact, let $v_\mu^*$ denote the Schwartz rearrangement of $|v_\mu|$. Consequently, $v_\mu^*$ is a positive, radially symmetric, and non-increasing function. Moreover, it is straightforward to verify that Lemma 4.3 in \cite{CJS2010} also holds for the Sobolev critical case, that is, 
\begin{equation*}
	\|\nabla v_\mu^*\|_{\theta}^{\theta}\leq\|\nabla v_\mu\|_{\theta}^{\theta},~~~\|\nabla v_\mu^*\|_{2}^{2}+\|v_\mu^*\nabla v_\mu^*\|_{2}^{2}\leq\|\nabla v_\mu\|_{2}^{2}+\|v_\mu\nabla v_\mu\|_{2}^{2},~~~\|v_\mu^*\|_{2}^{2}=\|v_\mu\|_{2}^{2}.
\end{equation*}
Thus, $v_\mu^*\in V_\mu(c)$ and $I_\mu(v_\mu^*)=I_\mu(v_\mu)$, which implies that $m_\mu(c)$ is reached by $v_\mu^*$ satisfying
\begin{equation*}
	-\mu\Delta_{\theta} v_\mu^*-\Delta v_\mu^*-v_\mu^*\Delta (v_\mu^{*2})+\tilde\lambda_{\mu} v_\mu^*=\tau|v_\mu^*|^{q-2}v_\mu^*+|v_\mu^*|^{2\cdot2^*-2}v_\mu^*~~~\text{in}~~~\R^N,
\end{equation*}
for some $\tilde\lambda_{\mu}\in\R$. Hereafter, without loss of generality, we still denote $v_\mu^*$ by $v_\mu$.
\end{remark}

\subsection{The mountain pass structure and Palais-Smale sequences}\label{Mountain pass struc1}
Starting from the previous local minimizer $u_{\mu,c}:=v_\mu$ obtained in Lemma \ref{local minimum for perturbed}, we can derive the following mountain pass structure for $I_{\mu}$ restricted on $S_r(c)$.
\begin{lemma}\label{mountain pass struc}
	For any $c\in(0,c_0)$ and fixed $\mu\in(0,1]$, there exists $\zeta_{0}(c)>0$ independent of $\mu$ such that
	\begin{equation*}\label{mountain pass struc eq1}
		M_{\mu}(c):=\inf\limits_{\gamma\in\Gamma_{\mu}(c)}\max\limits_{t\in[0,1]}I_{\mu}(\gamma(t))\ge\zeta_{0}(c)>\sup\limits_{\gamma\in\Gamma_{\mu}(c)}\max\{I_{\mu}(\gamma(0)), I_{\mu}(\gamma(1))\},
	\end{equation*}
	where
	\begin{equation*}\label{mountain pass struc eq2}
		\Gamma_{\mu}(c):=\{\gamma\in \mathcal{C}([0,1], S_r(c)): \gamma(0)=u_{\mu,c}, I_{\mu}(\gamma(1))<2m_{\mu}(c)\}.
	\end{equation*}
\end{lemma}
\begin{proof}
	Define
	\begin{equation*}
		V_0(c):=\{u\in S(c):\|\nabla u\|_2^2+\|u\nabla u\|_2^2<\rho_0\}.
	\end{equation*}
	Then for any $u\in\partial V_0(c)$, we have
	\begin{align*}
		I_\mu(u)&\ge\left(\|\nabla u\|_2^2+\|u\nabla u\|_2^2\right)\bigg(\frac{1}{2}-\frac{\tau \mathcal{C}_2(q,N)c^{\frac{4N-q(N-2)}{2(N+2)}}}{q}\left(\|\nabla u\|_2^2+\|u\nabla u\|_2^2\right)^{\frac{N(q-2)}{2(N+2)}-1}\\
		&\quad-\frac{1}{2\cdot2^*}\left(\frac{4}{\mathcal{S}}\right)^{\frac{2^*}{2}}\left(\|\nabla u\|_2^2+\|u\nabla u\|_2^2\right)^{\frac{2}{N-2}}\bigg)\\
		&=\rho_0 f(c,\rho_0):=\zeta_{0}(c)>0.
	\end{align*}
	For any $\gamma\in\Gamma_{\mu}(c)$, we have $\gamma(0)=u_{\mu,c}\in V_{\mu}(c)\subset V_{0}(c)$ and $I_{\mu}(\gamma(1))<2m_{\mu}(c)$, which implies that $\gamma(1)\notin V_{\mu}(c)$. Moreover, by the continuity of $\gamma(t)$ on $[0,1]$, there exists $t_0\in(0,1)$ such that
	\begin{equation*}
		\gamma(t_0)\in\partial V_{0}(c)~~~\text{and}~~~\max\limits_{t\in[0,1]}I_{\mu}(\gamma(t))\ge I_{\mu}(\gamma(t_0))\ge\zeta_{0}(c),
	\end{equation*}
	which implies that the conclusion is valid.
\end{proof}

We introduce the family of paths
\begin{equation*}
	\tilde{\Gamma}_{\mu}(c):=\{\tilde{\gamma}\in \mathcal{C}([0,1], \R\times S_r(c)): \tilde{\gamma}(0)=(0,u_{\mu,c}), \tilde{I}_{\mu}(\tilde{\gamma}(1))<2m_{\mu}(c)\}
\end{equation*}
and the associated minimax value
\begin{equation*}
	\tilde{M}_{\mu}(c):=\inf\limits_{\tilde{\gamma}\in\tilde{\Gamma}_{\mu}(c)}\max\limits_{t\in[0,1]}\tilde{I}_{\mu}(\tilde{\gamma}(t)),
\end{equation*}
where $ \tilde{I}_{\mu}$ is given by \eqref{fibermap1}.
Then we have the following result:

\begin{lemma}\label{eq of the mpv}
	Let $c\in(0,c_0)$. Then $M_{\mu}(c)=\tilde{M}_{\mu}(c)$.
\end{lemma}
\begin{proof}
	On the one hand, for any $\gamma\in\Gamma_{\mu}(c)$, let $\tilde{\gamma}(t):=(0,\gamma(t))$. It is easy to observe that
	\begin{equation*}
		\tilde{\gamma}\in\tilde{\Gamma}_{\mu}(c),~~\mbox{and}~~\tilde{I}_{\mu}(\tilde{\gamma}(t))=\tilde{I}_{\mu}(0,\gamma(t))=I_{\mu}(\gamma(t)),
	\end{equation*}
	which implies that $\tilde{M}_{\mu}(c)\leq M_{\mu}(c)$.
	On the other hand, for any $\tilde{\gamma}\in\tilde{\Gamma}_{\mu}(c)$, we have $\mathcal{H}\circ\tilde{\gamma}\in\Gamma_{\mu}(c)$.
	By Lemma \ref{mountain pass struc}, there exists $\tilde{\zeta}_{\mu}(c)>0$ such that
	\begin{equation*}
		\inf\limits_{\tilde{\gamma}\in\tilde{\Gamma}_{\mu}(c)}\max\limits_{t\in[0,1]}\tilde{I}_{\mu}(\tilde{\gamma}(t))\ge\zeta_{\mu}(c)>\tilde{\zeta}_{\mu}(c)\ge\sup\limits_{\tilde{\gamma}\in\tilde{\Gamma}_{\mu}(c)}\max\{\tilde{I}_{\mu}(\tilde{\gamma}(0)),\tilde{I}_{\mu}(\tilde{\gamma}(1))\},
	\end{equation*}
	where $\zeta_{\mu}(c):=\inf\limits_{u\in\partial V_{\mu}(c)}I_{\mu}(u)>0$. This implies that $M_{\mu}(c)\leq\tilde{M}_{\mu}(c)$.
	Combining these two results, we conclude that $M_{\mu}(c)=\tilde{M}_{\mu}(c)$. 
\end{proof}

Following the Ghoussoub's minimax approach introduced in \cite{Ghoussoub1993}, we establish a result showing the existence of a Palais-Smale sequence of $\tilde{I}_{\mu}$ at level $\tilde{M}_{\mu}(c)$ with an additional property.

\begin{lemma}\label{exist for ps seq}
	For any $\mu\in(0,1]$ and $c\in(0,c_0)$, there exists a sequence $\{u_n\}\subset S_r(c)$ satisfying
	\begin{equation*}
		I_{\mu}(u_n)\to M_{\mu}(c)>0,~~~~~~~~~\|I'_{\mu}|_{S_r(c)}(u_n)\|_{X^*}\to0,~~~~~~~~~Q_{\mu}(u_n)\to0.
	\end{equation*}
\end{lemma}
\begin{proof} We will adopt the notation from \cite[Theorem 5.2]{Ghoussoub1993}. Define the following sets:
	\begin{equation*}
		\Lambda_{\mu}^-(c):=\{u\in\Lambda_r(c):I_{\mu}(u)<0\},~~~~~\mathcal{E}_{\mu}^c:=\{u\in S_r(c):I_{\mu}(u)<2m_{\mu}(c)\}.
	\end{equation*}
Let
	\begin{equation*}
		A=\tilde{\gamma}([0,1]),~~~\mathcal{X}= \R\times S_r(c),~~~\mathcal{F}=\{\tilde{\gamma}([0,1]):\tilde{\gamma}\in\tilde{\Gamma}_{\mu}(c)\},
	\end{equation*}
	\begin{equation*}
		B=(\{0\}\times \Lambda_{\mu}^-(c))\cup(\{0\}\times \mathcal{E}_{\mu}^c),~~~F=\{(s,u)\in\R\times S_r(c):\tilde{I}_{\mu}(s,u)\ge\tilde{M}_{\mu}(c)\}.
	\end{equation*}
It is straightforward to verify that $\mathcal{F}$ is a homotopy stable family of compact subsets of $\mathcal{X}$ with an extended closed boundary $B$.
	Moreover, using the definition of $\tilde{M}_{\mu}(c)$ and $\tilde{I}_{\mu}(s,u)$, we can conclude that $F$ is a dual set for $\mathcal{F}$ with $\phi=\tilde{I}_{\mu}$ and $\tilde{d}=\tilde{M}_{\mu}(c)$. Consequently, by \cite[Theorem 5.2]{Ghoussoub1993}, for the minimizing sequence $\{y_n=(0,\beta_n)\}\subset\tilde{\Gamma}_{\mu}(c)$ with $\beta_n\ge0$ a.e. in $\R^N$, there exists a Palais-Smale sequence $\{(s_n,w_n)\}\subset\R\times S_r(c)$ for $\tilde{I}_{\mu}|_{\R\times S_r(c)}$ at the level $\tilde{M}_{\mu}(c)>0$. That is, as $n\to+\infty$,
	\begin{equation}\label{exist for ps seq eq2}
		\partial_s\tilde{I}_{\mu}(s_n,w_n)\to0~~~~\text{and}~~~~\|\partial_u\tilde{I}_{\mu}(s_n,w_n)\|_{(T_{w_n}S_r(c))^*}\to0,
	\end{equation}
	with the additional property that
	\begin{equation}\label{exist for ps seq eq3}
		|s_n|+\dist_{X}(w_n,\beta_n([0,1]))\to0.
	\end{equation}
	By \eqref{exist for ps seq eq3}, we know that $\{s_n\}$ is bounded. Moreover, it follows from \eqref{exist for ps seq eq2} that $Q_{\mu}(\mathcal{H}(s_n,w_n))\to0$ and that
	\begin{align*}
		&\mu e^{\theta(1+\gamma_{\theta})s_n}\int_{\R^N}|\nabla w_n|^{\theta-2}\nabla w_n\nabla\phi dx+e^{2s_n}\int_{\R^N}\nabla w_n\nabla \phi dx+2e^{(N+2)s_n}\int_{\R^N}|\nabla w_n|^{2}w_n\phi dx\nonumber\\
		&+2e^{(N+2)s_n}\int_{\R^N}|w_n|^{2}\nabla w_n\nabla \phi dx-\tau\, e^{\frac{N(q-2)s_n}{2}}\int_{\R^N}|w_n|^{q-2}w_n\phi dx-e^{N(2^*-1)s_n}\int_{\R^N}|w_n|^{2\cdot2^*-2}w_n\phi dx\\
		&=o(1)\|\phi\|_{X},~~~~\forall\phi\in T_{w_n}S_r(c),
	\end{align*}
	which implies that as $n\to+\infty$,
	\begin{equation}\label{exist for ps seq eq4}
		\langle I'_{\mu}(\mathcal{H}(s_n,w_n)),\mathcal{H}(s_n,\phi)\rangle=o(1)\|\phi\|_{X}=o(1)\|\mathcal{H}(s_n,\phi)\|_{X},~~~~\forall\phi\in T_{w_n}S_r(c).
	\end{equation}
	Let $u_n:=\mathcal{H}(s_n,w_n)\in S_r(c)$. Then taking into account \eqref{exist for ps seq eq4} and using Lemma \ref{isomorphism thm}, we deduce that $\{u_n\}\subset S_r(c)$ is a Palais-Smale sequence for $I_{\mu}$ restricted on $S_r(c)$ at the level $M_{\mu}(c)$, with $Q_{\mu}(u_n)\to0$.
\end{proof}

At this stage, since we are dealing with a critical exponent regime, it is necessary to derive a refined energy estimate, which is crucial for obtaining the compactness of the Palais-Smale sequence as $\mu\to0^+$.
Following the approach of Brezis-Nirenberg \cite{BN1983}, we define a function $u_{\eps}$ by
\begin{equation*}
	u_{\eps}(x):=\frac{(N(N-2)\eps)^{\frac{N-2}{8}}}{(\eps+|x|^2)^{\frac{N-2}{4}}},~~~\eps>0.
\end{equation*}
Clearly, the function $v_{\eps}:=u^2_{\eps}$ solves the equation $-\Delta v_{\eps}=v_{\eps}^{\frac{N+2}{N-2}}$.
Let $\xi\in C_0^{\infty}(\R^N)$ be a radially non-increasing cut-off function such that $\xi\equiv1$ in $B_1$ and $\xi\equiv0$ in $\R^N\backslash B_2$. We then define
\begin{equation}\label{test function}
	U_{\eps}(x):=\xi(x)u_{\eps}(x).
\end{equation}
As $\epsilon\to0$, the function $U_{\eps}$ satisfies the following estimates:
\begin{align}\label{test fun esti1}
	&4\int_{\R^N}|U_{\eps}|^2|\nabla U_{\eps}|^2dx=\int_{\R^N}|\nabla(U_{\eps}^2)|^2dx=\mathcal{S}^{\frac{N}{2}}+\mathcal{O}(\eps^{\frac{N-2}{2}}).\\
&\int_{\R^N}|U_{\eps}|^{2\cdot2^*}dx=\mathcal{S}^{\frac{N}{2}}+\mathcal{O}(\eps^{\frac{N}{2}}).\label{test fun esti2}\\
&\int_{\R^N}|\nabla U_{\eps}|^2dx=\mathcal{O}(\eps^{\frac{N-2}{4}}|\ln\eps|).\label{test fun esti3}
\end{align}
\begin{equation}\label{test fun esti4}
	\int_{\R^N}|U_{\eps}|^{r}dx=
	\begin{cases}
		\mathcal{O}(\eps^{\frac{N}{2}-\frac{(N-2)r}{8}}),~~~2^*<r<2\cdot2^*;\\
		\mathcal{O}(\eps^{\frac{N}{4}}|\ln\eps|),~~~r=2^*;\\
		\mathcal{O}(\eps^{\frac{(N-2)r}{8}}),~~~1\leq r<2^*.
	\end{cases}
\end{equation}
In particular, we have
\begin{equation}\label{test fun esti5}
	\int_{\R^N}|U_{\eps}|^{2}dx=\mathcal{O}(\eps^{\frac{N-2}{4}}).
\end{equation}
In the remainder of this subsection, we will prove the following lemma.
\begin{lemma}\label{energy estimate}
	Let $q\in(2,2+\frac{4}{N}), \tau>0$ and $c\in(0,c_0)$. Then there exist $0<\mu_0<1$ and $\delta_0>0$ such that
	\begin{equation*}
		M_{\mu}(c)\leq m_{\mu}(c)+\frac{\mathcal{S}^{\frac{N}{2}}}{2N}-\delta',~~~\text{for all}~~~\mu\in(0,\mu_0).
	\end{equation*}
\end{lemma}
\begin{proof}
	Let $u_{\mu,c}$ be a local minimizer of $I_{\mu}$ on $V_{\mu}(c)$. Then, by Lemma \ref{strong convergence}, we have
	\begin{equation}\label{energy estimate eq8}
		\|u_{\mu,c}\|_2^2=c,~~~I_{\mu}(u_{\mu,c})=m_{\mu}(c),~~~u_{\mu,c}(x)>0,~~~\forall~x\in\R^N
	\end{equation}
	and
	\begin{equation}\label{energy estimate eq3}
		\langle I'_{\mu}(u_{\mu,c})+\lambda_c u_{\mu,c},\phi\rangle=0, ~~~\forall \phi\in X.
	\end{equation}
	Using $\phi=u_{\mu,c}$ as a test function in \eqref{energy estimate eq3}, we obtain
	\begin{align}\label{energy estimate eq1}
		&\mu\int_{\R^N}|\nabla u_{\mu,c}|^{\theta}dx+\int_{\R^N}|\nabla u_{\mu,c}|^2dx+4\int_{\R^N}|u_{\mu,c}|^2|\nabla u_{\mu,c}|^2dx+\lambda_c\int_{\R^N}|u_{\mu,c}|^2dx\nonumber\\
		&-\tau\int_{\R^N}|u_{\mu,c}|^qdx-\int_{\R^N}|u_{\mu,c}|^{2\cdot2^*}dx=0.
	\end{align}
	Combining \eqref{energy estimate eq1} and $Q_{\mu}(u_{\mu,c})=0$, it follows from $\theta>\frac{4N}{N+2}$ that
	\begin{equation}\label{energy estimate eq2}
		\frac{\tau(4N+2q-Nq)}{q(N+2)}\|u_{\mu,c}\|_q^q\leq\lambda_c\|u_{\mu,c}\|_2^2+\frac{N-2}{N+2}\|\nabla u_{\mu,c}\|_2^2.
	\end{equation}
	Moreover, using $\phi=U_\eps$ as a test function in \eqref{energy estimate eq3}, we have
	\begin{align}\label{energy estimate eq4}
		&\mu\int_{\R^N}|\nabla u_{\mu,c}|^{\theta-2}\nabla u_{\mu,c}\nabla U_\eps dx+\int_{\R^N}\nabla u_{\mu,c}\nabla U_\eps dx+2\int_{\R^N}|u_{\mu,c}|^2\nabla u_{\mu,c}\nabla U_\eps dx+2\int_{\R^N}|\nabla u_{\mu,c}|^2u_{\mu,c}U_\eps dx\nonumber\\
		&+\lambda_c\int_{\R^N}u_{\mu,c}U_\eps dx-\tau\int_{\R^N}u_{\mu,c}^{q-1}U_\eps dx-\int_{\R^N}u_{\mu,c}^{2\cdot2^*-1}U_\eps dx=0.
	 \end{align}
	For any $\eps>0$ and $t>0$, we define
	 \begin{equation*}
		\widetilde{W}_{\eps,t}:=u_{\mu,c}+tU_\eps,~~~~W_{\eps,t}(x):=\eta^{\frac{N-2}{4}}\widetilde{W}_{\eps,t}(\eta x),
	 \end{equation*}
	 where $\eta=\left(\frac{\|\widetilde{W}_{\eps,t}\|_2^2}{c}\right)^{\frac{2}{N+2}}$. Then we immediately get that
	 \begin{equation}\label{energy estimate eq5}
		\|\nabla W_{\eps,t}\|_2^2=\eta^{\frac{2-N}{2}}\|\nabla \widetilde{W}_{\eps,t}\|_2^2,~~~\|\nabla W^2_{\eps,t}\|_2^2=\|\nabla \widetilde{W}^2_{\eps,t}\|_2^2,~~~\|W_{\eps,t}\|_q^q=\eta^{\frac{(N-2)q-4N}{4}}\|\widetilde{W}_{\eps,t}\|_q^q,
	 \end{equation}
	 \begin{equation}\label{energy estimate eq6}
		\|W_{\eps,t}\|_{2\cdot2^*}^{2\cdot2^*}=\|\widetilde{W}_{\eps,t}\|_{2\cdot2^*}^{2\cdot2^*},~~~\|W_{\eps,t}\|_2^2=\eta^{-\frac{N+2}{2}}\|\widetilde{W}_{\eps,t}\|_2^2=c.
	 \end{equation}
	 In view of \eqref{energy estimate eq5}-\eqref{energy estimate eq6}, we have
	 \begin{align*}
		I_0(W_{\eps,t})
		&\leq I_0(u_{\mu,c})+\frac{1}{2}\left(\eta^{\frac{2-N}{2}}-1\right)\int_{\R^N}|\nabla u_{\mu,c}|^2dx+t\eta^{\frac{2-N}{2}}\int_{\R^3}\nabla u_{\mu,c}\nabla U_\eps dx+\frac{t^2}{2}\eta^{\frac{2-N}{2}}\int_{\R^N}|\nabla U_\eps|^2dx\nonumber\\
		&\quad+\frac{t^4}{4}\int_{\R^N}|\nabla U_{\eps}^2|^2dx+t^2\int_{\R^N}|\nabla (u_{\mu,c}U_\eps)|^2dx+\frac{t^2}{2}\int_{\R^N}\nabla u_{\mu,c}^2\nabla U_\eps^2dx\nonumber\\
		&\quad+t\int_{\R^N}\nabla u_{\mu,c}^2\nabla(u_{\mu,c}U_\eps)dx+t^3\int_{\R^N}\nabla U_{\eps}^2\nabla(u_{\mu,c}U_\eps)dx\nonumber\\
		&\quad+\frac{\tau}{q}\left(1-\eta^{\frac{(N-2)q-4N}{4}}\right)\int_{\R^N}u_{\mu,c}^qdx-\frac{\tau}{q}\eta^{\frac{(N-2)q-4N}{4}}t^q\int_{\R^N}U_\eps^qdx-\frac{1}{2\cdot2^*}t^{2\cdot2^*}\int_{\R^N}U_\eps^{2\cdot2^*}dx\nonumber\\
		&:=L_1(t).
	 \end{align*}
	 It is not difficult to verify that, uniformly for small $\eps>0$, $L_1(t)\to-\infty$ as $t\to+\infty$ and $L_1(t)\to I_0(u_{\mu,c})$ as $t\to0$ due to $\eta\to1$. Hence, there exist $\eps_0>0$ and $0<t_1<t_2<+\infty$ such that for any $t\in(0,t_1)$ or $t\in(t_2,+\infty)$ and $\eps\in(0,\eps_0]$, we have
	 \begin{equation*}
		I_0(W_{\eps,t})<I_0(u_{\mu,c})+\frac{\mathcal{S}^{\frac{N}{2}}}{2N}.
	 \end{equation*}

	 Next, we address the case $t_1\leq t\leq t_2$.
	 By \eqref{test fun esti5}, we have
	 \begin{align*}
		\eta^{\frac{N+2}{2}}=\frac{\|u_{\mu,c}+tU_\eps\|_2^2}{c}=1+\frac{2t}{c}\int_{\R^N}u_{\mu,c}U_\eps dx+t^2\mathcal{O}\left(\eps^{\frac{N-2}{4}}\right)~~~\mbox{as}~~\eps\to0.
	 \end{align*}
	 Utilizing the Taylor expansion, we deduce that, as $\eps\to0$,
	 \begin{align}\label{energy estimate eq7}
		&1-\eta^{\frac{(N-2)q-4N}{4}}=1-\eta^{\frac{N+2}{2}\frac{(N-2)q-4N}{2(N+2)}}=\frac{t}{c}\frac{4N-(N-2)q}{N+2}\int_{\R^N}u_{\mu,c}U_\eps dx+t^2\mathcal{O}\left(\eps^{\frac{N-2}{4}}\right).\\
	&\eta^{\frac{2-N}{2}}-1=\eta^{\frac{N+2}{2}\frac{2-N}{N+2}}-1=\frac{t}{c}\frac{2(2-N)}{N+2}\int_{\R^N}u_{\mu,c}U_\eps dx-t^2\mathcal{O}\left(\eps^{\frac{N-2}{4}}\right).\label{energy estimate eq9}
	 \end{align}
	 Moreover, since the bounded function $u_{\mu,c}$ is a radial critical point of $I_\mu$, by \eqref{test fun esti3}, \eqref{test fun esti4}, \eqref{test fun esti5} and the regularity of $u_{\mu,c}$ (see Remark \ref{regular rmk}), as $\eps\to0$, we have the following estimates:
	 \begin{eqnarray}\label{energy estimate eq10}
		&&\int_{\R^N}u_{\mu,c}^2|\nabla U_\eps|^2dx=\mathcal{O}\left(\eps^{\frac{N-2}{4}}|\ln \eps|\right).\\
	&&\int_{\R^N}|\nabla u_{\mu,c}|^2U_\eps^2dx=\mathcal{O}\left(\eps^{\frac{N-2}{4}}\right).\label{energy estimate eq11}
	 \end{eqnarray}
	 \begin{align}\label{energy estimate eq12}
		\int_{\R^N}\nabla u^2_{\mu,c}\nabla U_\eps^2dx
=\mathcal{O}\left(\eps^{\frac{N-2}{4}}\sqrt{|\ln \eps|}\right).
	 \end{align}
	 \begin{align}\label{energy estimate eq13}
		\int_{\R^N}U_\eps^2\nabla u_{\mu,c}\nabla U_\eps dx&=CB_N^3\omega_N\eps^{\frac{3(N-2)}{8}}\left[\int_0^2\frac{3\xi^2(r)\xi'(r)r^{N-1}}{(\eps+r^2)^{\frac{3(N-2)}{4}}}dr+\left|\frac{3(2-N)}{2}\right|\int_0^2\frac{\xi^3(r)r^{N}}{(\eps+r^2)^{\frac{3N-2}{4}}}dr\right]\nonumber\\
		&:=\zeta_\eps(N)=
		\begin{cases}
			\mathcal{O}(\eps^{\frac{N+2}{8}})~~~&\text{if}~~~N>4;\\
			\mathcal{O}(\eps^{\frac{3}{4}}|\ln\eps|)~~~&\text{if}~~~N=4;\\
			\mathcal{O}(\eps^{\frac{3}{8}})~~~&\text{if}~~~N=3.
		\end{cases}
	 \end{align}
	 \begin{align}\label{energy estimate eq14}
		\int_{\R^N}|\nabla U_\eps|^2 u_{\mu,c}U_\eps dx
&\leq CB_N^3\omega_N\eps^{\frac{3(N-2)}{8}}\left(\mathcal{O}(1)+\eps^{\frac{2-N}{4}}\int_0^{\frac{2}{\sqrt{\eps}}}\frac{s^{N+1}}{(1+s^2)^{\frac{3N+2}{4}}}ds\right)\nonumber\\
		&=\mathcal{O}\left(\eps^{\frac{N-2}{8}}\right).
	 \end{align}
	 \begin{align}\label{energy estimate eq15}
		\int_{\R^N}|\nabla u_{\mu,c}|^{\theta-2}\nabla u_{\mu,c}\nabla U_\eps dx
&\le CB_N^{\theta}\omega_N\eps^{\frac{(N-2)\theta}{8\theta}}\left(\mathcal{O}(1)+\left(\eps^{\frac{N(2-\theta)}{4}}\int_0^{\frac{2}{\sqrt{\eps}}}\frac{s^{\theta+N-1}}{(1+s^2)^{\frac{(N+2)\theta}{4}}}ds\right)^{\frac{1}{\theta}}\right)\nonumber\\
		&=\mathcal{O}\left(\eps^{\frac{4N-\theta(N+2)}{8\theta}}\right).
	 \end{align}
	 \begin{align}
		\int_{\R^N}u_{\mu,c}U_\eps^{2\cdot2^*-1} dx=\mathcal{O}\left(\eps^{\frac{N-2}{8}}\right).\label{energy estimate eq16}
	 \end{align}
	From \eqref{test fun esti1}-\eqref{test fun esti5}, \eqref{energy estimate eq8} and \eqref{energy estimate eq2}-\eqref{energy estimate eq16}, we derive that
	 \begin{align}\label{energy estimate eq17}
		I_0(W_{\eps,t})
		&\leq\frac{1}{2}\int_{\R^N}|\nabla u_{\mu,c}|^2dx+\frac{1}{4}\int_{\R^N}|\nabla u^2_{\mu,c}|^2dx-\frac{\tau}{q}\int_{\R^N}u_{\mu,c}^qdx-\frac{1}{2\cdot2^*}\int_{\R^N}u_{\mu,c}^{2\cdot2^*}dx\nonumber\\
		&\quad+t^2\int_{\R^N}u_{\mu,c}^2|\nabla U_\eps|^2dx+t^2\int_{\R^N}|\nabla u_{\mu,c}|^2U_\eps^2dx+t^2\int_{\R^N}\nabla u^2_{\mu,c}\nabla U_\eps^2dx+2t^3\int_{\R^N}U_\eps^2\nabla u_{\mu,c}\nabla U_\eps dx\nonumber\\
		&\quad+2t^3\int_{\R^N}|\nabla U_\eps|^2 u_{\mu,c}U_\eps dx-\lambda_ct\int_{\R^N}u_{\mu,c}U_\eps dx-\mu t\int_{\R^N}|\nabla u_{\mu,c}|^{\theta-2}\nabla u_{\mu,c}\nabla U_\eps dx\nonumber\\
		&\quad+\frac{1}{2}\left(\eta^{\frac{2-N}{2}}-1\right)\int_{\R^N}|\nabla u_{\mu,c}|^2dx+\left(\eta^{\frac{2-N}{2}}-1\right)t\int_{\R^3}\nabla u_{\mu,c}\nabla U_\eps dx\nonumber\\
		&\quad+\frac{\tau}{q}\left(1-\eta^{\frac{(N-2)q-4N}{4}}\right)\int_{\R^N}u_{\mu,c}^qdx+\tau\left(1-\eta^{\frac{(N-2)q-4N}{4}}\right)t\int_{\R^N}u_{\mu,c}^{q-1}U_\eps dx\nonumber\\
		&\quad+\frac{t^2}{2}\eta^{\frac{2-N}{2}}\int_{\R^N}|\nabla U_\eps|^2dx+\frac{t^4}{4}\int_{\R^N}|\nabla U^2_\eps|^2dx-\frac{t^{2\cdot2^*}}{2\cdot2^*}\int_{\R^N}U_\eps^{2\cdot2^*}dx-t^{2\cdot2^*-1}\int_{\R^N}u_{\mu,c}U_\eps^{2\cdot2^*-1} dx\nonumber\\
		&\leq I_0(u_{\mu,c})+t^2\mathcal{O}\left(\eps^{\frac{N-2}{4}}|\ln \eps|\right)+t^2\mathcal{O}\left(\eps^{\frac{N-2}{4}}\right)+t^2\mathcal{O}\left(\eps^{\frac{N-2}{4}}\sqrt{|\ln \eps|}\right)+2t^3\zeta_\eps(N)+2t^3\mathcal{O}\left(\eps^{\frac{N-2}{8}}\right)\nonumber\\
		&\quad-\mu t\mathcal{O}\left(\eps^{\frac{4N-\theta(N+2)}{8\theta}}\right)+t^2\mathcal{O}\left(\eps^{\frac{N-2}{4}}\right)-t^2\mathcal{O}\left(\eps^{\frac{N-2}{4}}\sqrt{|\ln \eps|}\right)\nonumber\\
		&\quad+t^2\mathcal{O}\left(\eps^{\frac{N-2}{4}}|\ln \eps|\right)+\frac{t^4}{4}\mathcal{S}^{\frac{N}{2}}+t^4\mathcal{O}\left(\eps^{\frac{N-2}{2}}\right)-\frac{t^{2\cdot2^*}}{2\cdot2^*}\mathcal{S}^{\frac{N}{2}}-t^{2\cdot2^*}\mathcal{O}\left(\eps^{\frac{N}{2}}\right)-t^{2\cdot2^*-1}\mathcal{O}\left(\eps^{\frac{N-2}{8}}\right).
	 \end{align}
It is straightforward to observe that
	 \begin{equation*}
		\frac{t^4}{4}-\frac{t^{2\cdot2^*}}{2\cdot2^*}\leq\frac{1}{2N},~~~\forall t>0.
	 \end{equation*}
	 Moreover, by choosing $\mu=\eps^{\frac{\alpha}{\theta}}$ for some $\alpha>0$ satisfying $\alpha\in\left(\frac{\theta(N+2)-4N}{8},\frac{\theta(N+2)-3N-2}{8}\right)$, it follows from \eqref{energy estimate eq17} that there exist small $\eps_0>0$ and $\delta'>0$ such that
	 \begin{equation*}\label{energy estimate eq18}
		\sup_{t>0}I_0(W_{\eps_0,t})\leq I_0(u_{\mu,c})+\frac{\mathcal{S}^{\frac{N}{2}}}{2N}-2\delta'.
	 \end{equation*}
	 Clearly, $W_{\eps_0,t}\in S_r(c)$ for all $t\ge0$, $W_{\eps_0,0}=u_{\mu,c}$ and $I(W_{\eps_0,t})\leq I_{\mu}(W_{\eps_0,t})<2m_{\mu}(c)$ for sufficiently large $t>0$. Therefore, there exists a sufficiently large $t_0>0$ such that
	 \begin{equation*}
		I_{\mu}(W_{\eps_0,t_0})<2m_{\mu}(c).
	 \end{equation*}
	 Let $\gamma(t):=W_{\eps_0,tt_0}$. Then $\gamma(t)\in\Gamma_\mu(c)$. Moreover, observe that there exist $0<T_1<T_2$ such that
	\begin{equation*}
		\sup_{t>0}I_{\mu}(W_{\eps,t})=\sup_{t\in[T_1,T_2]}I_{\mu}(W_{\eps,t}),~~~~~~\sup_{t>0}I(W_{\eps,t})=\sup_{t\in[T_1,T_2]}I(W_{\eps,t}).
	\end{equation*}
From this, we deduce that
	\begin{align*}
		M_{\mu}(c)\leq\sup_{t>0}I_{\mu}(W_{\eps_0,t})&\leq C(\eps_0)\mu+I_0(u_{\mu,c})+\frac{\mathcal{S}^{\frac{N}{2}}}{2N}-2\delta'\leq I_{\mu}(u_{\mu,c})+\frac{\mathcal{S}^{\frac{N}{2}}}{2N}-\delta'=m_{\mu}(c)+\frac{\mathcal{S}^{\frac{N}{2}}}{2N}-\delta'.
	\end{align*}
	This completes the proof.
\end{proof}

\section{The case $2+\frac{4}{N}\leq q< 4+\frac{4}{N}$}\label{Mountain pass solution subcritical1}\setcounter{equation}{0}
\subsection{Mountain pass structure}\label{Mountain pass solution1}
In this section, we investigate the case $2+\frac{4}{N}\leq q< 4+\frac{4}{N}$. We first show that the functional $I_{\mu}$ exhibits the mountain pass geometry.

\begin{lemma}\label{mp struc1}
	Let $\mu\in(0,1]$ and $\tau>0$. Assume that one of the following conditions holds:
	\begin{enumerate}[label=(\roman*)]
		\item $N\ge3, ~q=2+\frac{4}{N}$, $c\in(0,\bar{c}_1)$.
		\item $N=3, ~2+\frac{4}{N}<q<4+\frac{4}{N}, c>0$.
		\item $N\ge4, ~2+\frac{4}{N}<q\leq2^*, c>0$.
		\item $N\ge4, ~2^*<q<4+\frac{4}{N}, c>0$.
	\end{enumerate}
	Then there exists $\mathcal{K}_{\mu}=\mathcal{K}_{\mu}(c,\tau,q,N)>0$ small enough such that
	\begin{equation}\label{mp struc1 eq1}
		0<\sup_{u\in \mathcal{A}_{\mu}(\mathcal{K}_{\mu},c)}I_{\mu}(u)<\inf_{u\in\partial \mathcal{A}_{\mu}(2\mathcal{K}_{\mu},c)}I_{\mu}(u),
	\end{equation}
	where
	\begin{equation*}
		\mathcal{A}_{\mu}(\mathcal{K}_{\mu},c):=\{u\in S(c):\frac{\mu}{\theta}\|\nabla u\|_{\theta}^{\theta}+\frac{1}{2}\|\nabla u\|_2^2+\int_{\R^N}|u|^2|\nabla u|^2dx<\mathcal{K}_{\mu}\},
	\end{equation*}
	\begin{equation*}
		\partial \mathcal{A}_{\mu}(2\mathcal{K}_{\mu},c):=\{u\in S(c):\frac{\mu}{\theta}\|\nabla u\|_{\theta}^{\theta}+\frac{1}{2}\|\nabla u\|_2^2+\int_{\R^N}|u|^2|\nabla u|^2dx=2\mathcal{K}_{\mu}\}.
	\end{equation*}
\end{lemma}
\begin{proof}
	For any $u\in \mathcal{A}_{\mu}(\mathcal{K}_{\mu},c), v\in\partial \mathcal{A}_{\mu}(2\mathcal{K}_{\mu},c)$, we have
	\begin{equation*}
		\frac{\mu}{\theta}\|\nabla u\|_{\theta}^{\theta}+\frac{1}{2}\|\nabla u\|_{2}^{2}+\int_{\R^N}|u|^2|\nabla u|^{2}dx<\mathcal{K}_{\mu},~~~\frac{\mu}{\theta}\|\nabla v\|_{\theta}^{\theta}+\frac{1}{2}\|\nabla v\|_{2}^{2}+\int_{\R^N}|v|^2|\nabla v|^{2}dx=2\mathcal{K}_{\mu}.
	\end{equation*}
	\textbf{Case 1:} $N=3, ~2+\frac{4}{N}\leq q<4+\frac{4}{N}$ or $N\ge4, ~2+\frac{4}{N}\leq q\leq2^*$.
	Using \eqref{gn} and the Sobolev inequality, we get that
	\begin{align*}
		&I_{\mu}(v)-I_{\mu}(u)\ge \mathcal{K}_{\mu}-\frac{\tau}{q}\int_{\R^N}|v|^{q}dx-\frac{1}{2\cdot2^*}\int_{\R^N}|v|^{2\cdot2^*}dx\\
		&\ge \mathcal{K}_{\mu}-\frac{\tau \mathcal{C}_1^q(q,N)c^{\frac{2q-N(q-2)}{4}}}{q}\left(\int_{\R^N}|\nabla v|^{2}dx\right)^{\frac{N(q-2)}{4}}-\frac{1}{2\cdot2^*}\left(\frac{4}{\mathcal{S}}\right)^{\frac{2^*}{2}}\left(\int_{\R^N}|v|^2|\nabla v|^{2}dx\right)^{\frac{N}{N-2}}\\
		&\ge \mathcal{K}_{\mu}-\frac{2^{\frac{N(q-2)}{2}}\tau \mathcal{C}_1^q(q,N)c^{\frac{2q-N(q-2)}{4}}}{q}\mathcal{K}_{\mu}^{\frac{N(q-2)}{4}}-\frac{2^{\frac{2}{N-2}}}{2^*}\left(\frac{4}{\mathcal{S}}\right)^{\frac{2^*}{2}}\mathcal{K}_{\mu}^{\frac{N}{N-2}}.
	\end{align*}
	It follows from $0<c<\bar{c}_1$ when $q=2+\frac{4}{N}$ and $\frac{N(q-2)}{4}>1$ when $q>2+\frac{4}{N}$ that there exists a sufficiently small $\mathcal{K}_{\mu}>0$ such that $I_{\mu}(v)-I_{\mu}(u)>0$, which gives \eqref{mp struc1 eq1}.

	\noindent\textbf{Case 2:} $N\ge4, ~2^*<q<4+\frac{4}{N}$. Using the interpolation inequality and the Sobolev inequality, we have
	\begin{align*}
		&I_{\mu}(v)-I_{\mu}(u)\ge \mathcal{K}_{\mu}-\frac{\tau}{q}\left(\int_{\R^N}|v|^{2^*}dx\right)^{\frac{2\cdot2^*-q}{2^*}}\left(\int_{\R^N}|v|^{2\cdot2^*}dx\right)^{\frac{q-2^*}{2^*}}-\frac{1}{2\cdot2^*}\left(\frac{4}{\mathcal{S}}\right)^{\frac{2^*}{2}}\left(\int_{\R^N}|v|^2|\nabla v|^{2}dx\right)^{\frac{N}{N-2}}\\
		&\ge \mathcal{K}_{\mu}-\frac{2^{q-2^*}\tau}{\mathcal{S}^{\frac{2^*}{2}}q}\left(\int_{\R^N}|\nabla v|^{2}dx\right)^{\frac{2\cdot2^*-q}{2}}\left(\int_{\R^N}|v|^2|\nabla v|^{2}dx\right)^{\frac{q-2^*}{2}}-\frac{1}{2\cdot2^*}\left(\frac{4}{\mathcal{S}}\right)^{\frac{2^*}{2}}\left(\int_{\R^N}|v|^2|\nabla v|^{2}dx\right)^{\frac{N}{N-2}}\\
		&\ge \mathcal{K}_{\mu}-\frac{2^{\frac{2^*+q}{2}}\tau}{\mathcal{S}^{\frac{2^*}{2}}q}\mathcal{K}_{\mu}^{\frac{N}{N-2}}-\frac{2^{\frac{2}{N-2}}}{2^*}\left(\frac{4}{\mathcal{S}}\right)^{\frac{2^*}{2}}\mathcal{K}_{\mu}^{\frac{N}{N-2}}.
	\end{align*}
This establishes equation \eqref{mp struc1 eq1}.
\end{proof}

\begin{lemma}\label{mp struc2}
	Under the assumptions of Lemma \ref{mp struc1}, the following properties hold:
	\begin{equation*}
		\check{\Gamma}_{\mu}(c):=\{\gamma\in \mathcal{C}([0,1], S_r(c)): \gamma(0)\in\mathcal{A}_{\mu}(\mathcal{K}_{\mu},c), I_{\mu}(\gamma(1))<0\}\neq\varnothing
	\end{equation*}
	and
	\begin{equation*}
		\check{M}_{\mu}(c):=\inf\limits_{\gamma\in\check{\Gamma}_{\mu}(c)}\max\limits_{t\in[0,1]}I_{\mu}(\gamma(t))\ge\check{\zeta}_{0}(c)>\sup\limits_{\gamma\in\check{\Gamma}_{\mu}(c)}\max\{I_{\mu}(\gamma(0)), I_{\mu}(\gamma(1))\},
	\end{equation*}
	where $\check{\zeta}_0(c)>0$ is independent of $\mu>0$.
\end{lemma}
\begin{proof}
The conclusion can be readily derived by employing Lemma \ref{mp struc1} and invoking classical arguments. Therefore, we omit the details here.
\end{proof}

\begin{lemma}\label{ps existence sec2}
	Under the assumptions of Lemma \ref{mp struc1}, there exists a sequence $\{u_n\}\subset S_r(c)$ satisfying
	\begin{equation*}
		I_{\mu}(u_n)\to \check{M}_{\mu}(c)>0,~~~~~~~~~\|I'_{\mu}|_{S_r(c)}(u_n)\|_{X^*}\to0,~~~~~~~~~Q_{\mu}(u_n)\to0.
	\end{equation*}
\end{lemma}
\begin{proof}
	Let $\tilde{I}_{\mu}$ be defined by \eqref{fibermap1}. As argued before, we set
	\begin{equation*}
		\overline{\Gamma}_{\mu}(c):=\{\overline{\gamma}\in \mathcal{C}([0,1], \R\times S_r(c)): \overline{\gamma}(0)=(0,\overline{\gamma}_1(0)), \overline{\gamma}_1(0)\in\mathcal{A}_{\mu}(\mathcal{K}_{\mu},c), \tilde{I}_{\mu}(\overline{\gamma}(1))<0\}
	\end{equation*}
	and 
	\begin{equation*}
		\overline{M}_{\mu}(c):=\inf\limits_{\overline{\gamma}\in\overline{\Gamma}_{\mu}(c)}\max\limits_{t\in[0,1]}\tilde{I}_{\mu}(\overline{\gamma}(t)).
	\end{equation*}
	Similar to Lemma \ref{eq of the mpv}, we can see that $\check{M}_{\mu}(c)=\overline{M}_{\mu}(c)$.
	Moreover, define
	\begin{equation*}
		\mathcal{E}_{\mu}^0:=\{u\in S_r(c):I_{\mu}(u)<0\}.
	\end{equation*}
	Let
	\begin{equation*}
		A=\overline{\gamma}([0,1]),~~~\mathcal{X}= \R\times S_r(c),~~~\mathcal{F}=\{\overline{\gamma}([0,1]):\overline{\gamma}\in\overline{\Gamma}_{\mu}(c)\},
	\end{equation*}
	\begin{equation*}
		B=(\{0\}\times \mathcal{A}_{\mu}(\mathcal{K}_{\mu},c))\cup(\{0\}\times \mathcal{E}_{\mu}^0),~~~F=\{(s,u)\in\R\times S_r(c):\tilde{I}_{\mu}(s,u)\ge\overline{M}_{\mu}(c)\}.
	\end{equation*}
By employing the same argument as in Lemma \ref{exist for ps seq}, we complete the proof.
\end{proof}

We now present a precise estimate of the energy level $\check{M}_{\mu}(c)$ as defined in Lemma \ref{mp struc2}. In order to achieve the desired result, we need to modify the test function  $U_\eps$ given by \eqref{test function} appropriately.
Define the radial symmetry nonnegative functions as follows:
\begin{equation}\label{test functions2}
	\widehat{U}_\eps(x):=B_N
	\begin{cases}
		\left(\frac{\sqrt{\eps}}{\eps+|x|^2}\right)^{\frac{N-2}{4}}~~~&\text{if}~~~0\leq |x|< \eps^{-\alpha},\\
		\left(\frac{\eps^{2\alpha+1/2}}{1+\eps^{2\alpha+1}}\right)^{\frac{N-2}{4}}\frac{1-\eps^{\beta}|x|}{1-\eps^{\beta-\alpha}}~~~&\text{if}~~~\eps^{-\alpha}\leq |x|<\eps^{-\beta},\\
		0~~~&\text{if}~~~|x|\ge \eps^{-\beta},
	\end{cases}
\end{equation}
where $B_N:=(N(N-2))^{\frac{N-2}{8}}$ and $\alpha>0, \beta>0$ are parameters to be determined later. Then we derive
\begin{align}\label{energy2 est2}
	&\|\widehat{U}_\eps\|_2^2=\int_{\R^N}|\widehat{U}_\eps|^{2}dx=\omega_N\int_0^{+\infty}|\widehat{U}_\eps(r)|^{2}r^{N-1}dr\nonumber\\
	&=\omega_NB_N^2\bigg[\int_0^{\eps^{-\alpha}}\frac{\eps^{\frac{N-2}{4}}r^{N-1}}{(\eps+r^2)^{\frac{N-2}{2}}}dr+\int_{\eps^{-\alpha}}^{\eps^{-\beta}}\left(\frac{\eps^{2\alpha+1/2}}{1+\eps^{2\alpha+1}}\right)^{\frac{N-2}{2}}\frac{(1-\eps^{\beta}r)^2r^{N-1}}{(1-\eps^{\beta-\alpha})^2}dr\bigg]\nonumber\\
	&=\omega_NB_N^2\bigg[\eps^{\frac{N+2}{4}}\int_0^{\eps^{-\alpha-\frac{1}{2}}}\frac{s^{N-1}}{(1+s^2)^{\frac{N-2}{2}}}ds\nonumber\\
	&\quad+\left(\frac{\eps^{2\alpha+1/2}}{1+\eps^{2\alpha+1}}\right)^{\frac{N-2}{2}}\frac{2\eps^{-\beta(N+2)}-\left[N(N+1)\eps^{-2\alpha}-2N(N+2)\eps^{-\beta-\alpha}+(N+1)(N+2)\eps^{-2\beta}\right]\eps^{-N\alpha}}{N(N+1)(N+2)(\eps^{-\beta}-\eps^{-\alpha})^2}\bigg]\nonumber\\
	&=\omega_NB_N^2\bigg[\eps^{\frac{N+2}{4}}\tilde{I}(\eps,\alpha)\nonumber\\
&~~~~~+\frac{2\eps^{-\beta(N+2)-\frac{N-2}{4}}-\left[N(N+1)\eps^{-2\alpha}-2N(N+2)\eps^{-\beta-\alpha}+(N+1)(N+2)\eps^{-2\beta}\right]\eps^{-N\alpha-\frac{N-2}{4}}}{N(N+1)(N+2)(\eps^{-\beta}-\eps^{-\alpha})^2\left(1+\eps^{-2\alpha-1}\right)^{\frac{N-2}{2}}}\bigg],
\end{align}
where
\begin{equation*}
	\tilde{I}(\eps,\alpha):=\int_0^{\eps^{-\alpha-\frac{1}{2}}}\frac{s^{N-1}}{(1+s^2)^{\frac{N-2}{2}}}ds=\mathcal{O}(\eps^{-2\alpha-1})~~~~\mbox{as}~~\eps\to0.
\end{equation*}
Choose $\max\{\frac{4N^2-4N+8-N(N-2)q}{16(N-2)},0\}<\alpha<\frac{N-2}{8}$ and $\beta>\alpha$ such that $c=\lim\limits_{\eps\to0}\|\widehat{U}_\eps\|_2^2$.
 From this, we derive that $\eps^{\beta}=\mathcal{O}(\eps^{\frac{(4\alpha+1)(N-2)}{4N}})$, and more precisely,
\begin{equation}\label{energy mass id1}
	\lim_{\eps\to0}\frac{\eps^{\frac{(4\alpha+1)(N-2)}{4N}}}{\eps^{\beta}}=\left(\frac{N(N+1)(N+2)c}{2\omega_NB_N^2}\right)^{\frac{1}{N}}.
\end{equation}
Moreover, we have
\begin{align}\label{energy2 est1}
	\|\nabla \widehat{U}_\eps\|_{2}^{2}&=\int_{\R^N}|\nabla \widehat{U}_\eps|^{2}dx=\omega_N\int_0^{+\infty}|\widehat{U}_\eps'(r)|^{2}r^{N-1}dr\nonumber\\
	&=\omega_NB_N^{2}\bigg[\left(\frac{N-2}{2}\right)^{2}\eps^{\frac{N-2}{4}}\int_0^{\eps^{-\alpha-\frac{1}{2}}}\frac{s^{N+1}}{(1+s^2)^{\frac{N+2}{2}}}ds+\left(\frac{\eps^{2\alpha+1/2}}{1+\eps^{2\alpha+1}}\right)^{\frac{N-2}{2}}\frac{\eps^{2\beta-N\beta}\left(1-\eps^{N\beta-N\alpha}\right)}{N\left(1-\eps^{\beta-\alpha}\right)^{2}}\bigg]\nonumber\\
	&=\mathcal{O}\left(\eps^{\frac{(4\alpha+1)(N-2)}{2N}}\right).
\end{align}
\begin{align}\label{energy2 est3}
	\|\widehat{U}_\eps\|_{2\cdot2^*}^{2\cdot2^*}&=\int_{\R^N}|\widehat{U}_\eps|^{2\cdot2^*}dx=\omega_N\int_0^{+\infty}|\widehat{U}_\eps(r)|^{2\cdot2^*}r^{N-1}dr\nonumber\\
	&=\omega_NB_N^{2\cdot2^*}\bigg[\int_0^{\eps^{-\alpha-\frac{1}{2}}}\frac{s^{N-1}}{(1+s^2)^N}ds+\left(\frac{\eps^{2\alpha+1/2}}{1+\eps^{2\alpha+1}}\right)^N\int_{0}^{1-\eps^{\beta-\alpha}}\frac{\eps^{-\beta N}s^{\frac{4N}{N-2}}(1-s)^{N-1}}{(1-\eps^{\beta-\alpha})^{\frac{4N}{N-2}}}ds\bigg]\nonumber\\
	&=\mathcal{S}^{\frac{N}{2}}+\omega_NB_N^{2\cdot2^*}\bigg[-\int_{\eps^{-\alpha-\frac{1}{2}}}^{+\infty}\frac{s^{N-1}}{(1+s^2)^N}ds+\eps^{-\beta N}\left(\frac{\eps^{2\alpha+1/2}}{1+\eps^{2\alpha+1}}\right)^N\int_{0}^{1-\eps^{\beta-\alpha}}\frac{s^{\frac{4N}{N-2}}(1-s)^{N-1}}{(1-\eps^{\beta-\alpha})^{\frac{4N}{N-2}}}ds\bigg]\nonumber\\
	&=\mathcal{S}^{\frac{N}{2}}+\mathcal{O}\bigg(\eps^{\frac{(4\alpha+1)(N+2)}{4}}\bigg).
\end{align}
\begin{align}\label{energy2 est4}
	&4\int_{\R^N}|\widehat{U}_\eps|^2|\nabla \widehat{U}_\eps|^2dx=\int_{\R^N}|\nabla (\widehat{U}_\eps^2)|^{2}dx=\omega_N\int_0^{+\infty}\left|\frac{d}{dr}\widehat{U}^2_\eps(r)\right|^{2}r^{N-1}dr\nonumber\\
	&=\omega_NB_N^4\bigg[(N-2)^2\int_0^{\eps^{-\alpha-\frac{1}{2}}}\frac{s^{N+1}}{(1+s^2)^N}ds+4\left(\frac{\eps^{2\alpha+1/2}}{1+\eps^{2\alpha+1}}\right)^{N-2}\int_0^{1-\eps^{\beta-\alpha}}\frac{\eps^{\beta(2-N)}s^2(1-s)^{N-1}}{(1-\eps^{\beta-\alpha})^4}ds\bigg]\nonumber\\
	&=\mathcal{S}^{\frac{N}{2}}+\omega_NB_N^{4}\bigg[-(N-2)^2\int_{\eps^{-\alpha-\frac{1}{2}}}^{+\infty}\frac{s^{N+1}}{(1+s^2)^N}ds+4\eps^{\beta(2-N)}\left(\frac{\eps^{2\alpha+1/2}}{1+\eps^{2\alpha+1}}\right)^{N-2}\int_0^{1-\eps^{\beta-\alpha}}\frac{s^2(1-s)^{N-1}}{(1-\eps^{\beta-\alpha})^4}ds\bigg]\nonumber\\
	&=\mathcal{S}^{\frac{N}{2}}+\mathcal{O}\left(\eps^{\frac{(4\alpha+1)(N^2-4)}{4N}}\right).
\end{align}
\begin{align}\label{energy2 est5}
	\|\widehat{U}_\eps\|_q^q=\omega_N\int_0^{+\infty}|\widehat{U}_\eps(r)|^{q}r^{N-1}dr\ge\omega_NB_N^q\eps^{\frac{N}{2}-\frac{(N-2)q}{8}}\int_0^{1}\frac{s^{N-1}}{(1+s^2)^{\frac{(N-2)q}{4}}}ds=\mathcal{O}\left(\eps^{\frac{N}{2}-\frac{(N-2)q}{8}}\right).
\end{align}

\begin{lemma}\label{energy estimate2}
	Let $q\in\left(\frac{2(N+2)}{N-2},2\cdot2^*\right), \tau>0$ and $c>0$. Then there exist $0<\mu_1<1$ and $\delta'>0$ such that
	\begin{equation*}
		\check{M}_{\mu}(c)\leq\frac{\mathcal{S}^{\frac{N}{2}}}{2N}-\delta',~~~\text{for all}~~~\mu\in(0,\mu_1).
	\end{equation*}
	Moreover, the same conclusion holds provided that $q\in(2,2\cdot2^*)$ for $N\ge3$ and $\tau>0$ is suffciently large.
\end{lemma}
\begin{proof}
	Define $\widehat{V}_\eps:=\frac{\sqrt{c}}{\|\widehat{U}_{\eps}\|_2}\widehat{U}_{\eps}\in S_r(c)$. Then we immediately get that $\frac{1}{2}\leq\frac{\sqrt{c}}{\|\widehat{U}_{\eps}\|_2}\leq\frac{3}{2}$ for sufficiently small $\eps>0$.
	It follows from \eqref{pertu fun} and \eqref{energy2 est2}-\eqref{energy2 est5} that
	\begin{align*}
		I_{0}(t^{\frac{N}{2}}\widehat{V}_{\eps}(tx))&=\frac{t^2}{2}\left(\frac{\sqrt{c}}{\|\widehat{U}_{\eps}\|_2}\right)^2\int_{\R^N}|\nabla \widehat{U}_{\eps}|^2dx+\left(\frac{\sqrt{c}}{\|\widehat{U}_{\eps}\|_2}\right)^{4}t^{N+2}\int_{\R^N}|\widehat{U}_{\eps}|^2|\nabla\widehat{U}_{\eps}|^2dx\\
		&\quad-\left(\frac{\sqrt{c}}{\|\widehat{U}_{\eps}\|_2}\right)^q\frac{\tau t^{q\gamma_q}}{q}\int_{\R^N}|\widehat{U}_{\eps}|^qdx-\left(\frac{\sqrt{c}}{\|\widehat{U}_{\eps}\|_2}\right)^{2\cdot2^*}\frac{t^{\frac{N(N+2)}{N-2}}}{2\cdot2^*}\int_{\R^N}|\widehat{U}_{\eps}|^{2\cdot2^*}dx\\
		&\leq\frac{t^2}{2}\mathcal{O}\left(\eps^{\frac{(4\alpha+1)(N-2)}{2N}}\right)+t^{N+2}\left[\frac{\mathcal{S}^{\frac{N}{2}}}{4}+\mathcal{O}\left(\eps^{\frac{(4\alpha+1)(N^2-4)}{4N}}\right)\right]\\
		&\quad-\frac{\tau t^{q\gamma_q}}{q}\mathcal{O}\left(\eps^{\frac{N}{2}-\frac{(N-2)q}{8}}\right)-\frac{t^{\frac{N(N+2)}{N-2}}}{2\cdot2^*}\left[\mathcal{S}^{\frac{N}{2}}+\mathcal{O}\bigg(\eps^{\frac{(4\alpha+1)(N+2)}{4}}\bigg)\right]\\
		&\leq\frac{1}{2N}\mathcal{S}^{\frac{N}{2}}+\frac{t^2}{2}\mathcal{O}\left(\eps^{\frac{(4\alpha+1)(N-2)}{2N}}\right)+t^{N+2}\mathcal{O}\left(\eps^{\frac{(4\alpha+1)(N^2-4)}{4N}}\right)\\
		&\quad-\frac{t^{\frac{N(N+2)}{N-2}}}{2\cdot2^*}\mathcal{O}\bigg(\eps^{\frac{(4\alpha+1)(N+2)}{4}}\bigg)-\frac{\tau t^{q\gamma_q}}{q}\mathcal{O}\left(\eps^{\frac{N}{2}-\frac{(N-2)q}{8}}\right),~~~~\forall t>0.
	\end{align*}
If $q>\frac{2(N+2)}{N-2}$, we can choose $\max\{\frac{4N^2-4N+8-N(N-2)q}{16(N-2)},0\}<\alpha<\frac{N-2}{8}$, such that there exist sufficiently small constants $\eps_1>0$ and $\delta'>0$ satisfying the inequality
	\begin{equation*}\label{claim esti2 eq1}
	\sup_{t>0}I_{0}(t^{\frac{N}{2}}\widehat{V}_{\eps_1}(tx))\leq\frac{\mathcal{S}^{\frac{N}{2}}}{2N}-2\delta'.
	\end{equation*}
On the other hand, if $2<q\leq\frac{2(N+2)}{N-2}$, we can select $\tau=\eps^{-1}$ and $\max\{\frac{4N^2-12N+8-N(N-2)q}{16(N-2)},0\}<\alpha<\frac{N-2}{8}$, such that the above inequality still holds.

Note that there exist constants $0<T_3<T_4$ such that
	\begin{equation*}
\sup_{t>0}I_{\mu}(t^{\frac{N}{2}}\widehat{V}_{\eps}(tx))=\sup_{t\in[T_3,T_4]}I_{\mu}(t^{\frac{N}{2}}\widehat{V}_{\eps}(tx)),~~~~~~\sup_{t>0}I_0(t^{\frac{N}{2}}\widehat{V}_{\eps}(tx))=\sup_{t\in[T_3,T_4]}I_0(t^{\frac{N}{2}}\widehat{V}_{\eps}(tx)).
	\end{equation*}
	From this, we conclude that
	\begin{align*}\label{claim esti2 eq}
		\sup_{t>0}I_{\mu}(t^{\frac{N}{2}}\widehat{V}_{\eps_1}(tx))&\leq C(\eps_1)\mu+\sup_{t\in[T_3,T_4]}I_0(t^{\frac{N}{2}}\widehat{V}_{\eps_1}(tx))\leq C(\eps_1)\mu+\frac{\mathcal{S}^{\frac{N}{2}}}{2N}-2\delta'\leq \frac{\mathcal{S}^{\frac{N}{2}}}{2N}-\delta',
	\end{align*}
	where $\mu>0$ is sufficiently small. Moreover, we have
	\begin{align*}
		I_{\mu}(t^{\frac{N}{2}}\widehat{V}_{\eps_1}(tx))&=\frac{\mu}{\theta}t^{\theta(1+\gamma_{\theta})}\int_{\R^N}|\nabla \widehat{V}_{\eps_1}|^{\theta}dx+\frac{t^2}{2}\int_{\R^N}|\nabla \widehat{V}_{\eps_1}|^{2}dx+t^{N+2}\int_{\R^N}|\widehat{V}_{\eps_1}|^2|\nabla \widehat{V}_{\eps_1}|^{2}dx\nonumber\\
		&\quad-\frac{\tau\, t^{\frac{N(q-2)}{2}}}{q}\int_{\R^N}|\widehat{V}_{\eps_1}|^qdx-\frac{t^{N(2^*-1)}}{22^*}\int_{\R^N}|\widehat{V}_{\eps_1}|^{22^*}dx,
	\end{align*}
	which implies that there exist constants $t_1>0$ small enough and $t_2>0$ large enough such that
	\begin{equation*}
		t_1^{\frac{N}{2}}\widehat{V}_{\eps_1}(t_1x)\in\mathcal{A}_\mu(\mathcal{K}_\mu,c)~~~~\text{and}~~~~I_\mu(t_2^{\frac{N}{2}}\widehat{V}_{\eps_1}(t_2x))<0.
	\end{equation*}
	Let $\gamma_{\eps_1}(t):=(t_1+(t_2-t_1)t)^{\frac{N}{2}}\widehat{V}_{\eps_1}((t_1+(t_2-t_1)t)x)$. Then $\gamma_{\eps_1}\in\check{\Gamma}_{\mu}(c)$ and we can deduce that
	\begin{equation*}
		\check{M}_{\mu}(c)\leq\sup_{t>0}I_{\mu}(t^{\frac{N}{2}}\widehat{V}_{\eps_1}(tx))\leq\frac{\mathcal{S}^{\frac{N}{2}}}{2N}-\delta',~~\text{for some small}~~\mu>0.
	\end{equation*}
	This completes our proof.
\end{proof}

\section{The case $4+\frac{4}{N}\leq q<2\cdot2^*$}\label{mass super quasi}\setcounter{equation}{0}
\subsection{Properties of $\Lambda_\mu(c)$}
We first develop several properties of $\Lambda_\mu(c)$ as defined in \eqref{Pohozaev type mainfold}.  By adapting the approach from \cite{Bartsch2017} (see also \cite{LZ2023}), we obtain
\begin{lemma}\label{poho mani pro1}
Let $0<\mu\leq1$ and $c>0$. Then $\Lambda_\mu(c)$ is a $\mathcal{C}^1$-manifold of codimension $1$ in $S(c)$, and therefore it is a $\mathcal{C}^1$-manifold of codimension $2$ in $X$.
\end{lemma}
\begin{lemma}\label{poho mani pro2}
For any $\mu\in(0,1]$ and any $u\in S(c)$, there exists a unique number $s_u>0$ such that $u_{s_u}\in \Lambda_\mu(c)$.
\end{lemma}
\begin{proof}
	Using \eqref{Pohozaev type mainfold id} and \eqref{mcproperty eq1-1}, we perform a direct computation to obtain $\psi_u'(s)=\frac{d}{ds}I_\mu(u_s)=\frac{1}{s}Q(u_s)$, which implies that
	\begin{equation*}
		\psi_u'(s)=0\Longleftrightarrow u_s\in\Lambda_\mu(c).
	\end{equation*}
	It is straightforward to observe that $\psi_u(s)\to0^+$ as $s\to0$ and $\psi_u(s)\to-\infty$ as $s\to+\infty$. Consequently, the maximum value $\max\limits_{s\in(0,+\infty)}\psi_u(s)$ is attained at some $s_u>0$, such that $\psi_u'(s_u)=0$ and $u_{s_u}\in\Lambda_\mu(c)$.
	Moreover, it follows from $q\gamma_q\ge N+2$ and $q\gamma_q\ge\theta(1+\gamma_\theta)$ that the critical point $s_u>0$ is unique for any $u\in X\backslash\{0\}$.
\end{proof}
From Lemma \ref{poho mani pro2}, we immediately get the following results.
\begin{corollary}\label{poho mani pro3}
	Let $\mu\in(0,1]$, $q\in[4+\frac{4}{N},2\cdot2^*)$, $\tau>0$ and $c>0$. Then for any $u\in\Lambda_\mu(c)$, there holds
	\begin{equation*}
		I_\mu(u)=\max\limits_{s>0}I_\mu(u_s).
	\end{equation*}
\end{corollary}
\begin{corollary}\label{poho mani pro4}
	Let $\mu\in(0,1]$, $q\in[4+\frac{4}{N},2\cdot2^*)$, $\tau>0$ and $c>0$. Then there holds
	\begin{equation*}
		\inf_{u\in\Lambda_\mu(c)}I_\mu(u)=\widehat{m}_\mu(c)=\inf_{u\in S(c)}\max_{s>0}I_\mu(u_s).
	\end{equation*}
\end{corollary}

Inspired by \cite[Lemma 5.3]{BJL2013}, we establish the following monotonicity property for $\widehat{m}_\mu(c)$.
\begin{lemma}\label{poho mani pro6}
	Let $\mu\in(0,1]$, $q\in[4+\frac{4}{N},2\cdot2^*)$, $\tau>0$ and $c>0$. Then the function $c\mapsto \widehat{m}_\mu(c)$ is non-increasing on $(0,+\infty)$.
\end{lemma}
\begin{proof}
	For any $0<\hat{c}_1<\hat{c}_2<+\infty$, it suffices to show that $\widehat{m}_\mu(\hat{c}_2)\leq\widehat{m}_\mu(\hat{c}_1)$. By the definition of $\widehat{m}_\mu(\hat{c}_1)$, there exists $u\in\Lambda_\mu(\hat{c}_1)$ such that
	\begin{equation*}
		I_\mu(u)<\widehat{m}_\mu(\hat{c}_1)+\eps,
	\end{equation*}
	for any $\eps>0$. Let $\eta\in\mathcal{C}_0^{\infty}(\R^N)$ a cut-off function satisfying $\eta\equiv1$ in $B_1$, $\eta\equiv0$ in $\R^N\backslash B_2$ and $0\leq\eta\leq1$ for $1\leq|x|\leq2$. Then, for $\delta\in(0,1]$ small, define
	\begin{equation*}
		u^{\delta}(x):=\eta(\delta x)u(x).
	\end{equation*}
	It is easy to verify that $u^{\delta}\to u$ in $X$ as $\delta\to0$. By continuity, we have
	\begin{equation}\label{poho mani pro6 eq1}
		I_\mu(u^{\delta})\to I_\mu(u)<\widehat{m}_\mu(\hat{c}_1)+\frac{\eps}{4},~~~~Q_\mu(u^{\delta})\to Q_\mu(u)=0.
	\end{equation}
	By Lemma \ref{poho mani pro2}, for any $\delta>0$, there exists some $s_{\delta}>0$ such that $u^{\delta}_{s_{\delta}}\in\Lambda_\mu(c)$. We claim that $\{s_{\delta}\}$ is bounded. Indeed, if $s_{\delta}\to+\infty$ as $\delta\to0$, recalling that $u^\delta\to u\neq0$ in $X$ as $\delta\to0$, we have
	\begin{align*}
		0=\lim_{\delta\to0}\frac{I_\mu(u^{\delta}_{s_{\delta}})}{s_{\delta}^{N+2}}&=\frac{\mu}{\theta}s_\delta^{\theta(1+\gamma_\theta)-N-2}\int_{\R^N}|\nabla u^\delta|^\theta dx +\frac{1}{2s_\delta^{N}}\int_{\R^N}|\nabla u^\delta|^2dx+\int_{\R^N}|u^\delta|^2|\nabla u^\delta|^2dx\nonumber\\
		&\quad-\frac{\tau}{q}s_\delta^{\frac{N(q-2)}{2}-N-2}\int_{\R^N}|u^\delta|^qdx-\frac{1}{2\cdot2^*}s_\delta^{N(2^*-1)-N-2}\int_{\R^N}|u^\delta|^{2\cdot2^*}dx.\\
		&=-\infty,
	\end{align*}
	which yields a contradiction. Thus, we may assume that, up to a subsequence, $s_{\delta}\to\hat s$ as $\delta\to0$. Then obviously we can check that $Q_\mu(u_{s_{\delta}}^{\delta})\to Q_\mu(u_{\hat s})$. By the uniqueness and the fact that $Q_\mu(u)=0$, we get $\hat s=1$. Moreover, we have 
	\begin{align}\label{poho mani pro2 eq1}
		I_\mu(u)-I_\mu(u_s)&=\frac{1-s^{N+2}}{N+2}Q_\mu(u)+\frac{\mu}{\theta}\left(1-s^{\theta(1+\gamma_\theta)}-\frac{\theta(1+\gamma_\theta)(1-s^{N+2})}{N+2}\right)\int_{\R^N}|\nabla u|^\theta dx\nonumber\\
		&\quad+\left(\frac{1-s^2}{2}-\frac{1-s^{N+2}}{N+2}\right)\int_{\R^N}|\nabla u|^2dx+\tau\left(\frac{N(q-2)(1-s^{N+2})}{2(N+2)q}-\frac{1-s^{q\gamma_q}}{q}\right)\int_{\R^N}|u|^qdx\nonumber\\
		&\quad+\left(\frac{1-s^{N+2}}{4}-\frac{(N-2)(1-s^{N(2^*-1)})}{4N}\right)\int_{\R^N}|u|^{2\cdot2^*}dx\nonumber\\
		&\ge\frac{1-s^{N+2}}{N+2}Q_\mu(u)+\left(\frac{1-s^2}{2}-\frac{1-s^{N+2}}{N+2}\right)\int_{\R^N}|\nabla u|^2dx\nonumber\\
		&\quad+\left(\frac{1-s^{N+2}}{4}-\frac{(N-2)(1-s^{N(2^*-1)})}{4N}\right)\int_{\R^N}|u|^{2\cdot2^*}dx.
	\end{align}
	Hence,
	\begin{align*}
		I_\mu(u_{s_{\delta}}^{\delta})&\leq I_\mu(u^\delta)-\frac{1-s_\delta^{N+2}}{N+2}Q_\mu(u^\delta)-\left(\frac{1-s_\delta^2}{2}-\frac{1-s_\delta^{N+2}}{N+2}\right)\int_{\R^N}|\nabla u^\delta|^2dx\\
		&\quad-\left(\frac{1-s_\delta^{N+2}}{4}-\frac{(N-2)(1-s_\delta^{N(2^*-1)})}{4N}\right)\int_{\R^N}|u^\delta|^{2\cdot2^*}dx,
	\end{align*}
	which, in combination with \eqref{poho mani pro6 eq1}, shows that there exists $\delta_0\in(0,1)$ sufficiently small such that
	\begin{equation}\label{poho mani pro6 eq8}
		I_\mu(u_{s_{\delta_0}}^{\delta_0})\leq I_\mu(u^{\delta_0})+\frac{\eps}{8}\leq I_\mu(u)+\frac{\eps}{4}<\widehat{m}_\mu(\hat{c}_1)+\frac{\eps}{2}.
	\end{equation}
	For the above $\delta_0>0$, we set $v\in\mathcal{C}^{\infty}_0(\R^N)$ such that $\supp v\subset B_{4/\delta_0}\backslash B_{2/\delta_0}$ and define $v^{\delta_0}:=\left(\frac{\hat{c}_2-\|u^{\delta_0}\|_2^2}{\|v\|_2^2}\right)^{\frac{1}{2}}v$.
	Then, we have $\|v^{\delta_0}\|_2^2=\hat{c}_2-\|u^{\delta_0}\|_2^2$. For any $\kappa\in(0,1)$, we set $w^{\kappa}:=u^{\delta_0}+v^{\delta_0}_\kappa$ with $\|v^{\delta_0}_\kappa\|_2^2=\|v^{\delta_0}\|_2^2$ and it is easy to verify that
	\begin{equation*}
		\dist\{\supp u^{\delta_0}, \supp v^{\delta_0}_\kappa\}\ge\frac{2}{\delta_0\kappa}-\frac{2}{\delta_0}>0.
	\end{equation*}
	Furthermore, we have
	\begin{equation}\label{poho mani pro6 eq3}
		\|\nabla w^{\kappa}\|_2^2=\|\nabla u^{\delta_0}\|_2^2+\kappa^2\|\nabla v^{\delta_0}\|_2^2,~~~~~~\|\nabla w^{\kappa}\|_\theta^\theta=\|\nabla u^{\delta_0}\|_\theta^\theta+\kappa^{\theta(1+\gamma_{\theta})}\|\nabla v^{\delta_0}\|_\theta^\theta,
	\end{equation}
	\begin{equation}\label{poho mani pro6 eq5}
		\|w^{\kappa}\nabla w^{\kappa}\|_2^2=\|u^{\delta_0}\nabla u^{\delta_0}\|_2^2+\kappa^{N+2}\|v^{\delta_0}\nabla v^{\delta_0}\|_2^2,~~~~
	\end{equation}
	\begin{equation}\label{poho mani pro6 eq7}
		\|w^{\kappa}\|_q^q=\|u^{\delta_0}\|_q^q+\kappa^{\frac{N(q-2)}{2}}\|v^{\delta_0}\|_q^q,~~~~~~\|w^{\kappa}\|_{2\cdot2^*}^{2\cdot2^*}=\|u^{\delta_0}\|_{2\cdot2^*}^{2\cdot2^*}+\kappa^{N(2^*-1)}\|v^{\delta_0}\|_{2\cdot2^*}^{2\cdot2^*}.
	\end{equation}
	By \eqref{poho mani pro6 eq3}-\eqref{poho mani pro6 eq7}, we deduce that as $\kappa\to0$,
	\begin{equation*}
		I_\mu(w^{\kappa})\to I_\mu(u^{\delta_0})~~~~\text{and}~~~~Q_\mu(w^{\kappa})\to Q_\mu(u^{\delta_0}).
	\end{equation*}
	Moreover, it follows from $\|w^{\kappa}\|_2^2=\|u^{\delta_0}\|_2^2+\|v^{\delta_0}\|_2^2=\hat{c}_2$ that $w^{\kappa}\in S(\hat{c}_2)$ for any $\kappa\in(0,1)$. Consequently, by Lemma \ref{poho mani pro2}, there exists $s_\kappa>0$ such that $w^{\kappa}_{s_\kappa}\in\Lambda_\mu(\hat{c}_2)$. Using a similar argument as before, we infer that $\{s_\kappa\}$ is bounded. Thus we may assume that up to a subsequence, $s_\kappa\to\hat s_0$ as $\kappa\to0$. 
	In view of \eqref{poho mani pro6 eq3}-\eqref{poho mani pro6 eq7} again, as $\kappa\to0$, we get
	\begin{equation*}
		\|\nabla w^{\kappa}_{s_\kappa}\|_2^2\to \hat s_0^2\|\nabla u^{\delta_0}\|_2^2,~~~\|\nabla w^{\kappa}_{s_\kappa}\|_\theta^\theta\to\hat s_0^{\theta(1+\gamma_{\theta})}\|\nabla u^{\delta_0}\|_\theta^\theta,~~~\|w^{\kappa}_{s_\kappa}\nabla w^{\kappa}_{s_\kappa}\|_2^2\to \hat s_0^{N+2}\|u^{\delta_0}\nabla u^{\delta_0}\|_2^2,
	\end{equation*}
	\begin{equation*}
		\|w^{\kappa}_{s_\kappa}\|_q^q\to\hat s_0^{\frac{N(q-2)}{2}}\|u^{\delta_0}\|_q^q,~~~~~\|w^{\kappa}_{s_\kappa}\|_{2\cdot2^*}^{2\cdot2^*}\to\hat s_0^{N(2^*-1)}\|u^{\delta_0}\|_{2\cdot2^*}^{2\cdot2^*},
	\end{equation*}
	which implies that there exists $\kappa_0\in(0,1)$ sufficiently small such that
	\begin{equation*}\label{poho mani pro6 eq9}
		I_\mu(w^{\kappa}_{s_\kappa})\leq I_\mu(u^{\delta_0}_{\hat s_0})+\frac{\eps}{2}.
	\end{equation*}
	Therefore, together with Corollary \ref{poho mani pro3} and \eqref{poho mani pro6 eq8} we deduce
	\begin{align*}
		\widehat{m}_\mu(\hat{c}_2)\leq I_\mu(w^{\kappa}_{s_\kappa})\leq I_\mu(u^{\delta_0}_{\hat s_0})+\frac{\eps}{2}\leq\max_{s>0}I_\mu(u^{\delta_0}_s)+\frac{\eps}{2}=I_\mu(u^{\delta_0}_{s_{\delta_0}})+\frac{\eps}{2}\leq \widehat{m}_\mu(\hat{c}_1)+\eps.
	\end{align*}
	This completes our proof, owing to the arbitrariness of $\eps>0$.
\end{proof}

\subsection{The existence of Palais-Smale sequence}
In this subsection, we will give an additional minimax characterization of $\widehat{m}_{\mu}(c)$.
To start with, we show the functional $I_\mu$ possesses a mountain pass geometry on the constraint $S(c)$.
\begin{lemma}\label{mp struc3}
Let $\mu\in(0,1]$ and $\tau>0$. Assume that one of the following conditions holds:
\begin{enumerate}[label=(\roman*)]
	\item $q=4+\frac{4}{N}$, $c\in(0,\bar{c}_3)$.
	\item $4+\frac{4}{N}<q<2\cdot2^*$, $c>0$.
\end{enumerate}
Then there exists $\mathcal{K}_{\mu}=\mathcal{K}_{\mu}(c,\tau,q,N)>0$ small enough such that
	\begin{equation}\label{mp struc3 eq1}
		0<\sup_{u\in \mathcal{A}_{\mu}(\mathcal{K}_{\mu},c)}I_{\mu}(u)<\inf_{u\in\partial \mathcal{A}_{\mu}(2\mathcal{K}_{\mu},c)}I_{\mu}(u).
	\end{equation}
\end{lemma}
\begin{proof}
 Using \eqref{gn2} and the Sobolev inequality, for any $u\in \mathcal{A}_{\mu}(\mathcal{K}_{\mu},c)$, $v\in\partial \mathcal{A}_{\mu}(2\mathcal{K}_{\mu},c)$, we get that
	\begin{align*}
		I_{\mu}(v)-I_{\mu}(u)&\ge \mathcal{K}_{\mu}-\frac{\tau \mathcal{C}_2(q,N)c^{\frac{4N-q(N-2)}{2(N+2)}}}{q}\left(\int_{\R^N}|v|^2|\nabla v|^{2}dx\right)^{\frac{N(q-2)}{2(N+2)}}-\frac{1}{2\cdot2^*}\left(\frac{4}{\mathcal{S}}\right)^{\frac{2^*}{2}}\left(\int_{\R^N}|v|^2|\nabla v|^{2}dx\right)^{\frac{N}{N-2}}\\
		&\ge \mathcal{K}_{\mu}-\frac{2^{\frac{N(q-2)}{2(N+2)}}\tau \mathcal{C}_2(q,N)c^{\frac{4N-q(N-2)}{2(N+2)}}}{q}\mathcal{K}_{\mu}^{\frac{N(q-2)}{2(N+2)}}-\frac{2^{\frac{2}{N-2}}}{2^*}\left(\frac{4}{\mathcal{S}}\right)^{\frac{2^*}{2}}\mathcal{K}_{\mu}^{\frac{N}{N-2}}.
	\end{align*}
	If $q=4+\frac{4}{N}$, we have
	\begin{align*}
		I_{\mu}(v)-I_{\mu}(u)\ge \left(1-\frac{2\tau N\mathcal{C}_2\left(4+\frac{4}{N},N\right)c^{\frac{2}{N}}}{4N+4}\right)\mathcal{K}_{\mu}-\frac{2^{\frac{2}{N-2}}}{2^*}\left(\frac{4}{\mathcal{S}}\right)^{\frac{2^*}{2}}\mathcal{K}_{\mu}^{\frac{N}{N-2}}.
	\end{align*}
	From $0<c<\bar{c}_3$, where $\bar{c}_3$ is defined by (\ref{mass quasi critical sup}), it follows that there exists sufficiently small $\mathcal{K}_{\mu}>0$ such that $I_{\mu}(v)-I_{\mu}(u)>0$, which implies \eqref{mp struc3 eq1}. If $4+\frac{4}{N}<q<2\cdot2^*$, choosing $\mathcal{K}_{\mu}>0$ suffciently small, we also have the same conclusion since $\frac{N(q-2)}{2(N+2)}>1$.
\end{proof}

Next, we present the following conclusions without proof, as the arguments are analogous to those in Lemma \ref{mp struc2} and Lemma \ref{ps existence sec2}.
\begin{lemma}\label{mp struc3 eq3}
	Under the assumptions of Lemma \ref{mp struc3}, there holds:
	\begin{equation*}
		\widehat{\Gamma}_{\mu}(c):=\{\gamma\in \mathcal{C}([0,1], S_r(c)): \gamma(0)\in\mathcal{A}_{\mu}(\mathcal{K}_{\mu},c), I_{\mu}(\gamma(1))<0\}\neq\varnothing
	\end{equation*}
	and
	\begin{equation*}
		\widehat{M}_{\mu}(c):=\inf\limits_{\gamma\in\widehat{\Gamma}_{\mu}(c)}\max\limits_{t\in[0,1]}I_{\mu}(\gamma(t))\ge\widehat{\zeta}_{0}(c)>\sup\limits_{\gamma\in\widehat{\Gamma}_{\mu}(c)}\max\{I_{\mu}(\gamma(0)), I_{\mu}(\gamma(1))\},
	\end{equation*}
	where $\widehat{\zeta}_0(c)>0$ is independent of $\mu>0$.
\end{lemma}

\begin{lemma}\label{ps existence sec3}
	Under the assumptions of Lemma \ref{mp struc3}, there exists a sequence $\{u_n\}\subset S_r(c)$ satisfying
	\begin{equation*}
		I_{\mu}(u_n)\to \widehat{M}_{\mu}(c)>0,~~~~~~~~~\|I'_{\mu}|_{S_r(c)}(u_n)\|_{X^*}\to0,~~~~~~~~~Q_{\mu}(u_n)\to0.
	\end{equation*}
\end{lemma}

\begin{lemma}\label{poho mani pro5}
	Under the assumptions of Lemma \ref{mp struc3}, we have $\widehat{M}_{\mu}(c)=\widehat{m}_\mu(c)$.
\end{lemma}
\begin{proof}
	On the one hand, for any $u\in\Lambda_\mu(c)$, there exist $s_1<0$ small and $s_2>0$ large such that
	\begin{equation*}
		u_{s_1}\in\mathcal{A}_\mu(\mathcal{K}_\mu,c)~~~~\text{and}~~~~I_\mu(u_{s_2})<0.
	\end{equation*}
	Let $\widehat{\gamma}(t):=u_{(1-t)s_1+ts_2},~\forall t\in[0,1]$. Clearly, $\widehat{\gamma}\in\widehat{\Gamma}_{\mu}(c)$. By Corollary \ref{poho mani pro3}, we have
	\begin{equation*}
		\widehat{M}_{\mu}(c)\leq\inf_{u\in\Lambda_\mu(c)}I_\mu(u)=\widehat{m}_\mu(c).
	\end{equation*}
	On the other hand, letting $s\to0$ in \eqref{poho mani pro2 eq1}, we get
	\begin{align*}
		I_\mu(u)\ge\frac{1}{N+2}Q_\mu(u)+\frac{N}{2(N+2)}\int_{\R^N}|\nabla u|^2dx+\frac{1}{2N}\int_{\R^N}|u|^{2\cdot2^*}dx\ge\frac{1}{N+2}Q_\mu(u),~~\forall u\in S(c).
	\end{align*}
	Thus for any $\gamma\in\widehat{\Gamma}_{\mu}(c)$, we have $Q_\mu(\gamma(1))\leq(N+2)I_\mu(\gamma(1))<0$.
By employing a similar argument as in Lemma \ref{poho mani pro2}, it can be readily verified that there exists $u_1\in\mathcal{A}_\mu(\mathcal{K}_\mu,c)$ such that $Q_\mu(u_1)>0$.  This implies that there exists some $t_0$ such that $\gamma(t_0)\in\Lambda_\mu(c)$. Consequently, we deduce that
	\begin{equation*}
		\max_{t\in[0,1]}I_\mu(\gamma(t))\ge\inf_{u\in\Lambda_\mu(c)}I_\mu(u)=\widehat{m}_\mu(c),
	\end{equation*}
which in turn implies that $\widehat{M}_{\mu}(c)\ge\widehat{m}_\mu(c)$.
This completes the proof.
\end{proof}
\begin{remark}
	By Lemma \ref{poho id}, any critical point of $I_\mu|_{S(c)}$ lies in the set  $\Lambda_\mu(c)$. Therefore, if  $\widehat{m}_\mu(c)$ attained, the minimizer must be a ground state critical point of $I_\mu|_{S(c)}$.
\end{remark}

Recalling the definition of $q_N$ given by \eqref{qn definition}, we now present the following result for further convenience.

\begin{lemma}\label{energy estimate2 masssuper}
	Let $q\in\left(q_N,\frac{4N}{N-2}\right), \tau>0$ and $c>0$. Then there exist $0<\mu_2<1$ and $\delta''>0$ such that
	\begin{equation*}
		\widehat{M}_{\mu}(c)\leq\frac{\mathcal{S}^{\frac{N}{2}}}{2N}-\delta'',~~~\text{for all}~~~\mu\in(0,\mu_2).
	\end{equation*}
	Moreover, the same conclusion holds if $q\in(4+\frac{4}{N},2\cdot2^*)$ for $N\ge3$ and $\tau>0$ is sufficiently large.
\end{lemma}

\section{The compactness of the Palais-Smale sequences}\label{compactness1}\setcounter{equation}{0}
In this section, we establish the compactness of the Palais-Smale sequences obtained in Lemma \ref{exist for ps seq}, Lemma \ref{ps existence sec2} and Lemma \ref{ps existence sec3}. For simplicity, we denote these three different mountain pass values, i.e., $M_{\mu}(c)$, $\check{M}_{\mu}(c)$ and $\widehat{M}_{\mu}(c)$, uniformly by $M_\mu^*(c)$.
\begin{lemma}\label{existence of sol}
	Assume that $\mu\in(0,1]$. Let $\{u_n\}\subset S_r(c)$ be a sequence obtained in Lemma \ref{exist for ps seq}, Lemma \ref{ps existence sec2} or Lemma \ref{ps existence sec3}. Then, there exists $u_{\mu}\not\equiv0$ such that, up to a subsequence, $u_n\rightharpoonup u_{\mu}$ in $X_r$ as $n\to+\infty$.
\end{lemma}
\begin{proof}
	We first establish the boundedness of $\{u_n\}$ in $X_r$, which we will discuss in the following two cases.

	\textit{Case 1:} $2<q\leq 4+\frac{4}{N}$.
	It follows from \eqref{gn2}, $I_{\mu}(u_n)\to M_{\mu}^*(c)$, and $Q_{\mu}(u_n)\to0$ that
	\begin{align*}
		M_{\mu}^*(c)+o_n(1)
		&\ge\frac{\mu}{\theta}\left(1-\theta(1+\gamma_{\theta})\frac{N-2}{N(N+2)}\right)\int_{\R^N}|\nabla u_n|^{\theta}dx+\left(\frac{1}{2}-\frac{N-2}{N(N+2)}\right)\int_{\R^N}|\nabla u_n|^{2}dx\\
		&+\frac{2}{N}\int_{\R^N}|u_n|^2|\nabla u_n|^2dx+\tau\left(\frac{\gamma_q(N-2)}{N(N+2)}-\frac{1}{q}\right)\mathcal{C}_2(q,N)c^{\frac{4N-q(N-2)}{2(N+2)}}\left(\int_{\R^N}|u_n|^2|\nabla u_n|^{2}dx\right)^{\frac{N(q-2)}{2(N+2)}}\\
	&\ge C_1\left(\int_{\R^N}|\nabla u_n|^{\theta}dx+\int_{\R^N}(1+|u_n|^2)|\nabla u_n|^{2}dx\right)-C_2\left(\int_{\R^N}(1+|u_n|^2)|\nabla u_n|^{2}dx\right)^{\frac{N(q-2)}{2(N+2)}},
	\end{align*}
for some $C_1, C_2>0$. For $q<4+\frac{4}{N}$ or $q<4+\frac{4}{N}$ with
$0<c<\bar{c}_3=\left(\frac{2N+2}{\tau N\mathcal{C}_2\left(4+\frac{4}{N},N\right)}\right)^{\frac{N}{2}}$, we deduce that $\{\int_{\R^N}(1+|u_n|^2)|\nabla u_n|^{2}dx\}$ and $\{\int_{\R^N}|\nabla u_n|^{\theta}dx\}$ are both bounded. 

	\textit{Case 2:} $4+\frac{4}{N}<q<2\cdot2^*$. Using $I_{\mu}(u_n)\to M_{\mu}^*(c)$ and $Q_{\mu}(u_n)\to0$ again, together with \eqref{gn2}, we obtain that $\{\|u_n\|_{\theta}^{\theta}\}$ is bounded and thus $\{u_n\}$ is bounded in $X_r$. 
	Therefore, up to translation, $u_n\rightharpoonup u_{\mu}$ weakly in $X_r$.

Next, we proceed by contradiction, assuming that $u_{\mu}\equiv0$. Then $\|u_n\|_r^r\to0$ in $L^r(R^N)$ for $r\in(2,2^*)$. By the interpolation inequality, we further obtain $\|u_n\|_r^r\to0$ in $L^r(R^N)$ for $r\in(2,\frac{N\theta}{N-\theta})$.
In particular,
	\begin{equation*}
		\int_{\R^N}|u_n|^{q}dx\to0~~~\text{and}~~~\int_{\R^N}|u_n|^{\frac{4N}{N-2}}dx\to0.
	\end{equation*}
	Therefore, by $Q_{\mu}(u_n)\to0$, we have
	\begin{equation*}
		\mu(1+\gamma_{\theta})\int_{\R^N}|\nabla u_n|^{\theta}dx+\int_{\R^N}|\nabla u_n|^2dx+(N+2)\int_{\R^N}|u_n|^2|\nabla u_n|^2dx\to0.
	\end{equation*}
Thus, we deduce that $I_{\mu}(u_n)\to0$, which contradicts the fact that $M_{\mu}^*(c)>0$. 
\end{proof}

\begin{lemma}\label{compact of sol}
	Assume that $\mu\in(0,1]$. Let $\{u_n\}\subset S_r(c)$ is a sequence obtained in Lemma \ref{exist for ps seq}, Lemma \ref{ps existence sec2} or  Lemma \ref{ps existence sec3}. Then, there exist $u_{\mu}\in X_r\backslash\{0\}$ and $\bar\lambda_{\mu}\in\R$, such that, up to a subsequence,
	\begin{equation}\label{compact of sol eq11}
		I_{\mu}(u_\mu)=M_{\mu}^*(c)~~~\text{and}~~~I'_{\mu}(u_{\mu})+\bar\lambda_{\mu}u_{\mu}=0.
	\end{equation}
	Moreover, if $\bar\lambda_{\mu}\neq0$, we have 
	\begin{equation*}
		\lim_{n\to+\infty}\|u_n-u_{\mu}\|_{X}=0.
	\end{equation*}
\end{lemma}
\begin{proof}
Using Lemma \ref{existence of sol}, we have
	\begin{equation}\label{compact of sol eq101}
		u_n\rightharpoonup u_{\mu}~~~~\text{in}~~~X_r,~~~~~~~u_n\to u_{\mu}~~~~\text{in}~~~L^{r}(\R^N),~~r\in\left(2,\frac{N\theta}{N-\theta}\right).
	\end{equation}
By standard arguments (see \cite{BL1983} or \cite{JZZ2025}), for any $w\in X$, we get
	\begin{align}\label{compact of sol eq1}
		\langle I'_{\mu}(u_n)+\lambda_n u_n, w\rangle=o_n(1)~~~\text{with}~~~\lambda_n=-\frac{1}{c}\langle I'_{\mu}(u_n),u_n\rangle.
	\end{align}
Since $\langle I'_{\mu}(u_n),u_n\rangle$ is bounded, it follows that $\{|\lambda_n|\}$ is also bounded. Consequently, there exists $\bar\lambda_{\mu}\in\R$ such that, up to a subsequence, $\lambda_n\to\bar\lambda_{\mu}$ as $n\to+\infty$. By the boundedness of $\{u_n\}\subset X$, we get that $I'_{\mu}(u_n)+\bar\lambda_{\mu}u_n\to0$.

	To prove \eqref{compact of sol eq11}, it suffices to show that for any $w\in X$,
	\begin{equation}\label{compact of sol eq12}
		\langle I'_{\mu}(u_n)+\lambda_n u_n, w\rangle\rightarrow\langle I'_{\mu}(u_{\mu})+\bar\lambda_{\mu} u_{\mu}, w\rangle~~\text{as}~~n\to+\infty.
	\end{equation}
	Since $u_n\rightharpoonup u_{\mu}$ in $X_r$, we readily obtain the following convergences:
	\begin{equation*}
		\begin{cases}
			\int_{\R^N}\nabla u_n\nabla wdx\to\int_{\R^N}\nabla u_{\mu}\nabla wdx,\\
			\lambda_n\int_{\R^N}u_nwdx\to\bar\lambda_{\mu}\int_{\R^N}u_{\mu}wdx,\\
			\int_{\R^N}|u_n|^{q-2}u_nwdx\to\int_{\R^N}|u_{\mu}|^{q-2}u_{\mu}wdx,\\
			\int_{\R^N}|u_n|^{2\cdot2^*-2}u_nwdx\to\int_{\R^N}|u_{\mu}|^{2\cdot2^*-2}u_{\mu}wdx.\\
		\end{cases}
	\end{equation*}
	Taking into account \eqref{compact of sol eq1}, it remains to prove that
	\begin{equation}\label{compact of sol eq2}
		\int_{\R^N}|\nabla u_n|^{\theta-2}\nabla u_n\nabla wdx\to\int_{\R^N}|\nabla u_{\mu}|^{\theta-2}\nabla u_{\mu}\nabla wdx
	\end{equation}
	and
	\begin{equation}\label{compact of sol eq3}
		\int_{\R^N}\left(u_nw|\nabla u_n|^2+|u_n|^2\nabla u_n\nabla w\right)dx\to\int_{\R^N}\left(u_{\mu}w|\nabla u_{\mu}|^2+|u_{\mu}|^2\nabla u_{\mu}\nabla w\right)dx.
	\end{equation}
	Since $\{|\nabla u_n|\}$ is bounded in $L^{\theta}(\R^N)$, it follows that $\{|\nabla u_n|^{\theta-2}\nabla u_n\}$ is bounded in $L^{\frac{\theta}{\theta-1}}(\R^N)^N$. Consequently,
	\begin{equation*}
		|\nabla u_n|^{\theta-2}\nabla u_n\rightharpoonup|\nabla u_{\mu}|^{\theta-2}\nabla u_{\mu}~~~\text{in}~~~L^{\frac{\theta}{\theta-1}}(\R^N)^N.
	\end{equation*}
	Combining this with the weak convergence for any $|\nabla w|\in L^{\theta}(\R^N)$,  we conclude that \eqref{compact of sol eq2} holds. Similarly, using the Young's inequality, we have that
	\begin{equation*}
		(|u_n||\nabla u_n|^2)^{\frac{4}{3}}\leq\frac{|u_n|^4}{3}+\frac{2|\nabla u_n|^{4}}{3}~~~\text{and}~~~(|u_n|^2|\nabla u_n|)^{\frac{4}{3}}\leq\frac{2|u_n|^4}{3}+\frac{|\nabla u_n|^{4}}{3},
	\end{equation*}
	which implies that $\{|u_n||\nabla u_n|^2\}$ and $\{|u_n|^2|\nabla u_n|\}$ are bounded in $L^{\frac{4}{3}}(\R^N)$ since $\{u_n\}$ is bounded in $X$. Thus, we also have 
	\begin{eqnarray*}\label{compact of sol eq4}
		&&u_n|\nabla u_n|^2\rightharpoonup u_{\mu}|\nabla u_{\mu}|^2~~~\text{in}~~~L^{\frac{4}{3}}(\R^N),\\
		&&|u_n|^2\nabla u_n\rightharpoonup |u_{\mu}|^2\nabla u_{\mu}~~~\text{in}~~~L^{\frac{4}{3}}(\R^N)^N.
	\end{eqnarray*}
By a similar argument, \eqref{compact of sol eq3} is thereby established.
	
	Finally, in view of \eqref{compact of sol eq101}, we select the test functions $w_n=u_n$ and $w=u_{\mu}$ in \eqref{compact of sol eq12}, respectively. By invoking the weak lower semi-continuity property (see \cite[Lemma 4.3]{CJS2010}), we infer that
	\begin{equation*}
		\mu\int_{\R^N}|\nabla u_n|^{\theta}\to\mu\int_{\R^N}|\nabla u_{\mu}|^{\theta},~~~\int_{\R^N}|\nabla u_n|^2dx\to\int_{\R^N}|\nabla u_{\mu}|^2dx,~~~\int_{\R^N}|u_n|^2|\nabla u_n|^2dx\to\int_{\R^N}|u_{\mu}|^2|\nabla u_{\mu}|^2dx.
	\end{equation*}
These convergences imply that
	\begin{equation*}
		\lambda_n\int_{\R^N}|u_n|^2dx\to\bar\lambda_{\mu}\int_{\R^N}|u_{\mu}|^2dx~~~\text{and}~~~I_{\mu}(u_\mu)=M_{\mu}^*(c).
	\end{equation*}
	Thus, if $\bar\lambda_{\mu}\neq0$,  it follows that  $\int_{\R^N}|u_n|^2dx\to\int_{\R^N}|u_{\mu}|^2dx$. Consequently, $u_n\to u_{\mu}$ strongly in $X_r$, thereby completing the proof.
\end{proof}

\section{Convergence issues}\label{Convergence issues}\setcounter{equation}{0}
In this section, we address the convergence behavior as $\mu\to0^+$ and demonstrate that the sequences of critical points of $I_{\mu}$, obtained in Lemma \ref{local minimum for perturbed} and Lemma \ref{compact of sol}, converge to critical points of $I=I_0$ restricted on $\widetilde{S}(c)$. More precisely, for any fixed $\mu_n\in(0,1]$, we  identify the critical points $v_{\mu_n}$ (or $u_{\mu_n}$, respectively) of $I_{\mu_n}|_{S(c_n)}$ with $0<c_n\leq c$. These critical points solve the following problem:
\begin{equation}\label{approxi solutions}
	\begin{cases}
		-\mu_n\Delta_{\theta} v_{\mu_n}-\Delta v_{\mu_n}-v_{\mu_n}\Delta (v_{\mu_n}^2)+\lambda_{\mu_n} v_{\mu_n}=\tau|v_{\mu_n}|^{q-2}v_{\mu_n}+|v_{\mu_n}|^{2\cdot2^*-2}v_{\mu_n}~~\text{in}~~\R^N,\\
		\int_{\R^N}|v_{\mu_n}|^2dx=c_n.
	\end{cases}
\end{equation}
We then analyze the convergence process as $\mu_n\to0^+,n\to+\infty$, considering the different types of critical points involved.  This analysis is crucial for proving our main results and is presented in detail in the following two subsections.

\subsection{Local minimizers for the original problem}
In this subsection, we primarily consider the convergence behavior of the critical points obtained in Lemma \ref{local minimum for perturbed}. For simplicity, we set $v_n:=v_{\mu_n}$, $\tilde\lambda_n:=\tilde\lambda_{\mu_n}$ and $m_0^*(c):=\lim\limits_{n\to+\infty}m_{\mu_n}(c)$.
 
\begin{proposition}\label{mu to0 local mini}
Let $2<q<2+\frac{4}{N}$ and $c\in(0,c_0)$. Then there exists $(\check v,\tilde\lambda)\in\widetilde{S}(c)\times\R^+$ satisfying
	\begin{equation}\label{local minimizer eq}
		-\Delta u-u\Delta (u^2)+\lambda u=\tau|u|^{q-2}u+|u|^{2\cdot2^*-2}u,~~~x\in\R^N.
	\end{equation}
\end{proposition}
\begin{proof}
We first claim that $m_0^*(c)<0$.
	In fact, observe that for any $0<\mu_2<\mu_1\leq1$ and $u\in X$, one has $I_{\mu_1}(u)\ge I_{\mu_2}(u)$ and $V_{\mu_1}(c)\subset V_{\mu_2}(c)$. Hence, we have
\begin{equation*}
	m_{\mu_1}(c)=\inf_{u\in V_{\mu_1}(c)}I_{\mu_1}(u)\ge\inf_{u\in V_{\mu_1}(c)}I_{\mu_2}(u)\ge\inf_{u\in V_{\mu_2}(c)}I_{\mu_2}(u)=m_{\mu_2}(c),
\end{equation*}
which implies that the energy $m_\mu(c)$ is non-decreasing with respect to $\mu\in(0,1]$. Therefore, by Lemma \ref{mcproperty} (\romannumeral1), we have $m_0^*(c)\leq m_1(c)<0$, thus confirming the claim. 
Moreover, we infer that $\{v_n\}$ is bounded in $H^1(\R^N)$ and $\{v_n\nabla v_n\}$ is bounded in $L^2(\R^N)^N$ due to the fact that $v_n\in V_{\mu_n}(c)$.
	Define
	\begin{equation*}
		\bar\delta:=\limsup_{n\to+\infty}\sup_{y\in\R^N}\int_{B_1(y)}|v_n|^2dx.
	\end{equation*}
	We now prove that $\bar\delta>0$. Suppose, for the sake of contradiction, that $\bar\delta=0$. It follows from Lemma \ref{lions lemma} that $v_n\to0$ in $L^{\beta}(\R^N)$ for all $2<\beta<2\cdot2^*$.
	Using the Sobolev inequality, we obtain
\begin{equation*}\begin{aligned}
I_{\mu_n}(v_n)&=\frac{\mu_n}{\theta}\int_{\R^N}|\nabla v_n|^\theta dx+\frac{1}{2}\int_{\R^N}|\nabla v_n|^2dx+\int_{\R^N}|v_n|^2|\nabla v_n|^2dx-\frac{1}{2\cdot2^*}\int_{\R^N}|v_n|^{2\cdot2^*}dx+o_n(1)\\
	 &\geq\xi_{\mu_n}(v_n)\left[\frac{1}{2}-\frac{1}{2\cdot2^*}\left(\frac{4}{\mathcal{S}}\right)^{\frac{2^*}{2}}\rho_0^{\alpha_2}\right]+o_n(1).
\end{aligned}
\end{equation*}
Taking into account that $f(c_0,\rho_0)=0$, see Remark \ref{remark f pro1}, we arrive at
\begin{equation*}
	m_0^*(c)=\lim\limits_{n\to+\infty}m_{\mu_n}(c)=\lim\limits_{n\to+\infty}I_{\mu_n}(v_n)\ge0,
\end{equation*}
which is a contradiction. Thus, $\bar\delta>0$, i.e., there exist $\{y_n\}\subset\R^N$ and $\bar\delta>0$ such that $\liminf\limits_{n\to+\infty}\int_{B_1(y_n)}|v_n|^2dx\ge\bar\delta>0$. Let $\check v_n(\cdot):=v_n(\cdot+y_n)$. Then
\begin{equation*}
	I_{\mu_n}(\check v_n)\to m_0^*(c),~~~~I'_{\mu_n}(\check v_n)+\tilde\lambda_n\check v_n=0,~~~~\int_{B_1(0)}|\check v_n|^2dx\ge\bar\delta>0,
\end{equation*}
which implies that there exists $\check v\in\widetilde{X}\backslash\{0\}$ such that, up to a subsequence,
\begin{equation*}
	\check v_n\rightharpoonup\check{v}~~~\text{in}~~~H^1(\R^N),~~~~\check v_n\nabla \check v_n\rightharpoonup\check{v}\nabla\check v~~~\text{in}~~~L^2(\R^N)^N,~~~~\check v_n\to \check{v}~~~\text{in}~~~L^{r}_{loc}(\R^N)~~~\text{for}~~~2<r<2^*.
\end{equation*}
Using the interpolation inequality , we have $\check v_n\to \check{v}$ in $L^{r}_{loc}(\R^N)$ for all $2<r<2\cdot2^*$.
Moreover, by testing \eqref{approxi solutions} with $v_n$, it is standard to show that $\{\tilde\lambda_n\}$ is bounded, and we may assume that $\tilde\lambda_n\to\tilde\lambda$ in $\R$.
	By Lemma \ref{moser inter} and Remark \ref{regular rmk}, we can prove that $\check{v}$ satisfies
	\begin{equation}\label{mu to0 lem1 eq8}
		\langle I'(\check v)+\tilde\lambda \check v,\phi\rangle=0, ~~~~\forall \phi\in H^1(\R^N)\cap L^{\infty}(\R^N),
	\end{equation}
which implies that $\check v$ is a nontrivial solution of \eqref{local minimizer eq}. Thus, with the aid of Remark \ref{regular rmk} and Theorem \ref{profile decom}, we can find a profile decomposition of $\{\check v_n\}$ satisfying
	\begin{equation*}
		\check v_n(\cdot+y_n^i)\rightharpoonup\tilde v_i~~\text{in}~~H^1(\R^N),~~~\check v_n(\cdot+y_n^i)\nabla \check v_n(\cdot+y_n^i)\rightharpoonup\tilde v_i\nabla \tilde v_i~~\text{in}~~~L^2(\R^N)^N~~\text{as}~~n\to+\infty,
	\end{equation*}
	\begin{equation*}
		\|\nabla \check v_n\|_2^2=\sum_{j=0}^i\|\nabla \tilde v_j\|_2^2+\|\nabla v_n^i\|_2^2+o_n(1),~~~\|\check v_n\nabla \check v_n\|_2^2=\sum_{j=0}^i\|\tilde v_j\nabla \tilde v_j\|_2^2+\|v_n^i\nabla v_n^i\|_2^2+o_n(1),
	\end{equation*}
	\begin{equation*}
		\|\check v_n\|_2^2=\sum_{j=0}^i\|\tilde v_j\|_2^2+\|v_n^i\|_{2}^2+o_n(1),~~~\limsup_{n\to+\infty}\|\check v_n\|_q^q=\sum_{j=0}^{\infty}\|\tilde v_j\|_q^q,~~~\|\check v_n\|_{2\cdot2^*}^{2\cdot2^*}=\sum_{j=0}^i\|\tilde v_j\|_{2\cdot2^*}^{2\cdot2^*}+\|v_n^i\|_{2\cdot2^*}^{2\cdot2^*}+o_n(1),
	\end{equation*}
	where $v_n^i(\cdot):=\check v_n-\sum_{j=0}^i\tilde v_j(\cdot-y_n^j)$. Thus, we obtain
	\begin{equation*}\label{cc prin eq1}
		I_{\mu_n}(\check v_n)=\frac{\mu_n}{\theta}\int_{\R^N}|\nabla \check v_n|^{\theta}dx+\sum_{j=0}^iI(\tilde v_j)+I(v_n^i)+o_n(1).
	\end{equation*}
	For $j=0,1,\cdots,i$, set
	\begin{equation*}
		\omega_j:=\left(\frac{\|\check v_n\|_2}{\|\tilde v_j\|_2}\right)^{2N},~~~~\omega_n^i:=\left(\frac{\|\check v_n\|_2}{\|v_n^i\|_2}\right)^{2N}.
	\end{equation*}
	Clearly, $\omega_j\ge1$ and $\omega_n^i\ge1$. Moreover, from the convergence of $\sum_{j=0}^i\|\tilde v_j\|_2^2$, there exists some $j_0\ge0$ such that
	\begin{equation*}
		\inf_{j\ge0}\omega_j=\omega_{j_0}=\left(\frac{\|\check v_n\|_2}{\|\tilde v_{j_0}\|_2}\right)^{2N}.
	\end{equation*}
	For $j=0,1,\cdots,i$, define the scaled functions
	\begin{equation}\label{scal pro eq}
		\tilde V_j(x):=\omega_j^{\frac{1-N}{2N}}\tilde v_j\left(\omega_j^{-\frac{1}{N}}x\right),~~~~V_n^i(x):=(\omega_n^i)^{\frac{1-N}{2N}}v_n^i\left((\omega_n^i)^{-\frac{1}{N}}x\right).
	\end{equation}
	We then obtain the following relations:
	\begin{equation*}
		\|\tilde V_j\|_2^2=\|V_n^i\|_2^2=c,~~~\|\nabla \tilde V_j\|_\theta^\theta=\omega_j^{\frac{-(N+1)\theta}{2N}+1}\|\nabla \tilde v_j\|_\theta^\theta,~~~\|\nabla \tilde V_j\|_2^2=\omega_j^{-\frac{1}{N}}\|\nabla \tilde v_j\|_2^2,~~~\|\tilde V_j\nabla \tilde V_j\|_2^2=\omega_j^{-1}\|\tilde v_j\nabla \tilde v_j\|_2^2.
	\end{equation*}
	This implies that $\tilde V_j\in V_{\mu}(c)$ for any $\mu\in(0,1]$, and similarly, $V_n^i\in V_{\mu}(c)$.
	Moreover, using the Young's inequality and the Sobolev inequality, for any $j=0,1,\cdots,i$, there exist $\eps_j>0$ and $C_{\eps_j}>0$ such that
	\begin{align*}
		\|\tilde v_j\|_q^q\leq\eps_j\|\tilde v_j\|_2^2+C_{\eps_j}\left(\frac{4}{\mathcal{S}}\right)^{\frac{2^*}{2}}\|\tilde v_j\nabla \tilde v_j\|_2^{2^*}.
	\end{align*}
	Therefore, choosing $\eps_j=\frac{(1-\omega_j^{-1-\frac{1}{N}})q}{3\tau\|\tilde v_j\|_2^2}\|\tilde v_j\nabla \tilde v_j\|_2^2$ and picking $c>0$ smaller if necessary, we have
	\begin{align*}
		I(\tilde v_j)&=\frac{I(\tilde V_j)}{\omega_j^{\frac{1}{N}}}+\frac{1-\omega_j^{-\frac{2}{N}}}{2}\|\nabla \tilde v_j\|_2^2+(1-\omega_j^{-1-\frac{1}{N}})\|\tilde v_j\nabla \tilde v_j\|_2^2\\
		&\quad-\frac{\tau}{q}\left(1-\omega_j^{\frac{1-N}{2N}(q-2)}\right)\|\tilde v_j\|_q^q-\frac{N-2}{4N}\left(1-\omega_j^{\frac{1-N}{2N}(2\cdot2^*-2)}\right)\|\tilde v_j\|_{2\cdot2^*}^{2\cdot2^*},\\
		&\ge\frac{I(\tilde V_j)}{\omega_j^{\frac{1}{N}}}+(1-\omega_j^{-1-\frac{1}{N}})\|\tilde v_j\nabla \tilde v_j\|_2^2-\frac{\tau}{q}\eps_j\|\tilde v_j\|_2^2-\left(\frac{\tau}{q}C_{\eps_j}+\frac{N-2}{4N}\right)\left(\frac{4}{\mathcal{S}}\right)^{\frac{2^*}{2}}\|\tilde v_j\nabla \tilde v_j\|_2^{2^*}\\
		&\ge\frac{I(\tilde V_j)}{\omega_j^{\frac{1}{N}}}+\frac{1-\omega_j^{-1-\frac{1}{N}}}{3}\|\tilde v_j\nabla \tilde v_j\|_2^2.
	\end{align*}
	Similarly, we also get
	\begin{align*}
		I(v_n^i)&\ge\frac{I(V_n^i)}{(\omega_n^i)^{\frac{1}{N}}}+\frac{1-(\omega_n^i)^{-1-\frac{1}{N}}}{3}\|v_n^i\nabla v_n^i\|_2^2.
	\end{align*}
	Thus as $n\to+\infty$ and $i\to+\infty$, we deduce that 
	\begin{align*}
		I_{\mu_n}(\check v_n)&\ge\frac{\mu_n}{\theta}\int_{\R^N}|\nabla \check v_n|^{\theta}dx+\sum_{j=0}^i\frac{I_{\mu_n}(\tilde V_j)}{\omega_j^{\frac{1}{N}}}+\sum_{j=0}^i\frac{1-\omega_j^{-1-\frac{1}{N}}}{3}\|\tilde v_j\nabla \tilde v_j\|_2^2\\
		&\quad+\frac{I_{\mu_n}(V_n^i)}{(\omega_n^i)^{\frac{1}{N}}}+\frac{1-(\omega_n^i)^{-1-\frac{1}{N}}}{3}\|v_n^i\nabla v_n^i\|_2^2+o_n(1)\\
		&\ge m_{\mu_n}(c)+\frac{1}{3}\inf_{j\ge0}\left(1-\omega_j^{-1-\frac{1}{N}}\right)\sum_{j=0}^i\|\tilde v_j\nabla \tilde v_j\|_2^2+\frac{1-(\omega_n^i)^{-1-\frac{1}{N}}}{3}\|v_n^i\nabla v_n^i\|_2^2+o_n(1).
	\end{align*}
	Let $\tilde\delta_1:=\min\{\omega_{j_0},\omega_n^i\}$. Then we have
	\begin{equation}\label{cc prin eq2}
		m_{\mu_n}(c)\ge m_{\mu_n}(c)+\frac{1-\tilde\delta_1^{-1-\frac{1}{N}}}{3}\|\check v_n\nabla \check v_n\|_2^2+o_n(1).
	\end{equation}
	Now, we claim that there exists some $\tilde\delta_2>0$ independent of $n$ such that 
	\begin{equation}\label{cc prin eq3}
		\|\check v_n\nabla \check v_n\|_2^2\ge\tilde\delta_2>0.
	\end{equation}
	Indeed, since $m_0^*(c)<0$, for sufficiently large $n$, we can take a positive number $\tilde\sigma$ such that
	\begin{equation*}
		I_{\mu_n}(\check v_n)\leq m_0^*(c)+\tilde\sigma<0.
	\end{equation*}
	Thus, for sufficiently large $n$, we get that
	\begin{equation*}
		\frac{\tau}{q}\|\check v_n\|_q^q+\frac{1}{2\cdot2^*}\|\check v_n\|_{2\cdot2^*}^{2\cdot2^*}\ge -(m_0^*(c)+\tilde\sigma)>0,
	\end{equation*}
	which, combined with the Gagliardo-Nirenberg inequality and Sobolev inequality, implies \eqref{cc prin eq3}. Thus it follows from \eqref{cc prin eq2} that as $n\to+\infty$,
	\begin{equation*}
		m_0^*(c)\ge m_0^*(c)+\frac{1-\tilde\delta_1^{-1-\frac{1}{N}}}{3}\tilde\delta_2,
	\end{equation*}
	which implies that $\tilde\delta_1\leq1$. Therefore, we deduce that
	\begin{equation*}
		\|\check v_n\|_2\leq\|\tilde v_{j_0}\|_2~~~~~~\text{or}~~~~~~\|\check v_n\|_2\leq\|v_n^i\|_2.
	\end{equation*}
	If $\|\check v_n\|_2\leq\|v_n^i\|_2$, then
	\begin{equation*}
		\lim_{n\to+\infty}\|\check v_n\|_q^q=0.
	\end{equation*}
	Similar to the argument as before, we can easily conclude that this is impossible since $m_0^*(c)<0$. Therefore, it is necessary that
	\begin{equation*}
		\|\check v_n\|_2\leq\|\tilde v_{j_0}\|_2,
	\end{equation*}
	which implies that
	\begin{equation*}
		\|\check v_n\|_2^2=\|\tilde v_{j_0}\|_2^2,~~~~\|\nabla \check v_n\|_2^2=\|\nabla\tilde v_{j_0}\|_2^2,~~~~\|\check v_n\nabla \check v_n\|_2^2=\|\tilde v_{j_0}\nabla\tilde v_{j_0}\|_2^2.
	\end{equation*}
	Moreover, by Theorem \ref{profile decom}, there exists a fixed point $y^{j_0}\in\R^N$ such that
	\begin{equation*}
		\check v_n(x)=\tilde v_{j_0}(x-y_n^{j_0})\rightharpoonup \tilde v_{j_0}(x-y^{j_0})~~\text{as}~~n\to+\infty.
	\end{equation*}
	Thus, by the uniqueness of the weak limit, there holds $\check v=\tilde v_{j_0}(\cdot-y^{j_0})$.  Consequently,
	\begin{equation*}
		\check v_n\to \check v~~~\text{in}~~~H^1(\R^N)~~~\text{and}~~~\check v_n\nabla \check v_n\to\check v\nabla \check v~~~\text{in}~~~L^2(\R^N)^N.
	\end{equation*}
	Since
	\begin{equation*}
		Q_{\mu_n}(\check v_n)+\tau\gamma_q\|\check v_n\|_q^q+\gamma_{2\cdot2^*}\|\check v_n\|_{2\cdot2^*}^{2\cdot2^*}\to Q_0(\check v)+\tau\gamma_q\|\check v\|_q^q+\gamma_{2\cdot2^*}\|\check v\|_{2\cdot2^*}^{2\cdot2^*},
	\end{equation*}
	it follows that
	\begin{equation*}
		\mu_n\|\nabla \check v_n\|_\theta^\theta\to0~~~\text{and}~~~I(\check v)=\lim_{n\to+\infty}m_{\mu_n}(c)=m_0^*(c).
	\end{equation*}

	To proceed, it suffices to demonstrate that $\tilde\lambda>0$. In fact, from \eqref{mu to0 lem1 eq8}, $\check v$ satisfies the following Pohozaev identity
	\begin{align*}\label{poho id local eq1}
		&\frac{2-N}{2}\int_{\R^N}|\nabla \check v|^2dx+(2-N)\int_{\R^N}|\check v|^2|\nabla \check v|^2dx-\frac{\tilde\lambda N}{2}\int_{\R^N}|\check v|^2dx+\frac{\tau N}{q}\int_{\R^N}|\check v|^qdx+\frac{N}{2\cdot2^*}\int_{\R^N}|\check v|^{2\cdot2^*}dx=0.
	\end{align*}
	Thus, we get
	\begin{equation*}
		0\ge m_0^*(c)=\lim\limits_{n\to+\infty}I_{\mu_n}(\check v_{n})=I(\check v)=\frac{1}{N}\int_{\R^N}|\nabla \check v|^2dx+\frac{2}{N}\int_{\R^N}|\check v|^2|\nabla \check v|^2dx-\frac{\tilde\lambda}{2}\int_{\R^N}|\check v|^2dx,
	\end{equation*}
	which gives that
	\begin{equation*}
		\tilde\lambda\|\check v\|_2^2\ge\frac{2}{N}\int_{\R^N}|\nabla \check v|^2dx+\frac{4}{N}\int_{\R^N}|\check v|^2|\nabla \check v|^2dx>0.
	\end{equation*}
	Therefore, $\tilde\lambda>0$, and the proof is complete.
\end{proof}

Building on the above result, we are now ready to complete the proof of Theorem \ref{thm: main result1-1}.

\begin{proof}[{\bf{Proof of Theorem \ref{thm: main result1-1}} \rm}] 
	Let $\mu_n\to0^+$. By Lemma \ref{local minimum for perturbed}, there exists a sequence $\{v_{\mu_n}\}\subset V_{\mu_n}(c)$ with $0<c< c_0$ such that
	\begin{equation*}
		I'_{\mu_n}(v_{\mu_n})+\tilde\lambda_{\mu_n}v_{\mu_n}=0,~~~~I_{\mu_n}(v_{\mu_n})=m_{\mu_n}(c)\to m_0^*(c).
	\end{equation*}
Then, Proposition \ref{mu to0 local mini} implies that there exist $0\neq\check v\in\widetilde{S}(c)\cap L^{\infty}(\R^N)$ and $\tilde{\lambda}>0$ such that 
	\begin{equation*}
		I'(\check v)+\tilde\lambda \check v=0,~~~~I(\check v)=m_0^*(c).
	\end{equation*}
	Define
	\begin{equation*}
		m_0(c):=\inf\{I(v):v\in\widetilde{S}(c), dI|_{\widetilde{S}(c)}(v)=0\}.
	\end{equation*}
	Clearly, $m_0(c)\leq m_0^*(c)<0$. Taking a sequence $v_n\in\widetilde{S}(c)$ with $dI|_{\widetilde{S}(c)}(v_n)=0$ and $I(v_n)\to m_0(c)$, and applying a similar argument as in Proposition \ref{mu to0 local mini}, we obtain a function $\tilde v_0\in \widetilde{S}(c)\cap L^{\infty}(\R^N)\backslash\{0\}$ and a constant $\tilde\lambda_0>0$ such that, up to a subsequence, there holds
	\begin{equation*}
		I'(\tilde v_0)+\tilde\lambda_0 \tilde v_0=0,~~~~I(\tilde v_0)=m_0(c).
	\end{equation*}
	Moreover, since $Q_0(\tilde v_0)=0$, we have
	\begin{align*}
		I(\tilde v_0)&=\left(\frac{1}{2}-\frac{N-2}{N(N+2)}\right)\int_{\R^N}|\nabla \tilde v_0|^2dx+\frac{2}{N}\int_{\R^N} |\tilde v_0|^2|\nabla \tilde v_0|^2dx-\tau\left(\frac{1}{q}-\frac{(N-2)\gamma_q}{N(N+2)}\right)\int_{\R^N}|\tilde v_0|^qdx.
	\end{align*}
	Combining \eqref{gn2} and $I(\tilde v_0)=m_0(c)<0$, we get that
	\begin{equation}\label{pro local min eq1}
		\frac{2}{N}\int_{\R^N} |\tilde v_0|^2|\nabla \tilde v_0|^2dx<\tau\left(\frac{1}{q}-\frac{(N-2)\gamma_q}{N(N+2)}\right)\mathcal{C}_2(q,N)c^{\frac{4N-q(N-2)}{2(N+2)}}\left(\int_{\R^N} |\tilde v_0|^2|\nabla \tilde v_0|^2dx\right)^{\frac{N(q-2)}{2(N+2)}}.
	\end{equation}
	Since $q<4+\frac{4}{N}$, it follows from \eqref{pro local min eq1} that the ground state $\tilde v_0$ lies within $\mathcal{V}_0(c)$, as defined by \eqref{local ground state ball}. 
\end{proof}

\subsection{Mountain pass solutions for the original problem}
This subsection addresses the convergence issue of the critical points obtained in Lemma \ref{compact of sol}. To begin,
we recall the definition of $M_{\mu}^*(c)$ from Section \ref{compactness1} and define $M^*_0(c):=\lim\limits_{n\to+\infty}M_{\mu_n}^*(c_n)$. By applying Lemma \ref{mountain pass struc}, Lemma \ref{mp struc2} and Lemma \ref{mp struc3 eq3}, we establish that  $M^*_0(c)>0$ for all relevant values of $c$. Building on this, we establish the following proposition:

\begin{proposition}\label{converge process for mu to 0}
	Suppose $N\ge3, 2<q<2\cdot2^*, c>0$ is fixed. Let $\mu_n\to0^+$ as $n\to+\infty$, and consider a sequence $\{\bar w_n\}\subset S_r(c_n)$ with $0<c_n\leq c$ and $\{\bar\lambda_n\}\subset\R$ satisfying 
	\begin{equation}\label{converge process for mu to 0 eq1}
		I_{\mu_n}(\bar w_n)\to M_0^*(c)\in\left(0,\frac{1}{2N}\mathcal{S}^{\frac{N}{2}}\right),~~~I'_{\mu_n}(\bar w_n)+\bar\lambda_{n}\bar w_n=0,
	\end{equation}
	Then there exist $\bar u\in H^1_r(\R^N)\cap L^{\infty}(\R^N)\backslash\{0\}$ and $\bar\lambda\in\R$ such that, up to a subsequence, one of the following holds: 
	\begin{enumerate}[label=(\roman*)]
		\item either $\bar w_n\rightharpoonup \bar u$ weakly in $\widetilde{X}$, where $\bar u$ solves \eqref{quasilinear eq2} for some $\bar\lambda\in\R$, and
		\begin{equation*}
			I(\bar u)\leq M_0^*(c)-\frac{1}{2N}\mathcal{S}^{\frac{N}{2}};
		\end{equation*}
		\item or as $n\to+\infty$, we have that
		\begin{equation*}
			\mu_n\|\nabla \bar w_n\|_{\theta}^{\theta}\to0,~~~\|\nabla \bar w_n\|_2^2\to\|\nabla \bar u\|_2^2,~~~\|\bar w_n\nabla \bar w_n\|_2^2\to \|\bar u\nabla \bar u\|_2^2,
		\end{equation*}
with $I(\bar u)=M_0^*(c)$.
		Moreover, if $\bar\lambda\neq0$, then $\bar w_n\to \bar u$ in $L^2(\R^N)$. In particular, $\bar u$ is a critical point of $I$ on $\widetilde{S}(c)$, where $c=\lim\limits_{n\to+\infty}c_n$.
	\end{enumerate}
\end{proposition}
\begin{proof}
For the sequence $\{\bar w_n\}\subset S_r(c_n)$ satisfying \eqref{converge process for mu to 0 eq1}, we first show that
	\begin{equation}\label{the boundedness of the appro sol eq}
		\sup_{n\in\mathbb{N}}\max\left\{\mu_n\int_{\R^N}|\nabla \bar w_n|^{\theta}dx,\int_{\R^N}|\bar w_n|^2|\nabla \bar w_n|^{2}dx,\int_{\R^N}|\nabla \bar w_n|^{2}dx\right\}<+\infty.
	\end{equation}
	In fact, from the equation $I'_{\mu_n}(\bar w_n)+\bar\lambda_{n}\bar w_n=0$ and Lemma \ref{poho id}, we know that $Q_{\mu_n}(\bar w_n)=0$ for each $n\in\mathbb{N}$.
	Considering that $I_{\mu_n}(\bar w_n)\to M_0^*(c)$ as $n\to+\infty$, we analyze the boundedness of $\{\bar w_n\}$ by examining the following two cases.
	
	\textit{Case 1:} $2<q<4+\frac{4}{N}$.
	By \eqref{gn2} we get that
	\begin{align*}
		&M_0^*(c)+1\ge I_{\mu_n}(\bar w_n)=I_{\mu_n}(\bar w_n)-\frac{N-2}{N(N+2)}Q_{\mu_n}(\bar w_n)\\
		&\ge\frac{\mu_n}{\theta}\left(1-\theta(1+\gamma_{\theta})\frac{N-2}{N(N+2)}\right)\int_{\R^N}|\nabla \bar w_n|^{\theta}dx+\left(\frac{1}{2}-\frac{N-2}{N(N+2)}\right)\int_{\R^N}|\nabla \bar w_n|^{2}dx\\
		&\quad+\frac{2}{N}\int_{\R^N}|\bar w_n|^2|\nabla \bar w_n|^2dx+\tau\left(\frac{\gamma_q(N-2)}{N(N+2)}-\frac{1}{q}\right)\mathcal{C}_2(q,N)c^{\frac{4N-q(N-2)}{2(N+2)}}\left(\int_{\R^N}|\bar w_n|^2|\nabla \bar w_n|^{2}dx\right)^{\frac{N(q-2)}{2(N+2)}},
	\end{align*}
	which implies that the sequences $\{\int_{\R^N}(1+|\bar w_n|^2)|\nabla \bar w_n|^{2}dx\}$ and $\{\mu_n\|\nabla \bar w_n\|_{\theta}^{\theta}\}$ are bounded. Moreover, by the Sobolev inequality, $\{\bar w_n\}$ is bounded in $H^1(\R^N)$ and $L^{2\cdot2^*}(\R^N)$.
	
\textit{Case 2:} $4+\frac{4}{N}\le q<2\cdot2^*$. By $Q_{\mu_n}(\bar w_n)=0$, we get
	\begin{align*}
		M_0^*(c)+1&\ge I_{\mu_n}(\bar w_n)=I_{\mu_n}(\bar w_n)-\frac{1}{q\gamma_q}Q_{\mu_n}(\bar w_n).
	\end{align*}
By similar arguments as in Case 1, we deduce that \eqref{the boundedness of the appro sol eq} also holds in this case.
	
Consequently, there exists $\bar u\in H^1_r(\R^N)$ with $\|\bar u \nabla \bar u\|_2^2<+\infty$ such that, up to a subsequence, $\bar w_n\rightharpoonup \bar u$ in $H^1_r(\R^N)$ and $\bar w_n\nabla \bar w_n\rightharpoonup \bar u\nabla \bar u$ in $L^2(\R^N)^N$.
 We next show that $\bar u\not \equiv0$. Suppose, by contradiction, that $\bar u\equiv0$. Then $\bar w_n\to0$ in $L^r(\R^N)$ for $r\in(2,2^*)$. Using \eqref{the boundedness of the appro sol eq} and the interpolation inequality, we further infer that $\bar w_n\to0$ in $L^r(\R^N)$ for all $r\in(2,2\cdot2^*)$. 
	From $Q_{\mu_n}(\bar w_n)=0$, we obtain
	\begin{equation*}
		\mu_n(1+\gamma_{\theta})\int_{\R^N}|\nabla \bar w_n|^{\theta}dx+\int_{\R^N}|\nabla \bar w_n|^2dx+(N+2)\int_{\R^N}|\bar w_n|^2|\nabla \bar w_n|^{2}dx-\gamma_{2\cdot2^*} \int_{\R^N}|\bar w_n|^{2\cdot2^*}dx=o_n(1).
	\end{equation*}
	Thus, we have
	\begin{equation*}
		\frac{N+2}{4}\int_{\R^N}|\bar w_n|^{2\cdot2^*}dx\ge(N+2)\int_{\R^N}|\bar w_n|^2|\nabla \bar w_n|^{2}dx+o_n(1)\ge\frac{(N+2)\mathcal{S}}{4}\left(\int_{\R^N}|\bar w_n|^{2\cdot2^*}dx\right)^{\frac{N-2}{N}}+o_n(1),
	\end{equation*}
	which implies that either $\int_{\R^N}|\bar w_n|^{2\cdot2^*}dx\to0$ or $\int_{\R^N}|\bar w_n|^{2\cdot2^*}dx\ge\mathcal{S}^{\frac{N}{2}}+o_n(1)$.
	The former case implies
	\begin{equation*}
		\mu_n(1+\gamma_{\theta})\int_{\R^N}|\nabla \bar w_n|^{\theta}dx+\int_{\R^N}|\nabla \bar w_n|^2dx+(N+2)\int_{\R^N}|\bar w_n|^2|\nabla \bar w_n|^{2}dx\to0,
	\end{equation*}
	which leads to $I_{\mu_n}(\bar w_n)\to0$. This contradicts $M_0^*(c)\neq0$.  The latter case gives
	\begin{align*}
		&M_0^*(c)+o_n(1)=I_{\mu_n}(\bar w_n)\\
		&=\frac{\mu_n}{\theta}\left(1-\frac{\theta(1+\gamma_{\theta})}{N+2}\right)\int_{\R^N}|\nabla \bar w_n|^{\theta}dx+\left(\frac{1}{2}-\frac{1}{N+2}\right)\int_{\R^N}|\nabla \bar w_n|^2dx+\frac{1}{2N}\int_{\R^N}|\bar w_n|^{2\cdot2^*}dx+o_n(1)\\
		&\ge\frac{1}{2N}\int_{\R^N}|\bar w_n|^{2\cdot2^*}dx+o_n(1)\ge\frac{1}{2N}\mathcal{S}^{\frac{N}{2}}+o_n(1),
	\end{align*}
which is also a contradiction since $M_0^*(c)<\frac{1}{2N}\mathcal{S}^{\frac{N}{2}}$. Thus, $\bar u\not \equiv0$.

We now claim that
	\begin{equation}\label{converge process for mu to 0 eq2}
		\{\bar\lambda_n\}\subset\R~~\text{is bounded}.
	\end{equation}
	Indeed, by \eqref{converge process for mu to 0 eq1}, for any $\mu_n\in(0,1]$, we have $\bar\lambda_n=-\frac{1}{c_n}\langle I'_{\mu_n}(\bar w_n),\bar w_n\rangle$. By \eqref{the boundedness of the appro sol eq}, it suffices to show that
	\begin{equation*}
		\liminf_{n\to+\infty}c_n>0.
	\end{equation*}
	To see this, suppose for contradiction that $\lim\limits_{n\to+\infty}c_n=0$. Then by \eqref{gn2}, we have $\|\bar w_n\|_q^q\to0$. Since $Q_{\mu_n}(\bar w_n)=0$, arguing as before, we obtain either $\int_{\R^N}|\bar w_n|^{2\cdot2^*}dx\to0$ or $\int_{\R^N}|\bar w_n|^{2\cdot2^*}dx\ge\mathcal{S}^{\frac{N}{2}}+o_n(1)$.
	In both cases, we reach a contradiction, thus proving the claim \eqref{converge process for mu to 0 eq2}. Therefore, we may assume that $\bar\lambda_n\to\bar\lambda$ in $\R$. 
	Moreover, as in Proposition \ref{mu to0 local mini}, $(\bar u,\bar\lambda)$ satisfies \eqref{local minimizer eq} and $Q_0(\bar u)=0$.

		To establish the alternatives, set $\tilde{w}_n:=\bar w_n-\bar u, \forall n\in\mathbb{N}$. Then we have
	\begin{equation*}\label{compact of sol eq6}
		\tilde{w}_n\rightharpoonup 0~~\text{in}~~H^1_r(\R^N),~~~\tilde{w}_n\to 0~~\text{in}~~L^{\gamma}(\R^N),~~\gamma\in(2,2\cdot2^*),~~~\tilde{w}_n\to 0~~\text{a.e. in}~~\R^N.
	\end{equation*}
	Using the Brezis-Lieb type lemma (see \cite{Willem1996}), we obtain
	\begin{equation}\label{compact of sol eq7}
		\begin{cases}
			\|\nabla \bar w_n\|_{\theta}^{\theta}=\|\nabla \tilde{w}_n\|_{\theta}^{\theta}+\|\nabla \bar u\|_{\theta}^{\theta}+o_n(1),\\
			\|\nabla \bar w_n\|_2^2=\|\nabla \tilde{w}_n\|_2^2+\|\nabla \bar u\|_2^2+o_n(1),\\
			\|\bar w_n\|_2^2=\|\tilde{w}_n\|_2^2+\|\bar u\|_2^2+o_n(1),\\
			\|\bar w_n\|_q^q=\|\tilde{w}_n\|_q^q+\|\bar u\|_q^q+o_n(1),\\
			\|\bar w_n\|_{2\cdot2^*}^{2\cdot2^*}=\|\tilde{w}_n\|_{2\cdot2^*}^{2\cdot2^*}+\|\bar u\|_{2\cdot2^*}^{2\cdot2^*}+o_n(1).
		\end{cases}
	\end{equation}
	Moreover, by Remark \ref{regular rmk} and Lemma \ref{decomposition of quasi}, we have
	\begin{equation}\label{compact of sol eq8}
		\int_{\R^N}|\bar w_n|^2|\nabla \bar w_n|^2dx=\int_{\R^N}|\tilde{w}_n|^2|\nabla \tilde{w}_n|^2dx+\int_{\R^N}|\bar u|^2|\nabla \bar u|^2dx+o_n(1).
	\end{equation}
	Therefore, by $Q_{\mu_n}(\bar w_n)=Q_0(\bar u)=0$, we deduce that
	\begin{align}\label{compact of sol eq10}
		\mu_n(1+\gamma_{\theta})\|\nabla \bar w_n\|_{\theta}^{\theta}+\|\nabla \tilde{w}_n\|_2^2+(N+2)\int_{\R^N}|\tilde{w}_n|^2|\nabla \tilde{w}_n|^2dx-\gamma_{2\cdot2^*}\|\tilde{w}_n\|_{2\cdot2^*}^{2\cdot2^*}+o_n(1)=0.
	\end{align}
	For simplicity, up to a subsequence, considering $\mu_n\to0^+$, we may assume 
	\begin{equation}\label{compact of sol eq11011}
		\mu_n\|\nabla \tilde{w}_n\|_{\theta}^{\theta}\to \iota_1,~~\mu_n\|\nabla \bar w_n\|_{\theta}^{\theta}\to \iota_1,~~\|\nabla \tilde{w}_n\|_2^2\to \iota_2,~~\int_{\R^N}|\tilde{w}_n|^2|\nabla \tilde{w}_n|^2dx\to \iota_3,~~\|\tilde{w}_n\|_{2\cdot2^*}^{2\cdot2^*}\to \iota_4,
	\end{equation}
	where $ \iota_1, \iota_2, \iota_3, \iota_4\ge0$.
	Passing to the limit as $n\to+\infty$ in \eqref{compact of sol eq10},  we infer that
	\begin{equation}\label{compact of sol eq13}
		(1+\gamma_{\theta})\iota_1+\iota_2+(N+2)\iota_3=\frac{N+2}{4}\iota_4.
	\end{equation}
	Using the Sobolev inequality, we have
	\begin{align*}\label{compact of sol eq9}
		(N+2)\int_{\R^N}|\tilde{w}_n|^2|\nabla \tilde{w}_n|^2dx\leq\gamma_{2\cdot2^*}\|\tilde{w}_n\|_{2\cdot2^*}^{2\cdot2^*}+o_n(1)\leq\gamma_{2\cdot2^*}\left(\frac{4}{\mathcal{S}}\right)^{\frac{2^*}{2}}\left(\int_{\R^N}|\tilde{w}_n|^2|\nabla \tilde{w}_n|^2dx\right)^{\frac{2^*}{2}}+o_n(1).
	\end{align*}
	Letting $n\to+\infty$, we deduce that
	\begin{equation*}
		(N+2)\iota_3\leq\gamma_{2\cdot2^*}\left(\frac{4}{\mathcal{S}}\right)^{\frac{2^*}{2}}\iota_3^{\frac{2^*}{2}},
	\end{equation*}
	which implies that either $\iota_3=0$ or $\iota_3\ge\frac{\mathcal{S}^{\frac{N}{2}}}{4}$. If $\iota_3\ge\frac{\mathcal{S}^{\frac{N}{2}}}{4}$, then by \eqref{compact of sol eq7} and \eqref{compact of sol eq8}, we have
	\begin{equation*}
	M_0^*(c)+o_n(1)=I_{\mu_n}(\bar u)+\frac{\mu_n}{\theta}\|\nabla \tilde{w}_n\|_{\theta}^{\theta}+\frac{1}{2}\|\nabla \tilde{w}_n\|_2^2+\int_{\R^N}|\tilde{w}_n|^2|\nabla \tilde{w}_n|^2dx-\frac{1}{2\cdot2^*}\|\tilde{w}_n\|_{2\cdot2^*}^{2\cdot2^*}+o_n(1).
	\end{equation*}
	Taking the limit in above equation, and using \eqref{compact of sol eq11011} and \eqref{compact of sol eq13}, we obtain 
	\begin{align*}\label{compact of local sol eq12-1}
		M_0^*(c)&=\lim_{n\to+\infty}I_{\mu_n}(\bar u)+\frac{1}{\theta}\iota_1+\frac{1}{2}\iota_2+\iota_3-\frac{N-2}{4N}\iota_4\nonumber\\
		&=I(\bar u)+\frac{N+2}{4}\iota_4+\frac{1}{\theta}(1-\theta(1+\gamma_{\theta}))\iota_1-\frac{1}{2}\iota_2-(N+1)\iota_3-\frac{N-2}{4N}\iota_4\nonumber\\
		&= I(\bar u)+\left(\frac{(N^2+N+2)(1+\gamma_{\theta})}{N(N+2)}+\frac{1}{\theta}-(1+\gamma_{\theta})\right)\iota_1+\left(\frac{N^2+N+2}{N(N+2)}-\frac{1}{2}\right)\iota_2+\frac{2}{N}\iota_3\nonumber\\
		&\ge I(\bar u)+\frac{1}{2N}\mathcal{S}^{\frac{N}{2}},
	\end{align*}
	which completes the proof of alternative (i) in Proposition \ref{converge process for mu to 0}. If instead $\iota_3=0$, then combining the Sobolev inequality and \eqref{compact of sol eq13}, we get that
	\begin{equation*}\label{compact of sol eq14}
		\iota_1=\iota_2=\iota_3=\iota_4=0.
	\end{equation*}
This implies that as $n\to+\infty$,
	\begin{equation*}
		\mu_n\|\nabla \bar w_n\|_{\theta}^{\theta}\to0,~~\|\nabla \bar w_n\|_2^2\to\|\nabla \bar u\|_2^2,~~\int_{\R^N}|\bar w_n|^2|\nabla \bar w_n|^2dx\to\int_{\R^N}|\bar u|^2|\nabla \bar u|^2dx,~~\|\bar w_n\|_{2\cdot2^*}^{2\cdot2^*}\to\|\bar u\|_{2\cdot2^*}^{2\cdot2^*}.
	\end{equation*}
	From these limits, we deduce that $I_0(\bar u)=\lim\limits_{n\to+\infty}I_{\mu_n}(\bar w_n)=M_0^*(c)$.
	Moreover, if $\bar\lambda\neq0$, without loss of generality, we may assume that $\bar\lambda>0$. Since $\bar\lambda_n\to\bar\lambda$, we can assume $\bar\lambda_n>0$ for sufficiently large $n$. We can then choose $c_n=c$, and since
	\begin{equation*}
		\langle I'_{\mu_n}(\bar w_n)+\bar\lambda_n \bar w_n,\bar w_n\rangle\to\langle I'(\bar u)+\bar\lambda \bar u,\bar u\rangle,
	\end{equation*}
	we conclude that  $\bar w_n\to \bar u$ strongly in $L^2(\R^N)$, which completes the proof.
\end{proof}

With the above results in hand, we are now in a position to complete the proofs of the remaining main theorems, namely Theorem \ref{thm: main result1} through Theorem \ref{thm: main result4}.

\begin{proof}[{\bf{Proof of Theorem \ref{thm: main result1}} \rm}] 
	By Lemma \ref{exist for ps seq} and Lemma \ref{compact of sol}, there exists a sequence $\{u_{\mu_n}\}\subset S_r(c_n)$ with $0<c_n\leq c$ satisfying
	\begin{equation*}
		I'_{\mu_n}(u_{\mu_n})+\bar\lambda_{\mu_n}u_{\mu_n}=0 ~~\text{in}~~ X^*_r,~~~~I_{\mu_n}(u_{\mu_n})\to M_0(c),~~~~u_{\mu_n}\ge0.
	\end{equation*}
	Moreover, using Lemma \ref{energy estimate}, we have
	\begin{equation*}
		0<M_0(c)=\lim\limits_{n\to+\infty}M_{\mu_n}(c)\leq \lim\limits_{n\to+\infty}m_{\mu_n}(c)+\frac{\mathcal{S}^{\frac{N}{2}}}{2N}-\delta'=m_0^*(c)+\frac{\mathcal{S}^{\frac{N}{2}}}{2N}-\delta'<\frac{\mathcal{S}^{\frac{N}{2}}}{2N}.
	\end{equation*}
	Hence, one of the alternatives in Proposition \ref{converge process for mu to 0} holds.  Suppose that alternative (\romannumeral1) of Proposition \ref{converge process for mu to 0} occurs, that is, up to a subsequence, $u_{\mu_n}\rightharpoonup \bar u\neq0$ in $H^1(\R^N)$, where $\bar u$ solves problem \eqref{quasilinear eq2} for some $\bar \lambda\in\R$ and
	\begin{equation*}
		I(\bar u)\leq M_0(c)-\frac{1}{2N}\mathcal{S}^{\frac{N}{2}},
	\end{equation*}
	which gives that
	\begin{align}\label{compact of sol eq15}
		\lim_{n\to+\infty}M_{\mu_n}(c_n)\ge \lim_{n\to+\infty}I_{\mu_n}(\bar u)+\frac{1}{2N}\mathcal{S}^{\frac{N}{2}}\ge \lim_{n\to+\infty}m_{\mu_n}(\|\bar u\|_2^2)+\frac{1}{2N}\mathcal{S}^{\frac{N}{2}}\ge \lim_{n\to+\infty}m_{\mu_n}(c_n)+\frac{1}{2N}\mathcal{S}^{\frac{N}{2}}.
	\end{align}
	However, by Lemma \ref{energy estimate}, for any $c\in(0,c_0)$, we have 
	\begin{equation*}
		M_{\mu_n}(c)\leq m_{\mu_n}(c)+\frac{\mathcal{S}^{\frac{N}{2}}}{2N}-\delta',~~~\text{for some}~~~\delta'>0,
	\end{equation*}
	which yields a contradiction with \eqref{compact of sol eq15}. Consequently, this shows that necessarily (\romannumeral2) of Proposition \ref{converge process for mu to 0} occurs, namely $I'(\bar u)+\bar\lambda \bar u=0$ and $I(\bar u)=M_0(c)$.

	It now remains to show that there exists $\tau^*_1>0$ independent of $n\in\mathbb{N}$ such that if $\tau>\tau^*_1$, the Lagrange multiplier $\bar\lambda>0$. In fact, since $u_{\mu_n}\in S_r(c_n)$ is bounded in $H^1(\R^N)$ and $L^{2\cdot2^*}(\R^N)$, by \eqref{gn2}, there exist constants $K_1, K_2, K_3>0$ independent of $n\in\mathbb{N}$ such that
	\begin{equation*}
		K_1\leq\int_{\R^N}|u_{\mu_n}|^qdx\leq\mathcal{C}_2(q,N)c^{\frac{4N-q(N-2)}{2(N+2)}}\left(\int_{\R^N} |u_{\mu_n}|^2|\nabla u_{\mu_n}|^2dx\right)^{\frac{N(q-2)}{2(N+2)}}\leq K_2
	\end{equation*}
	and
	\begin{equation*}
		\int_{\R^N}|\nabla u_{\mu_n}|^2dx\leq K_3.
	\end{equation*}
	We define the constant
	\begin{equation*}\label{pro of thm tau^*}
		\tau^*_1:=\frac{(N-2)K_3}{(N+2-4\gamma_q)K_1}.
	\end{equation*}
	Then, we have
	\begin{align*}
		\tau^*_1&>\lim_{n\to+\infty}\frac{(N-2)\|\nabla u_{\mu_n}\|_2^2}{(N+2-4\gamma_q)\|u_{\mu_n}\|_q^q}=\frac{(N-2)\|\nabla \bar u\|_2^2}{(N+2-4\gamma_q)\|\bar u\|_q^q}>0.
	\end{align*}
From the equations $I'(\bar u)+\bar\lambda \bar u=0$ and $Q_{0}(\bar u)=0$, we obtain 
	\begin{equation}\label{pro local min eq4}
		\bar\lambda\|\bar u\|_2^2=\tau\left(1-\frac{4\gamma_q}{N+2}\right)\|\bar u\|_q^q-\frac{N-2}{N+2}\|\nabla \bar u\|_2^2.
	\end{equation}
	Therefore, if $\tau>\tau^*_1$,  it follows from \eqref{pro local min eq4} that $\lim\limits_{n\to+\infty}\bar\lambda_n=\bar\lambda>0$.
This completes the proof.
\end{proof}

\begin{proof}[{\bf{Proof of Theorem \ref{thm: main result2}} \rm}]
	By Lemma \ref{ps existence sec2} and Lemma \ref{compact of sol}, there exists a sequence $\{u_{\mu_n}\}\subset S_r(c_n)$ with $0<c_n\leq c$ satisfying
	\begin{equation*}
		I'_{\mu_n}(u_{\mu_n})+\bar\lambda_{\mu_n}u_{\mu_n}=0 ~~\text{in}~~ X^*_r,~~~~I_{\mu_n}(u_{\mu_n})\to \check M_0(c),~~~~u_{\mu_n}\ge0.
	\end{equation*}
	Moreover, using Lemma \ref{energy estimate2}, we have
	\begin{equation*}
		0<\check M_0(c)=\lim\limits_{n\to+\infty}\check M_{\mu_n}(c)\leq \frac{\mathcal{S}^{\frac{N}{2}}}{2N}-\delta'<\frac{\mathcal{S}^{\frac{N}{2}}}{2N}.
	\end{equation*}
	Hence, one of the alternatives in Proposition \ref{converge process for mu to 0} holds. Now we suppose that (\romannumeral1) of Proposition \ref{converge process for mu to 0} takes place, that is, up to a subsequence, $u_{\mu_n}\rightharpoonup \bar u\neq0$ in $H^1(\R^N)$, where $\bar u$ solves problem \eqref{quasilinear eq2} for some $\bar \lambda\in\R$ and
	\begin{equation}\label{proof result2}
		I(\bar u)\leq\check M_0(c)-\frac{1}{2N}\mathcal{S}^{\frac{N}{2}}<0.
	\end{equation}
We now distinguish the following three cases.

	\textit{Case 1:} $q=2+\frac{4}{N}$.
	From \eqref{mass normal critical sup}, \eqref{gn} and the fact that $Q_0(\bar u)=0$, we deduce that
	\begin{equation*}
		I(\bar u)\ge\frac{N^2+4}{2N(N+2)}\int_{\R^N}|\nabla \bar u|^{2}dx+\frac{2}{N}\int_{\R^N}|\bar u|^2|\nabla \bar u|^2dx-\frac{\tau(N^2+4)}{2(N+2)^2}\mathcal{C}_1^{2+\frac{4}{N}}(2+\frac{4}{N},N)c^{\frac{2}{N}}\int_{\R^N}|\nabla \bar u|^{2}dx\ge0,
	\end{equation*}
	which is a contradiction with \eqref{proof result2}.

	\textit{Case 2:} $N=3, ~2+\frac{4}{N}<q<4+\frac{4}{N}$ or $N\ge4, ~2+\frac{4}{N}<q\leq2^*$.
	To begin, we define
	\begin{equation}\label{compact of sol eq17}
		\rho_*:=\rho_*(c,\tau,q,N):=\left(\frac{1}{\tau\mathcal{C}_1^q(q,N)}\frac{(N^2+4)q}{N[4N-q(N-2)]}\right)^{\frac{4}{N(q-2)-4}}c^{\frac{q(N-2)-2N}{N(q-2)-4}}.
	\end{equation}
We then proceed by distinguishing between the following two subcases.

	\textit{Subcase 1:} $\|\nabla \bar u\|_2^2+\|\bar u\nabla \bar u\|_2^2\leq\rho_*$. From \eqref{gn}, \eqref{compact of sol eq17} and $Q_0(\bar u)=0$, we obtain
	\begin{align}\label{compact of sol eq20}
		I(\bar u)&\ge\left(\frac{1}{2}-\frac{N-2}{N(N+2)}\right)\int_{\R^N}|\nabla \bar u|^{2}dx\nonumber\\
		&\quad+\frac{2}{N}\int_{\R^N}|\bar u|^2|\nabla \bar u|^2dx+\tau\left(\frac{\gamma_q(N-2)}{N(N+2)}-\frac{1}{q}\right)\mathcal{C}_1^q(q,N)c^{\frac{2N-q(N-2)}{4}}\left(\int_{\R^N}|\nabla \bar u|^{2}dx\right)^{\frac{N(q-2)}{4}}\nonumber\\
		&\ge0.
	\end{align}
	
	\textit{Subcase 2:} $\|\nabla \bar u\|_2^2+\|\bar u\nabla \bar u\|_2^2\geq\rho_*$. From \eqref{mass sub restriction1}, \eqref{gn2} and $Q_0(\bar u)=0$, we have
	\begin{align}\label{compact of sol eq19}
		&I(\bar u)\ge\left(\frac{1}{2}-\frac{N-2}{N(N+2)}\right)\int_{\R^N}|\nabla \bar u|^{2}dx\nonumber\\
		&+\frac{2}{N}\int_{\R^N}|\bar u|^2|\nabla \bar u|^2dx+\tau\left(\frac{\gamma_q(N-2)}{N(N+2)}-\frac{1}{q}\right)\mathcal{C}_2(q,N)c^{\frac{4N-q(N-2)}{2(N+2)}}\left(\int_{\R^N}|\bar u|^2|\nabla \bar u|^{2}dx\right)^{\frac{N(q-2)}{2(N+2)}}\nonumber\\
		&\ge D_1(N)\left(\|\nabla \bar u\|_2^2+\|\bar u\nabla \bar u\|_2^2\right)-\tau\left(\frac{4N-q(N-2)}{2q(N+2)}\right)\mathcal{C}_2(q,N)c^{\frac{4N-q(N-2)}{2(N+2)}}\left(\|\nabla \bar u\|_2^2+\|\bar u\nabla \bar u\|_2^2\right)^{\frac{N(q-2)}{2(N+2)}}\nonumber\\
		&\ge\left(\|\nabla \bar u\|_2^2+\|\bar u\nabla \bar u\|_2^2\right)^{\frac{N(q-2)}{2(N+2)}}\left(D_1(N)\left(\|\nabla \bar u\|_2^2+\|\bar u\nabla \bar u\|_2^2\right)^{\frac{4N+4-Nq}{2(N+2)}}-\tau\left(\frac{4N-q(N-2)}{2q(N+2)}\right)\mathcal{C}_2(q,N)c^{\frac{4N-q(N-2)}{2(N+2)}}\right)\nonumber\\
		&\ge\rho_*^{\frac{N(q-2)}{2(N+2)}}\left(D_1(N)\rho_*^{\frac{4N+4-Nq}{2(N+2)}}-\tau\left(\frac{4N-q(N-2)}{2q(N+2)}\right)\mathcal{C}_2(q,N)c^{\frac{4N-q(N-2)}{2(N+2)}}\right)\nonumber\\
		&\ge0.
	\end{align}
	Thus, both \eqref{compact of sol eq20} and \eqref{compact of sol eq19} contradict to \eqref{proof result2}.

	\textit{Case 3:} $N\ge4, ~2^*<q <4+\frac{4}{N}$. By $Q_0(\bar u)=0$, the interpolation inequality and the Sobolev inequality, we have
	\begin{align}\label{compact of sol eq21}
		I(\bar u)&=\left(\frac{1}{2}-\frac{N-2}{N(N+2)}\right)\int_{\R^N}|\nabla \bar u|^{2}dx+\frac{2}{N}\int_{\R^N}|\bar u|^2|\nabla \bar u|^2dx+\tau\left(\frac{\gamma_q(N-2)}{N(N+2)}-\frac{1}{q}\right)\int_{\R^N}|\bar u|^qdx\nonumber\\
		&\ge\left(\frac{1}{2}-\frac{N-2}{N(N+2)}\right)\int_{\R^N}|\nabla \bar u|^{2}dx+\frac{2}{N}\int_{\R^N}|\bar u|^2|\nabla \bar u|^2dx\nonumber\\
		&\quad+\tau\left(\frac{\gamma_q(N-2)}{N(N+2)}-\frac{1}{q}\right)\frac{2^{q-2^*}}{\mathcal{S}^{\frac{2^*}{2}}}\left(\int_{\R^N}|\nabla \bar u|^{2}dx\right)^{\frac{2\cdot2^*-q}{2}}\left(\int_{\R^N}|\bar u|^2|\nabla \bar u|^{2}dx\right)^{\frac{q-2^*}{2}}.\nonumber
	\end{align}
	Then, there exists a sufficiently small $\tilde\rho_*>0$ small enough such that for any $\bar u\in W_{rad}^{1,\theta}(\R^N)\cap H_{rad}^{1}(\R^N)$ satisfying $\|\nabla \bar u\|_2^2+\|\bar u\nabla \bar u\|_2^2\leq\tilde\rho_*$, we get $I(\bar u)\ge0$. Otherwise, if $\|\nabla \bar u\|_2^2+\|\bar u\nabla \bar u\|_2^2\geq\tilde\rho_*$, by choosing the mass $c>0$ sufficiently small and arguing similarly as in \eqref{compact of sol eq19}, we obtain a contradiction with \eqref{proof result2}. Thus, conclusion (\romannumeral2) of Proposition \ref{converge process for mu to 0} holds, namely there exist $\bar u\in H^1_r(\R^N)\cap L^{\infty}(\R^N)\backslash\{0\}$, $\bar u\ge0$ and $\bar\lambda\in\R$ such that
	\begin{equation*}
		I'(\bar u)+\bar\lambda \bar u=0,~~~~~0<\|\bar u\|_2^2\leq c,~~~~~I(\bar u)=\check M_0(c).
	\end{equation*}
Following the same reasoning as in the proof of Theorem \ref{thm: main result1}, for $\tau>\tau^*_1$, we obtain $\bar\lambda>0$. Therefore, $\bar u\in H^1(\R^N)\cap L^{\infty}(\R^N)$ is a critical point of $I$ on $\widetilde{S}(c)$.
\end{proof}

\begin{proof}[{\bf{Proof of Theorem \ref{thm: main result3-1}} \rm}]
	By Lemma \ref{ps existence sec3} and Lemma \ref{compact of sol}, there exists a sequence $\{u_{\mu_n}\}\subset S_r(c_n)$ with $0<c_n\leq c$ satisfying
	\begin{equation*}
		I'_{\mu_n}(u_{\mu_n})+\bar\lambda_{\mu_n}u_{\mu_n}=0 ~~\text{in}~~ X^*_r,~~~~I_{\mu_n}(u_{\mu_n})\to \widehat M_0(c),~~~~u_{\mu_n}\ge0.
	\end{equation*}
	Moreover, by Lemma \ref{energy estimate2 masssuper}, there exists $\tau^*_2>0$ large enough such that
	\begin{equation*}
		0<\widehat M_0(c)=\lim\limits_{n\to+\infty}\widehat M_{\mu_n}(c)\leq \frac{\mathcal{S}^{\frac{N}{2}}}{2N}-\delta''<\frac{\mathcal{S}^{\frac{N}{2}}}{2N}.
	\end{equation*}
	Using $Q_0(\bar u)=0$, we deduce that
	\begin{equation*}
		I(\bar u)=\left(\frac{1}{2}-\frac{1}{q\gamma_q}\right)\int_{\R^N}|\nabla \bar u|^{2}dx+\left(1-\frac{N+2}{q\gamma_q}\right)\int_{\R^N}|\bar u|^2|\nabla \bar u|^2dx+\left(\frac{N+2}{4q\gamma_q}-\frac{N-2}{4N}\right)\int_{\R^N}|\bar u|^{2\cdot2^*}dx\ge0,
	\end{equation*}
	which implies that (\romannumeral1) of Proposition \ref{converge process for mu to 0} cannot take place.
	Hence, conclusion (\romannumeral2) of Proposition \ref{converge process for mu to 0} holds, namely there exist $\bar u\in H^1_r(\R^N)\cap L^{\infty}(\R^N)\backslash\{0\}$, $\bar u\ge0$ and $\bar\lambda\in\R$ such that
	\begin{equation*}\label{main re3 eq2}
		I'(\bar u)+\bar\lambda \bar u=0,~~~~~0<\|\bar u\|_2^2\leq c,~~~~~I(\bar u)=\widehat M_0(c).
	\end{equation*}
	Moreover, proceeding as in the proof of Theorem \ref{thm: main result1}, for $\tau>\tau^*_1$, we get $\bar\lambda>0$. 
	Let us set $\tau^*=\max\{\tau_1^*,\tau_2^*\}>0$. Consequently, for any $\tau>\tau^*$, it follows from Proposition \ref{converge process for mu to 0} that $\|\bar u\|_2^2=c$ and  $\bar u\in H^1(\R^N)\cap L^{\infty}(\R^N)$ is a critical point of $I$ on $\widetilde{S}(c)$. 
	Finally, to obtain a ground state normalized solution, we define
	\begin{equation*}
		\widehat{m}_0(c):=\inf\{I(v):v\in\widetilde{S}(c), dI|_{\widetilde{S}(c)}(v)=0\}.
	\end{equation*}
	Clearly, $\widehat{m}_0(c)\leq I(v)=\widehat{M}_0(c)$. We then take a sequence $\bar u_n\in\widetilde{S}(c)$ such that $dI|_{\widetilde{S}(c)}(\bar u_n)=0$ and $I(\bar u_n)\to \widehat m_0(c)$. By employing a similar argument as in Proposition \ref{converge process for mu to 0}, we obtain $\bar u_0\in \widetilde{S}(c)\cap L^{\infty}(\R^N)\backslash\{0\}$ and $\bar\lambda_0\in\R$ such that, up to a subsequence, there holds
	\begin{equation*}
		I'(\bar u_0)+\bar\lambda_0 \bar u_0=0,~~~~I(\bar u_0)=\widehat m_0(c).
	\end{equation*}
	Using a similar argument as in the previous proof, we can show that $\bar\lambda_0>0$, and hence $\|\bar u_0\|_2^2=c$. This implies that $\bar u_0$ is a minimizer of $\widehat{m}_0(c)$. 
\end{proof}

\begin{proof}[{\bf{Proof of Theorem \ref{thm: main result3}} \rm}]
	It suffices to show that $\bar\lambda>0$ for any $\tau>0$. We first show $\bar\lambda\ge0$ . To this end, we argue by contradiction that $\bar\lambda<0$. Then for sufficiently large $n$, we have $\bar\lambda_{\mu_n}<0$. In view of Lemma \ref{compact of sol}, we have
	\begin{equation}\label{main re3 eq4}
		I_{\mu_n}(u_{\mu_n})=\widehat M_{\mu_n}(c)~~~~\text{and}~~~~\|u_{\mu_n}\|_2^2=c.
	\end{equation}
	For simplicity, for any fixed $n$ sufficiently large, define $\bar w:=u_{\mu_n}$ and let $\bar w_{t,s}(x):=s^{\frac{N}{2}}\sqrt{t}\bar w(sx)$ with $t,s>0$. We also define
	\begin{align*}
		h_I(t,s):=I_{\mu_n}(\bar w_{t,s})&=\frac{\mu_n}{\theta}s^{\theta(1+\gamma_\theta)}t^{\frac{\theta}{2}}\int_{\R^N}|\nabla \bar w|^\theta dx+\frac{s^2}{2}t\int_{\R^N}|\nabla \bar w|^2dx+s^{N+2}t^2\int_{\R^N}|\bar w|^2|\nabla \bar w|^2dx\nonumber\\
		&\quad-\frac{\tau}{q}s^{\frac{N(q-2)}{2}}t^{\frac{q}{2}}\int_{\R^N}|\bar w|^qdx-\frac{1}{2\cdot2^*}s^{N(2^*-1)}t^{2^*}\int_{\R^N}|\bar w|^{2\cdot2^*}dx,\\
		h_Q(t,s):=Q_{\mu_n}(\bar w_{t,s})&=\mu_n(1+\gamma_\theta)s^{\theta(1+\gamma_\theta)}t^{\frac{\theta}{2}}\int_{\R^N}|\nabla \bar w|^\theta dx+s^2t\int_{\R^N}|\nabla \bar w|^2dx\nonumber\\
		&\quad+(N+2)s^{N+2}t^2\int_{\R^N}|\bar w|^2|\nabla \bar w|^{2}dx-\tau\gamma_qs^{\frac{N(q-2)}{2}}t^{\frac{q}{2}}\int_{\R^N}|\bar w|^qdx\nonumber\\
		&\quad-\gamma_{2\cdot2^*}s^{N(2^*-1)}t^{2^*} \int_{\R^N}|\bar w|^{2\cdot2^*}dx.
	\end{align*}
	By a direct calculation, we have
	\begin{align*}
		\frac{\partial h_I}{\partial t}(1,1)&=\frac{\mu_n}{2}\int_{\R^N}|\nabla \bar w|^\theta dx+\frac{1}{2}\int_{\R^N}|\nabla \bar w|^2dx+2\int_{\R^N}|\bar w|^2|\nabla \bar w|^2dx-\frac{\tau}{2}\int_{\R^N}|\bar w|^qdx-\frac{1}{2}\int_{\R^N}|\bar w|^{2\cdot2^*}dx\\
		&=-\frac{\bar \lambda_{\mu_n}}{2}\int_{\R^N}|\bar w|^2dx>0,\\
		\frac{\partial h_I}{\partial s}(1,1)&=\mu_n(1+\gamma_\theta)\int_{\R^N}|\nabla \bar w|^\theta dx+\int_{\R^N}|\nabla \bar w|^2dx+(N+2)\int_{\R^N}|\bar w|^2|\nabla \bar w|^{2}dx\\
		&\quad-\tau\gamma_q\int_{\R^N}|\bar w|^qdx-\gamma_{2\cdot2^*} \int_{\R^N}|\bar w|^{2\cdot2^*}dx=0,\\
		\frac{\partial^2 h_I}{\partial s^2}(1,1)&=\mu_n(1+\gamma_\theta)(\theta(1+\gamma_{\theta})-N-2)\int_{\R^N}|\nabla \bar w|^\theta dx-N\int_{\R^N}|\nabla \bar w|^2dx\\
		&\quad+\tau\gamma_q\left(N+1-\frac{N(q-2)-2}{2}\right)\int_{\R^N}|\bar w|^qdx+\frac{N+2}{4}\left(N+1-\frac{N^2+N+2}{N-2}\right)\int_{\R^N}|\bar w|^{2\cdot2^*}dx\\
		&<0,
	\end{align*}
	which imply that for $|\delta_{s}|$ small enough and $\delta_t>0$, there holds
	\begin{equation}\label{main re3 eq3}
		h_I(1-\delta_t,1+\delta_s)<h_I(1,1).
	\end{equation}
	Clearly, $h_Q(1,1)=0$. Now, we show that $\frac{\partial h_Q}{\partial s}(1,1)\neq0$. In fact, if this is not true, by a direct computation, we get
	\begin{align*}
		\frac{\partial h_Q}{\partial s}(1,1)&=\mu_n\theta(1+\gamma_\theta)^2\int_{\R^N}|\nabla \bar w|^\theta dx+2\int_{\R^N}|\nabla \bar w|^2dx+(N+2)^2\int_{\R^N}|\bar w|^2|\nabla \bar w|^{2}dx\\
		&\quad-\tau q\gamma_q^2\int_{\R^N}|\bar w|^qdx-\gamma_{2\cdot2^*}N(2^*-1) \int_{\R^N}|\bar w|^{2\cdot2^*}dx=0.
	\end{align*}
	Together with $Q_{\mu_n}(\bar w)=0$, this yields
	\begin{align*}
		&\mu_n(1+\gamma_\theta)(N+2-\theta(1+\gamma_\theta))\int_{\R^N}|\nabla \bar w|^\theta dx+N\int_{\R^N}|\nabla \bar w|^2dx\\
		&=\tau \gamma_q\left(N+2-q\gamma_q\right)\int_{\R^N}|\bar w|^qdx+\gamma_{2\cdot2^*}\left(N+2-N(2^*-1)\right)\int_{\R^N}|\bar w|^{2\cdot2^*}dx.
	\end{align*}
	However, this is impossible due to the fact that $\frac{4N}{N+2}<\theta<\frac{4N+4}{N+2}$, $q\gamma_q\ge N+2$ and $N+2<N(2^*-1)$. Thus, by the implicit function theorem, we infer that there exist $\eps>0$ and a continuous function $\hat h:[1-\eps,1+\eps]\mapsto\R$ such that $\hat h(1)=1$ and $h_Q(t,\hat h(t))=0$ for $t\in[1-\eps,1+\eps]$. Combining \eqref{main re3 eq4} and \eqref{main re3 eq3}, we deduce that
	\begin{equation*}
		\widehat{M}_{\mu_n}((1-\eps)c)=\widehat{m}_{\mu_n}((1-\eps)c)=\inf_{u\in\Lambda_{\mu_n}((1-\eps)c)}I_{\mu_n}(u)\leq I_{\mu_n}(\bar w_{1-\eps,\hat h(1-\eps)})<I_{\mu_n}(\bar w)=\widehat{M}_{\mu_n}(c),
	\end{equation*}
	which contradicts Lemma \ref{poho mani pro6}, which states that $\widehat{m}_\mu(c)=\widehat{M}_\mu(c)$ is non-increasing on $(0,+\infty)$ for any $\mu>0$. Therefore, we infer that $\bar\lambda\ge0$. 
	Taking into account \eqref{pro local min eq4}, we derive that
	\begin{align*}
		&\frac{N-2}{N+2}\|\nabla \bar u\|_2^2\leq\tau\left(1-\frac{4\gamma_q}{N+2}\right)\|\bar u\|_q^q\leq\tau\left(1-\frac{4\gamma_q}{N+2}\right)\mathcal{C}_2(q,N)c^{\frac{4N-q(N-2)}{2(N+2)}}\left(\int_{\R^N}|\bar u|^2|\nabla \bar u|^2\right)^{\frac{N(q-2)}{2(N+2)}},
	\end{align*}
which implies that
	\begin{equation*}
		\|\nabla \bar u\|_2^2\to0~~~\text{as}~~~c\to0.
	\end{equation*}
	Thus, it follows from \eqref{pro local min eq4} that there exists a sufficiently small $c_2^*>0$ such that for any $c\in(0,c_2^*)$, we have $\bar\lambda>0$.
	From this, using Proposition \ref{converge process for mu to 0} we conclude that $\|\bar u\|_2^2=c$ and  $\bar u\in H^1(\R^N)\cap L^{\infty}(\R^N)$ is a critical point of $I$ on $\widetilde{S}(c)$. 
	The ground state solution can be obtained using the same approach.
\end{proof}

\begin{proof}[{\bf{Proof of Theorem \ref{thm: main result4}} \rm}]
	Suppose that $(u,\lambda)\in\widetilde{S}(c)\times \R^+$ is a solution to \eqref{quasi eq}. Then, by \eqref{quasi eq} and the Pohozaev identity, we obtain
	\begin{equation*}\label{result4 eq1}
		\int_{\R^N}|\nabla u|^2dx+4\int_{\R^N}|u|^2|\nabla u|^2dx+\lambda\int_{\R^N}|u|^2dx=\tau\int_{\R^N}|u|^{q}dx+\int_{\R^N}|u|^{2\cdot2^*}dx
	\end{equation*}
and
	\begin{equation*}\label{result4 eq2}
		\frac{2-N}{2}\int_{\R^N}|\nabla u|^2dx+(2-N)\int_{\R^N}|u|^2|\nabla u|^2dx-\frac{\lambda N}{2}\int_{\R^N}|u|^2dx+\frac{\tau N}{q}\int_{\R^N}|u|^qdx+\frac{N}{2\cdot2^*}\int_{\R^N}|u|^{2\cdot2^*}dx=0.
	\end{equation*}
	Combining these two equations, we get
	\begin{equation*}
		0<\lambda\|u\|_2^2=\tau\frac{4N-(N-2)q}{(N+2)q}\|u\|_q^q-\frac{N-2}{N+2}\|\nabla u\|_2^2\leq0,
	\end{equation*}
	which is a contradiction. The proof is complete.
\end{proof}

\medskip
\appendices
\section{Appendix}\label{appendixa}\setcounter{equation}{0}

This section is devoted to providing an $L^{\infty}$ estimate, a Brezis-Lieb type result, and a profile decompositions for the approximating solutions $\{u_n\}\subset\widetilde{X}$.

To be more specific, assume that $(\mu_n,u_n,\lambda_n)\in(0,1]\times\widetilde{X}\times\R$ satisfies the following equation with a general nonlinearity:
\begin{equation}\label{appro eq1 general}
-\mu\Delta_{\theta} u-\Delta u-u\Delta (u^2)+\lambda u=\mathcal{L}'(u)~~\text{in}~~\R^N,
\end{equation}
where the function $\mathcal{L}:\R\to[0,\infty)$ is of class $\mathcal{C}^1$ and satisfies
\begin{equation}\label{general non assu}
	\mathcal{L}'(s)\leq C(|s|+|s|^{2\cdot2^*-1}),
\end{equation} 
for any $s\in\R$ and some constant $C>0$.
In the following, we always assume that
\begin{equation*}
	\sup_{n\in\mathbb{N}}\max\left\{\mu_n\|\nabla u_n\|_{\theta}^\theta,~~\|u_n\nabla u_n\|_2^2,~~\|\nabla u_n\|_2^2,~~|\lambda_n|\right\}<+\infty.
\end{equation*}
Thus, there exists $(u_0,\lambda_0)$ such that as $n\to+\infty$,
\begin{equation*}
	u_n\rightharpoonup u_0~~\text{in}~~H^1(\R^N),~~~u_n\nabla u_n\rightharpoonup u_0\nabla u_0~~\text{in}~~L^2(\R^N)^N,~~~u_n\to u_0~~\text{a.e.}~~\text{in}~~\R^N,~~~\lambda_n\to\lambda_0~~\text{in}~~\R.
\end{equation*}
To begin, we establish the following result by using Moser's iteration.


\begin{lemma}\label{moser inter}
	There exists a constant $C>0$ such that $\|u_n\|_{\infty}\leq C$ and $\|u_0\|_{\infty}\leq C$.
\end{lemma}
\begin{proof}
	Following the idea from \cite[Proposition 3.1]{LLW20131} with slight modifications, for any $T>0$, we define the truncated function $u_n^T$ as
	\begin{equation*}
		u_n^T:=
		\begin{cases}
			T~~~~&\mbox{if}~~~u_n>T,\\
			u_n~~~~&\mbox{if}~~~u_n\leq T.
		\end{cases}
	\end{equation*}
	Taking $\phi=u_n|u_n^T|^{4(r-1)}$ with $r>1$ as the test function in \eqref{appro eq1 general}, and using \eqref{general non assu} and the boundedness of $\{\lambda_n\}$, we obtain
	 \begin{align}\label{mu to0 lem1 eq1}
		&(4r-3)\int_{\{u_n\leq T\}}u_n^{4(r-1)}|\nabla u_n|^2dx+\int_{\{u_n>T\}}T^{4(r-1)}|\nabla u_n|^2dx\nonumber\\
		&+2(4r-3)\int_{\{u_n\leq T\}}u_n^{4(r-1)}u_n^2|\nabla u_n|^2dx+2\int_{\{u_n>T\}}T^{4(r-1)}u_n^2|\nabla u_n|^2dx\nonumber\\
		&\leq-\lambda_n\int_{\R^N}u_n^2|u_n^T|^{4(r-1)}dx+C\int_{\R^N}(|u_n|^2+|u_n|^{2\cdot2^*})|u_n^T|^{4(r-1)}dx\nonumber\\
		&\leq C\left(\int_{\R^N}u_n^{2}(u_n^T)^{4(r-1)}dx+\int_{\R^N}u_n^{2\cdot2^*}(u_n^T)^{4(r-1)}dx\right).
	 \end{align}
	 On the other hand, by a direct computation, we deduce that
	 \begin{align}\label{mu to0 lem1 eq1-1}
		&\int_{\R^N}|\nabla(u_n(u_n^T)^{2(r-1)})|^2dx+\int_{\R^N}|\nabla(u_n^2(u_n^T)^{2(r-1)})|^2dx\nonumber\\
		&=(2r-1)^2\int_{\{u_n\leq T\}}u_n^{4(r-1)}|\nabla u_n|^2dx+\int_{\{u_n>T\}}T^{4(r-1)}|\nabla u_n|^2dx\nonumber\\
		&\quad+4r^2\int_{\{u_n\leq T\}}u_n^{2(2r-1)}|\nabla u_n|^2dx+4\int_{\{u_n>T\}}T^{4(r-1)}u_n^2|\nabla u_n|^2dx.
	 \end{align}
	 Combining \eqref{mu to0 lem1 eq1}, \eqref{mu to0 lem1 eq1-1}, and applying the Sobolev inequality, we have
	 \begin{align}\label{mu to0 lem1 eq2}
		&\left(\int_{\R^N}|u_n(u_n^T)^{2(r-1)}|^{2^*}dx\right)^{\frac{2}{2^*}}+\left(\int_{\R^N}|u_n^2(u_n^T)^{2(r-1)}|^{2^*}dx\right)^{\frac{2}{2^*}}\nonumber\\
		&\leq Cr^2\left(\int_{\R^N}u_n^{2}(u_n^T)^{4(r-1)}dx+\int_{\R^N}a_n\left(u_n^2(u_n^T)^{2(r-1)}\right)^{2}dx\right),
	 \end{align}
	where $a_n:=u_n^{\frac{8}{N-2}}$. We claim that $a_n\in L^{t}(\R^N)$ for some $t>\frac{N}{2}$. In fact, by \eqref{mu to0 lem1 eq2}, for any $K>0$ to be determined later, we apply the H\"older inequality to obtain
	\begin{align}\label{mu to0 lem1 eq3}
		\left(\int_{\R^N}|u_n^2(u_n^T)^{2(r-1)}|^{2^*}dx\right)^{\frac{2}{2^*}}&\leq Cr^2\int_{\R^N}u_n^{2}(u_n^T)^{4(r-1)}dx+Cr^2K^{\frac{8}{N-2}}\int_{\R^N}\left(u_n^2(u_n^T)^{2(r-1)}\right)^{2}dx\nonumber\\
		&\quad+Cr^2\left(\int_{|u_n|\ge K}u_n^{2\cdot2^*}dx\right)^{\frac{2}{N}}\left(\int_{\R^N}\left(u_n^2(u_n^T)^{2(r-1)}\right)^{\frac{2N}{N-2}}dx\right)^{\frac{N-2}{N}}.
	\end{align}
	Now, let $4r=2\cdot2^*$ and choose $K>0$ sufficiently large such that
	\begin{equation*}
		Cr^2\left(\int_{|u_n|\ge K}u_n^{2\cdot2^*}dx\right)^{\frac{2}{N}}<\frac{1}{2}.
	\end{equation*}
	Then, letting $T\to+\infty$ in \eqref{mu to0 lem1 eq3}, we have
	\begin{align*}
		\left(\int_{\R^N}u_n^{2\cdot2^*r}dx\right)^{\frac{2}{2^*}}\leq 2Cr^2\int_{\R^N}u_n^{4r-2}dx+2Cr^2K^{\frac{8}{N-2}}\int_{\R^N}u_n^{4r}dx,
	\end{align*}
	which implies that $u_n\in L^{2\cdot2^*r}(\R^N)$, i.e., $a_n\in L^{t}(\R^N)$ with $t=\frac{Nr}{2}>\frac{N}{2}$. Thus the claim holds.

	Back to \eqref{mu to0 lem1 eq2}, utilizing the H\"older inequality again, we derive that
	\begin{align}\label{mu to0 lem1 eq4}
		&\left(\int_{\R^N}|u_n(u_n^T)^{2(r-1)}|^{2^*}dx\right)^{\frac{2}{2^*}}+\left(\int_{\R^N}|u_n^2(u_n^T)^{2(r-1)}|^{2^*}dx\right)^{\frac{2}{2^*}}\nonumber\\
		&\leq Cr^2\left(\left(\int_{\R^N}(u_n^T)^{4t'(r-1)}u_n^{4t'-2^*}dx\right)^{\frac{1}{t'}}\left(\int_{\R^N}u_n^{\frac{2^*-2t'}{t'-1}}dx\right)^{\frac{t'-1}{t'}}+\|a_n\|_t\left(\int_{\R^N}\left(u_n^2(u_n^T)^{2(r-1)}\right)^{2t'}dx\right)^{\frac{1}{t'}}\right)\nonumber\\
		&\leq Cr^2\left(\left(\int_{\R^N}(u_n^T)^{4t'(r-1)}u_n^{4t'-2^*}dx\right)^{\frac{1}{t'}}+\left(\int_{\R^N}\left(u_n^2(u_n^T)^{2(r-1)}\right)^{2t'}dx\right)^{\frac{1}{t'}}\right),
	\end{align}
	where $t'=\frac{t}{t-1}<\frac{N}{N-2}$. Assume that
	\begin{equation*}
		\int_{\R^N}u_n^{4t'r-2^*}dx+\int_{\R^N}u_n^{4rt'}dx<+\infty.
	\end{equation*}
	Then, letting $T\to+\infty$ in \eqref{mu to0 lem1 eq4}, we get
	\begin{align*}
		&\left(\int_{\R^N}u_n^{2^*(2r-1)}dx\right)^{\frac{2}{2^*}}+\left(\int_{\R^N}u_n^{2\cdot2^*r}dx\right)^{\frac{2}{2^*}}\leq Cr^2\left(\left(\int_{\R^N}u_n^{4t'r-2^*}dx\right)^{\frac{1}{t'}}+\left(\int_{\R^N}u_n^{4rt'}dx\right)^{\frac{1}{t'}}\right),
	\end{align*}
	which implies
	\begin{align*}
		&\left(\int_{\R^N}u_n^{2^*(2r-1)}dx\right)^{\frac{1}{2\cdot2^*r}}+\left(\int_{\R^N}u_n^{2\cdot2^*r}dx\right)^{\frac{1}{2\cdot2^*r}}\leq (Cr^2)^{\frac{1}{4r}}\left(\left(\int_{\R^N}u_n^{4t'r-2^*}dx\right)^{\frac{1}{4rt'}}+\left(\int_{\R^N}u_n^{4rt'}dx\right)^{\frac{1}{4rt'}}\right).
	\end{align*}
	Set $d:=\frac{2^*}{2t'}>1$, and define $r_{i+1}=r_id$ for $i\in\mathbb{N}$ with $4r_0t'=2\cdot2^*$. Clearly, the sequence $\{r_i\}$ is positive and increasing since $r_0>0$. Then we obtain
	\begin{align*}
		&\left(\int_{\R^N}u_n^{2^*(2r_{i+1}-1)}dx\right)^{\frac{1}{2\cdot2^*r_{i+1}}}+\left(\int_{\R^N}u_n^{2\cdot2^*r_{i+1}}dx\right)^{\frac{1}{2\cdot2^*r_{i+1}}}\\
		&\leq (Cr_{i+1}^2)^{\frac{1}{4r_{i+1}}}\left(\left(\int_{\R^N}u_n^{2^*(2r_i-1)}dx\right)^{\frac{1}{2\cdot2^*r_i}}+\left(\int_{\R^N}u_n^{2\cdot2^*r_i}dx\right)^{\frac{1}{2\cdot2^*r_i}}\right).
	\end{align*}
	By iterating this process, we arrive at
	\begin{align*}
		&\left(\int_{\R^N}u_n^{2^*(2r_{i+1}-1)}dx\right)^{\frac{1}{2\cdot2^*r_{i+1}}}+\left(\int_{\R^N}u_n^{2\cdot2^*r_{i+1}}dx\right)^{\frac{1}{2\cdot2^*r_{i+1}}}\\
		&\leq\prod _{k=0}^{i+1} (Cr_{k}^2)^{\frac{1}{4r_{k}}}\left(\left(\int_{\R^N}u_n^{2^*}dx\right)^{\frac{1}{2\cdot2^*}}+\left(\int_{\R^N}u_n^{2\cdot2^*}dx\right)^{\frac{1}{2\cdot2^*}}\right).
	\end{align*}
	Letting $i\to+\infty$, we deduce that
	\begin{equation*}
		\|u_n\|_{L^{\infty}(\R^N)}\leq C\left(\|u_n\|^{\frac{1}{2}}_{L^{2^*}(\R^N)}+\|u_n\|_{L^{2\cdot2^*}(\R^N)}\right).
	\end{equation*}
Thus, there exists a constant $C>0$ such that $\|u_n\|_{L^{\infty}(\R^N)}\leq C$.  Similarly, the bound $\|u_0\|_{L^{\infty}(\R^N)}\leq C$ can be derived in the same way.
\end{proof}
\begin{remark}\label{regular rmk}
	By Lemma \ref{moser inter}, we get that $\|u_n\|_{\infty}\leq C$, which together with the similar arguments as in \cite[Theorem 4.1]{LZ2023}, we can see that $(u_0,\lambda_0)$ satisfies the following equation:
\begin{equation}\label{appro eq2 general}
-\Delta u-u\Delta (u^2)+\lambda u=\mathcal{L}'(u)~~\text{in}~~\R^N.
\end{equation}
	By Corollary 4.23 and Theorem 4.24 of \cite{HL2000}, we have $u_0\in\mathcal{C}^{1,\beta_0}(\R^N)$ for some $\beta_0>0$. For the approximating solutions $u_n$, it follows from the regularity result in \cite{To1984} that $u_n\in\mathcal{C}^{1,\beta_1}_{loc}(\R^N)$ for some $\beta_1\in(0,1)$.
\end{remark}

\begin{lemma}\label{decomposition of quasi}
	Let $\bar{u}_n:=u_n-u_0$. Then we have
	\begin{equation*}
		\int_{\R^N}|u_n|^2|\nabla u_n|^2dx=\int_{\R^N}|\bar{u}_n|^2|\nabla \bar{u}_n|^2dx+\int_{\R^N}|u_0|^2|\nabla u_0|^2dx+o_n(1).
	\end{equation*}
\end{lemma}
\begin{proof}
	By a direct computation, we obtain
	\begin{align}\label{decomposition of quasi eq1}
		\int_{\R^N}|u_n|^2|\nabla u_n|^2dx&=\int_{\R^N}|\bar{u}_n+u_0|^2|\nabla \bar{u}_n+\nabla u_0|^2dx\nonumber\\
		&=\int_{\R^N}|\bar{u}_n|^2|\nabla \bar{u}_n|^2dx+2\int_{\R^N}|\bar{u}_n|^2\nabla \bar{u}_n\nabla u_0dx+\int_{\R^N}|\bar{u}_n|^2|\nabla u_0|^2dx\nonumber\\
		&\quad+2\int_{\R^N}\bar{u}_nu_0|\nabla \bar{u}_n|^2dx+4\int_{\R^N}\bar{u}_nu_0\nabla \bar{u}_n\nabla u_0dx+2\int_{\R^N}\bar{u}_nu_0|\nabla u_0|^2dx\nonumber\\
		&\quad+\int_{\R^N}|u_0|^2|\nabla \bar{u}_n|^2dx+2\int_{\R^N}\nabla \bar{u}_n\nabla u_0|u_0|^2dx+\int_{\R^N}|u_0|^2|\nabla u_0|^2dx.
	\end{align}
	Thus, we need only prove that the right hand side of \eqref{decomposition of quasi eq1} tends to $0$ as $n\to+\infty$, except for the first and last terms. 
	Following the idea in \cite[Lemma 3.1]{LW2014 poten}, we first show that 
	\begin{equation}\label{decomposition of quasi eq2}
		\mu_n|\nabla u_n|^\theta\to0~~\text{in}~~L^{1}_{loc}(\R^N),~~~~u_n\nabla u_n\to u_0\nabla u_0~~\text{in}~~L^{2}_{loc}(\R^N)^N,~~~~\nabla u_n\to\nabla u_0~~\text{in}~~L^{2}_{loc}(\R^N)^N.
	\end{equation}
	To this aim, taking $\phi=u_n\varphi$ with $\varphi\in C_0^{\infty}(\R^N)$ as the test function in \eqref{appro eq1 general}, we obtain
	\begin{align}\label{decomposition of quasi eq3}
		&\mu_n\int_{\R^N}|\nabla u_n|^{\theta}\varphi dx+\mu_n\int_{\R^N}|\nabla u_n|^{\theta-2}\nabla u_n\nabla \varphi u_n dx+\int_{\R^N}|\nabla u_n|^{2}\varphi dx+\int_{\R^N}\nabla u_n\nabla \varphi u_ndx\nonumber\\
		&+4\int_{\R^N}|u_n|^2|\nabla u_n|^2\varphi dx+2\int_{\R^N}u_n^3\nabla u_n\nabla \varphi dx+\lambda_n\int_{\R^N}u_n^2\varphi dx=\int_{\R^N}\mathcal{L}'(u_n)u_n\varphi dx.
	\end{align}
	By the boundedness of $\{\mu_n\|\nabla u_n\|_{\theta}^\theta\}$, we deduce that 
	\begin{align*}
		\mu_n\int_{\R^N}|\nabla u_n|^{\theta-2}\nabla u_n\nabla \varphi u_n dx
		\leq C\mu_n^{\frac{1}{\theta}}\left(\mu_n\int_{\R^N}|\nabla u_n|^{\theta}dx\right)^{\frac{\theta-1}{\theta}}\left(\int_{\R^N}|\nabla \varphi|^\theta dx\right)^{\frac{1}{\theta}}
		\to0~~~\text{as}~~~n\to+\infty.
	\end{align*}
We now claim that
	\begin{align*}
		&\int_{\R^N}\nabla u_n\nabla \varphi u_ndx+2\int_{\R^N}u_n^3\nabla u_n\nabla \varphi dx+\lambda_n\int_{\R^N}u_n^2\varphi dx-\int_{\R^N}\mathcal{L}'(u_n)u_n\varphi dx\\
		&\to\int_{\R^N}\nabla u_0\nabla \varphi u_0dx+2\int_{\R^N}u_0^3\nabla u_0\nabla \varphi dx+\lambda_0\int_{\R^N}u_0^2\varphi dx-\int_{\R^N}\mathcal{L}'(u_0)u_0\varphi dx.
	\end{align*}
	We only prove that
	\begin{equation*}\label{decomposition of quasi eq5}
		\lim_{n\to+\infty}\int_{\R^N}\nabla u_n\nabla \varphi u_ndx=\int_{\R^N}\nabla u_0\nabla \varphi u_0dx,~~~\lim_{n\to+\infty}\int_{\R^N}\mathcal{L}'(u_n)u_n\varphi dx=\int_{\R^N}\mathcal{L}'(u_0)u_0\varphi dx,
	\end{equation*}
	since the convergence of the other terms can be established using similar arguments.
	For any $\eps>0$, there exists a bounded set $\Omega$ such that $\supp\varphi\subset\Omega$ satisfying
	\begin{equation*}
		\int_{\R^N\backslash\Omega}\nabla u_n\nabla \varphi u_ndx<\eps~~~~\text{and}~~~~\int_{\R^N\backslash\Omega}\mathcal{L}'(u_n)u_n\varphi dx<\eps, ~~~\text{for any}~n\in\mathbb{N}.
	\end{equation*}
	And it follows from \eqref{general non assu} and $\|u_n\|_{L^{\infty}(\R^N)}\leq C$ that for any measurable subset $E\subset\R^N$ with $meas(E)$ sufficiently small,
	\begin{equation*}
		\int_{E}\nabla u_n\nabla \varphi u_ndx\leq\|u_n\|_\infty\|\nabla \varphi\|_\infty\|\nabla u_n\|_{L^2(E)}(meas(E))^{\frac{1}{2}}<\eps
	\end{equation*}
	and
	\begin{equation*}
		\int_{E}\mathcal{L}'(u_n)u_n\varphi dx\leq C\int_{E}(|u_n|^2+|u_n|^{2\cdot2^*})\varphi dx\leq Cmeas(E)<\eps, ~~~\text{for any}~n\in\mathbb{N},
	\end{equation*}
	where $meas(E)$ denotes the Lebesgue measure of the subset $E$.
	Therefore, the claim holds by using the Vitali's convergence theorem (see \cite[Theorem 3]{CLW2002}).
	Taking $\phi=u_0\varphi$ as the test function in \eqref{appro eq2 general}, we have
	\begin{align}\label{decomposition of quasi eq4}
		&\int_{\R^N}|\nabla u_0|^{2}\varphi dx+\int_{\R^N}\nabla u_0\nabla \varphi u_0dx+4\int_{\R^N}|u_0|^2|\nabla u_0|^2\varphi dx+2\int_{\R^N}u_0^3\nabla u_0\nabla \varphi dx+\lambda_0\int_{\R^N}u_0^2\varphi dx\nonumber\\
		&=\int_{\R^N}\mathcal{L}'(u_0)u_0\varphi dx.
	\end{align}
	Combining \eqref{decomposition of quasi eq3} and \eqref{decomposition of quasi eq4}, we obtain \eqref{decomposition of quasi eq2}. Since $u_0\in\mathcal{C}^{1,\beta_0}(\R^N)$, using the similar arguments as in \cite[Proposition A.1]{gg1}, we get that for any $B_R\subset\R^N$,
	\begin{equation}\label{decomposition of quasi eq6}
		\int_{B_R}|\bar{u}_n|^2|\nabla \bar{u}_n|^2dx\to0.
	\end{equation}
	Notice that $u_0\in H^1(\R^N)$, and by the regularity of $u_0$, we deduce that $u_0(x)\to0$ as $|x|\to+\infty$.
	Thus back to \eqref{decomposition of quasi eq1}, for suffciently large $R>0$, it follows from Lemma \ref{moser inter} and \eqref{decomposition of quasi eq6} that
	\begin{align*}
		\int_{\R^N}\bar{u}_nu_0|\nabla \bar{u}_n|^2dx&=\int_{B_R}\bar{u}_nu_0|\nabla \bar{u}_n|^2dx+\int_{\R^N\backslash B_R}\bar{u}_nu_0|\nabla \bar{u}_n|^2dx\to0~~\text{as}~~n\to+\infty.
	\end{align*}
The arguments for the other terms in \eqref{decomposition of quasi eq1}, except for the first and last ones, follow in an analogous manner.
\end{proof}
	
Inspired by \cite[Theorem 1.4]{Mederski2020},  we establish the following profile decompositions for the approximating solutions $\{u_n\}\subset\widetilde{X}$.
\begin{theorem}\label{profile decom}
	There are sequences $\{\tilde u_i\}_{i=0}^{\infty}\subset H^1(\R^N)$, $\{y_n^i\}_{i=0}^{\infty}\subset\R^N$ for any $n\ge1$, such that $y_n^0=0,|y_n^i-y_n^j|\to+\infty$ for $i\neq j$, and passing to a subsequence, the following conditions hold for any $i\ge0$:
	\begin{equation*}\label{profile decom eq1}
		u_n(\cdot+y_n^i)\rightharpoonup\tilde u_i~~\text{in}~~H^1(\R^N)~~\text{as}~~n\to+\infty,
	\end{equation*}
	\begin{equation*}\label{profile decom eq1-1}
		u_n(\cdot+y_n^i)\nabla u_n(\cdot+y_n^i)\rightharpoonup\tilde u_i\nabla \tilde u_i~~\text{in}~~L^2(\R^N)^N~~\text{as}~~n\to+\infty,
	\end{equation*}
	\begin{equation*}\label{profile decom eq2}
		\lim_{n\to+\infty}\int_{\R^N}|\nabla u_n|^2dx=\sum_{j=0}^i\int_{\R^N}|\nabla \tilde u_j|^2dx+\lim_{n\to+\infty}\int_{\R^N}|\nabla u_n^i|^2dx,
	\end{equation*}
	\begin{equation*}\label{profile decom eq3}
		\lim_{n\to+\infty}\int_{\R^N}|u_n|^2|\nabla u_n|^2dx=\sum_{j=0}^i\int_{\R^N}|\tilde u_j|^2|\nabla \tilde u_j|^2dx+\lim_{n\to+\infty}\int_{\R^N}|u_n^i|^2|\nabla u_n^i|^2dx,
	\end{equation*}
	where $u_n^i(\cdot):=u_n-\sum_{j=0}^i\tilde u_j(\cdot-y_n^j)$ and
	\begin{equation}\label{profile decom eq4}
		\limsup_{n\to+\infty}\int_{\R^N}\mathcal{L}(u_n)dx=\sum_{j=0}^{i}\int_{\R^N}\mathcal{L}(\tilde u_j)dx+\limsup_{n\to+\infty}\int_{\R^N}\mathcal{L}(u_n^i)dx.
	\end{equation}
	Moreover, if in addition $\mathcal{L}$ satisfies
	\begin{equation}\label{pro of L}
		\lim_{s\to0}\frac{\mathcal{L}(s)}{|s|^2}=\lim_{s\to+\infty}\frac{\mathcal{L}(s)}{|s|^{2\cdot2^*}}=0,
	\end{equation}
	then 
	\begin{equation}\label{pro of L eq}
		\lim_{i\to+\infty}\left(\limsup_{n\to+\infty}\int_{\R^N}\mathcal{L}(u_n^i)dx\right)=0.
	\end{equation}
\end{theorem}
\begin{remark} We note that Theorem \ref{profile decom} relies heavily on the Brezis-Lieb type property developed in Lemma \ref{decomposition of quasi}. This property requires that the weak limit $\tilde u_i$ for every $i\ge0$ actually weakly solves the limiting equation \eqref{appro eq2 general}.
	Compared with the decomposition results in \cite{gg1}, our approach only requires the regularity and uniformly boundedness of the solutions, without depending on the sign of $\lambda\in\R$. Furthermore, we perform a more refined decomposition of the approximating solution sequence $\{u_n\}$, and our method allows for iteration up to an infinite number of times. We believe that Theorem \ref{profile decom} possesses significant potential for applications in this field.
\end{remark}

Before proceeding to Theorem \ref{profile decom}, we present the following new variants of Lions' lemma.
\begin{lemma}\label{profile decom lem}
	Suppose that $\{w_n\}\subset H^1(\R^N)$ and $\{w_n\nabla w_n\}\subset L^2(\R^N)^N$ are bounded, and that for some $r>0$, we have
	\begin{equation*}
		\sup_{y\in\R^N}\int_{B_r(y)}|w_n|^2dx\to0~~~\text{as}~~~n\to+\infty.
	\end{equation*}
	Then, for any continuous function $\mathcal{L}:\R\to[0,\infty)$ satisfying \eqref{pro of L}, it holds
	\begin{equation*}
		\int_{\R^N}\mathcal{L}(w_n)dx\to0~~~\text{as}~~~n\to+\infty.
	\end{equation*}
	\end{lemma}
\begin{proof}
Let $\eps>0$ and $2<\beta<2\cdot2^*$ be given, and suppose that $\mathcal{L}$ satisfies \eqref{pro of L}. Then, there exist constants $0<\delta<M$ and $c_\eps>0$ such that
\begin{equation*}
	\mathcal{L}(s)\leq\eps|s|^2~~\text{if}~|s|\in[0,\delta],~~~\mathcal{L}(s)\leq\eps|s|^{2\cdot2^*}~~\text{if}~|s|\in(M,+\infty),~~~\mathcal{L}(s)\leq c_\eps|s|^\beta~~\text{if}~|s|\in(\delta,M].
\end{equation*}
Therefore, by using Lemma \ref{lions lemma}, we deduce that
\begin{equation*}
	\limsup_{n\to+\infty}\int_{\R^N}\mathcal{L}(w_n)dx\leq\eps\limsup_{n\to+\infty}\int_{\R^N}(|w_n|^2+|w_n|^{2\cdot2^*})dx.
\end{equation*}
Since the Sobolev inequality ensures that $\{w_n\}$ is bounded in $L^{2}(\R^N)$ and $L^{2\cdot2^*}(\R^N)$, we conclude by letting $\eps\to0$.
\end{proof}

\begin{proof}[{\bf{Proof of Theorem \ref{profile decom}} \rm}]
	First we claim that: Up to a subsequence, there exist sequences $k\in\mathbb{N}\cup\{\infty\}$, $\{\tilde u_i\}_{i=0}^k\subset\widetilde{X}$, for $0\leq i<k+1$ (if $k=\infty$, then $k+1=\infty$ as well), $\{u_n^i\}\subset\widetilde{X}, \{y_n^i\}\subset\R^N$ and positive numbers $\{c_i\}_{i=0}^k, \{r_i\}_{i=0}^k$ such that $y_n^0=0, r_0=0$ and for any $0\leq i<k+1$ one has
	\begin{enumerate}
		\item $u_n(\cdot+y_n^i)\rightharpoonup\tilde u_i$ in $H^1(\R^N)$,~~$u_n(\cdot+y_n^i)\nabla u_n(\cdot+y_n^i)\rightharpoonup\tilde u_i\nabla \tilde u_i$ in $L^2(\R^N)^N$ and $u_n(\cdot+y_n^i)\chi _{B_n}\to\tilde u_i$ in $L^2(\R^N)$ as $n\to+\infty$, $\tilde u_i\neq0$ if $i\ge1$;
		\item $|y_n^i-y_n^j|\ge n-r_i-r_j$ for $0\leq j\neq i<k+1$ and sufficiently large $n$;
		\item $u_n^{-1}=u_n$ and $u_n^i:=u_n^{i-1}-\tilde u_i(\cdot-y_n^i)$ for $n\ge1$;
		\item $\int_{B_{r_i}(y_n^i)}|u_n^{i-1}|^2dx\ge c_i\ge\frac{1}{2}\sup_{y\in\R^N}\int_{B_{r_i}(y)}|u_n^{i-1}|^2dx$ for sufficiently large $n$, $r_i\ge\max\{i,r_{i-1}\}$, if $i\ge1$, and
		\begin{equation*}
			c_i=\frac{3}{4}\lim_{r\to+\infty}\limsup_{n\to+\infty}\sup_{y\in\R^N}\int_{B_r(y)}|u_n^{i-1}|^2dx>0.
		\end{equation*}
		\item 
		\begin{equation*}
			\lim_{n\to+\infty}\|\nabla u_n\|_2^2=\sum_{j=0}^i\|\nabla \tilde u_j\|_2^2+\lim_{n\to+\infty}\|\nabla u_n^i\|_2^2,~~\lim_{n\to+\infty}\|u_n\nabla u_n\|_2^2=\sum_{j=0}^i\|\tilde u_j\nabla \tilde u_j\|_2^2+\lim_{n\to+\infty}\|u_n^i\nabla u_n^i\|_2^2.
		\end{equation*}
	\end{enumerate}
	By the boundedness of $\{u_n\}$ in $\widetilde{X}$, we obtain
	\begin{equation*}
		u_n\rightharpoonup u_0~~\text{in}~~H^1(\R^N),~~~u_n\nabla u_n\rightharpoonup u_0\nabla u_0~~\text{in}~~L^2(\R^N)^N,~~~u_n\chi_{B_n}\to u_0~~\text{in}~~L^2(\R^N),
	\end{equation*}
	where $\chi_{B_n}$ denotes the characteristic function on the ball $B_n$. Let $\tilde{u}_0:=u_0, ~u_n^0:=u_n-\tilde u_0$. If
	\begin{equation*}
		\lim_{n\to+\infty}\sup_{y\in\R^N}\int_{B_r(y)}|u_n^0|^2dx=0
	\end{equation*}
	for every $r\ge1$, then the proof of this claim is completed with $K=0$. Otherwise, following an argument analogous to \cite[Theorem 1.4]{Mederski2020}, we can deduce that the first four conclusions hold.
	Noting that $|y_n^1|\ge n-r_1$ and that $\{u_n(\cdot+y_n^1)\}$ is bounded in $\widetilde{X}$, up to a subsequence, we deduce that there exists $\tilde{u}_1\neq0$ such that
	\begin{equation*}
		u_n(\cdot+y_n^1)\rightharpoonup \tilde u_1~\text{in}~H^1(\R^N),~u_n(\cdot+y_n^1)\nabla u_n(\cdot+y_n^1)\rightharpoonup \tilde u_1\nabla \tilde u_1~\text{in}~L^2(\R^N)^N,~u_n(\cdot+y_n^1)\chi_{B_n}\to \tilde u_1~\text{in}~L^2(\R^N).
	\end{equation*}
	Since $u_n(\cdot+y_n^1)$ weakly solves \eqref{appro eq1 general}, by the standard argument as in Lemma \ref{moser inter}, we deduce that $\tilde{u}_1$ weakly solves \eqref{appro eq2 general}. Thus, by Remark \ref{regular rmk} and Lemma \ref{decomposition of quasi}, we obtain
	\begin{align*}
		&\int_{\R^N}|\nabla u_n^0(\cdot+y_n^1)|^2dx=\int_{\R^N}|\nabla \tilde u_1|^2dx+\int_{\R^N}|\nabla u_n^1(\cdot+y_n^1)|^2dx+o_n(1),\\
		&\int_{\R^N}|u_n^0(\cdot+y_n^1)|^2|\nabla u_n^0(\cdot+y_n^1)|^2dx=\int_{\R^N}|\tilde u_1|^2|\nabla \tilde u_1|^2dx+\int_{\R^N}|u_n^1(\cdot+y_n^1)|^2|\nabla u_n^1(\cdot+y_n^1)|^2dx+o_n(1),
	\end{align*}
	where $u_n^1:=u_n^0-\tilde u_1(\cdot-y_n^1)$. Applying Remark \ref{regular rmk} and Lemma \ref{decomposition of quasi} again, we have
	\begin{align*}
		&\int_{\R^N}|\nabla u_n|^2dx=\int_{\R^N}|\nabla \tilde u_0|^2dx+\int_{\R^N}|\nabla \tilde u_1|^2dx+\int_{\R^N}|\nabla u_n^1|^2dx+o_n(1),\\
		&\int_{\R^N}|u_n|^2|\nabla u_n|^2dx=\int_{\R^N}|\tilde u_0|^2|\nabla \tilde u_0|^2dx+\int_{\R^N}|\tilde u_1|^2|\nabla \tilde u_1|^2dx+\int_{\R^N}|u_n^1|^2|\nabla u_n^1|^2dx+o_n(1).
	\end{align*}
	Proceeding similarly as in \cite{Mederski2020} and iterating this procedure, we can prove this claim.
	Moreover, using Lemma \ref{profile decom lem}, we can see that \eqref{profile decom eq4} and \eqref{pro of L eq} also hold, thus completing the proof.
\end{proof}


\section*{Declaration of competing interest}
The authors declare that they have no known competing financial interests or personal relationships that could have appeared to influence the work reported in this paper.

\section*{Data availability}
No data was used for the research described in the article.

\section*{Acknowledgements}
This research is partially supported by NSFC(Grant No.12471102), NSF of Jilin Province (20250102004JC), and the
Research Project of the Education Department of Jilin Province (JJKH20250296KJ).

\end{document}